\documentclass[11pt,a4paper]{amsart}
\usepackage[left=3.5cm,right=3.5cm,top=3.5cm,bottom=3.5cm]{geometry}
\usepackage{enumerate}
\usepackage{leftidx}
\usepackage{mathtools}
\usepackage{amsmath,amscd,amssymb,amsthm}
\usepackage[arrow,matrix]{xy}
\usepackage[OT2,T1]{fontenc}
\usepackage{tikz}
\usepackage{graphicx}
\usepackage{hyperref}

\DeclareMathOperator{\gal}{Gal}

\DeclareMathOperator{\im}{Im}
\DeclareMathOperator{\aut}{Aut}
\DeclareMathOperator{\ab}{ab}
\DeclareMathOperator{\sh}{sh}

\theoremstyle{definition}
\newtheorem{definition}{Definition}[section]
\newtheorem{example}[definition]{Example}
\newtheorem{remark}[definition]{Remark}
\newtheorem*{remark*}{Remark}

\theoremstyle{plain}
\newtheorem{theorem}[definition]{Theorem}
\newtheorem{corollary}[definition]{Corollary}
\newtheorem{lemma}[definition]{Lemma}
\newtheorem{proposition}[definition]{Proposition}
\newtheorem{conjecture}[definition]{Conjecture}
\newtheorem*{theorem*}{Theorem}
\newtheorem*{theoremA}{Theorem A}
\newtheorem*{theoremB}{Theorem B}
\newtheorem*{theoremC}{Theorem C}

\newcommand{\F}{{ \mathbb F }}
\newcommand{\N}{{ \mathbb N }}
\newcommand{\Q}{{ \mathbb Q }}
\newcommand{\Z}{{ \mathbb Z }}
\newcommand{\R}{{ \mathbb R }}

\author[A. Ferraguti]{Andrea Ferraguti}
\address{Max Planck Institute for Mathematics\\
Vivatsgasse 7\\
53111 Bonn, Germany\\
}
\address{DICATAM, Università degli Studi di Brescia\\
via Branze 43\\
I-25123 Brescia, Italy\\
}
\email{and.ferraguti@gmail.com}

\author[C. Pagano]{Carlo Pagano}
\address{Max Planck Institute for Mathematics\\
Vivatsgasse 7\\
53111 Bonn, Germany\\
}
\email{carlein90@gmail.com}

\author[D. Casazza]{Daniele Casazza}
\address{ICMAT, Campus de Cantoblanco, 13-15 Calle de Nicol\'as Cabrera, 28049 Madrid, Spain}
\email{daniele.casazza@icmat.es}

\title[The inverse problem for arboreal Galois representations]{The inverse problem for arboreal Galois representations of index two}

\begin{document}

\begin{abstract}
This paper introduces a systematic approach towards the inverse problem for arboreal Galois representations of finite index attached to quadratic polynomials. Let $F$ be a field of characteristic $\neq 2$ and $f$ be a monic, quadratic polynomial in $F[x]$. Let $\rho_f$ be the arboreal Galois representation associated to $f$, taking values in the group $\Omega_{\infty}$ of automorphisms of the infinite, rooted, regular binary tree. We give a complete description of the maximal closed subgroups of each closed subgroup of index at most two of $\Omega_{\infty}$, and then we show how this description can be naturally given in terms of linear relations modulo squares between certain universal functions evaluated in elements of the post-critical orbit of $f$. We use such description in order to derive necessary and sufficient criteria for $\im(\rho_f)$ to be a given index two subgroup of $\Omega_\infty$. These criteria depend exclusively on the arithmetic of the critical orbit of $f$. Afterwards, we prove that if $\phi=x^2+t\in\Q(t)[x]$, then there exist exactly five distinct subgroups of index two of $\Omega_{\infty}$ that can appear as images of $\rho_{\phi_{t_0}}$ for infinitely many $t_0\in\Q$, where $\phi_{t_0}$ is the specialized polynomial. We show that two of these groups appear infinitely often, by providing explicit examples, and if Vojta's conjecture over $\Q$ holds true, then so do the remaining three. Finally, we give an explicit description of the derived series of each subgroup of index two. Using this, we introduce a sequence of combinatorial invariants for subgroups of index two. With a delicate use of these invariants we are able to establish that subgroups of index at most two of $\Omega_\infty$ are pairwise non-isomorphic as topological groups, a result of independent interest. This implies, in particular, that the five aforementioned groups are pairwise distinct topological groups, and therefore yield five genuinely different instances of the infinite inverse Galois problem over $\Q$. 
\end{abstract}

\maketitle

\section{Introduction}

Let $F$ be a field of characteristic $\neq 2$, fix a separable closure $F^{\text{sep}}$, and let $f\in F(x)$ be a quadratic rational function. Assume that $f^{(n)}$ has $2^n$ distinct zeroes for every $n$, where $f^{(n)}$ denotes the $n$-fold iteration, and let $\rho_f\colon \gal(F^{\text{sep}}/F)\to \Omega_{\infty}$ be the associated arboreal Galois representation, where $\Omega_{\infty}$ is the automorphism group of the infinite, rooted, regular binary tree (cf.\ Section \ref{translations} for the precise definition). It is a central problem in arithmetic dynamics to compute the image of such representations. The main conjecture in the field (see \cite[Conjecture 3.11]{jones2}) predicts that the image of $\rho_f$ has finite index in $\Omega_{\infty}$, provided that $f$ does not satisfy a set of special conditions. Following the philosophy that sees arboreal representations as analogues of Galois representations attached to Tate modules of abelian varieties, the above conjecture would be a dynamical avatar of Serre's open image theorem.

In general, it is an extremely hard task to compute $\im(\rho_f)$. The first results on this topic date back to Odoni \cite{odoni1},\cite{odoni2}, although they did not involve the modern language of arboreal Galois representations. When $f$ is a monic, quadratic polynomial with rational coefficients, Stoll \cite{Stoll} proved that $\rho_f$ is surjective if and only if the adjusted post-critical orbit of $f$ is a linearly independent set in $\Q^{\times}/(\Q^{\times})^2$, and then constructed infinite families of examples of polynomials of the form $x^2+a$ whose post-critical orbit satisfies such condition.

Nowadays, several other results are available for quadratic polynomials and rational functions over various ground fields, such as \cite{boston,ferra3,ferra2,hindes3,jones4,jones5}. Stoll's work has been generalized in various directions (see for example \cite{hindes2},\cite{jones4}), but always with the goal of proving surjectivity for the arboreal representation attached to a polynomial over certain specific base fields. On the other hand, it is a very natural question to ask which conjugacy classes of closed subgroups of $\Omega_{\infty}$ can arise as images of arboreal representations attached to quadratic polynomials. \footnote{Here $\Omega_{\infty}$ denotes the automorphism group of a fixed model for the tree. As explained in Section \ref{basics}, for each quadratic polynomial one has a well-defined conjugacy class of images of closed subgroups of $\Omega_{\infty}$.} We shall refer to this as the \emph{inverse problem for arboreal Galois representations}. 

In this paper, we address this problem for closed subgroups of $\Omega_{\infty}$ of index two (that are precisely maximal subgroups, since $\Omega_\infty$ is a pro-2-group). Our approach is based on a systematic classification of the maximal closed subgroups of each index two closed subgroup of $\Omega_{\infty}$ that translates very naturally into a classification based on linear dependence relations among (universal functions evaluated in) elements of the post-critical orbit of $f$. The strength of our approach is that it is completely algebraic, so that it does not depend in any way on the arithmetic of the ground field $F$. This is reflected in one of our side results, namely Corollary \ref{PCF_polynomials}. Here we show that a monic quadratic polynomial over a field of characteristic $\neq 2$ that is post-critically finite or having coefficient in a field $F$ such that $F^{\times}/(F^{\times})^2$ is a finite group, has an arboreal representation of infinite index. This was already known for post-critically finite rational functions over global fields (see \cite[Theorem 3.1]{jones2}) (with some condition relating the degree of the function and the characteristic of the field), but the proof critically relies on a deep result of Ihara that makes use of the fact that the ground field is a global one.

As we show in Section \ref{subgroups of index 2, abstract}, maximal subgroups of $\Omega_\infty$ are parametrized by the set of non-zero vectors $\underline{a}$ in $\F_2^{(\Z_{\geq 0})}$, the $\F_2$-vector space of maps $\Z_{\geq 0}\to \F_2$ that attain non-zero values at finitely many points. We denote by $M_{\underline{a}}$ the maximal subgroup of $\Omega_\infty$ corresponding to the non-zero vector $\underline{a}$. Our first main result, Theorem A below, establishes necessary and sufficient conditions for $\im(\rho_f)$ to coincide with a given $M_{\underline{a}}$. Recall that the adjusted post-critical orbit of a polynomial $f=(x-\gamma)^2-\delta$ is the sequence defined by $c_0\coloneqq \delta$, $c_n\coloneqq f^{(n+1)}(\gamma)$, for $n\geq 1$.\footnote{The indexing is shifted by 1 with respect to the usual one in order to keep consistence with the notation that will be introduced in Section \ref{subgroups of index 2, abstract}.}

\begin{remark*}
In order to emphasize the importance of the critical point of a quadratic polynomial for the results of this paper we will, in accordance with most of the existing literature on the topic, write monic quadratic polynomials in the form $(x-\gamma)^2-\delta$ (notice that any monic, quadratic polynomial over a field of characteristic different from 2 can be written in this form). Moreover, all results of this paper that involve the adjusted post-critical orbit are valid for arboreal representation with an arbitrary basepoint $z\in F$. This means that instead of considering the tree of roots of the $f^{(n)}$'s, one considers the tree of roots of $f^{(n)}-z$, where $z\in F$ is fixed. In order to ease the exposition, we have not stated our results in this generality; instead we took $z=0$ everywhere. Nevertheless, all arguments require no non-trivial modification, and hold for any $z$ up to replacing the adjusted post-critical orbit $\{c_n\}_{n\geq 0}$ of $f$ with the sequence $\{c_0+z,c_1-z,\ldots,c_n-z,\ldots\}$.
\end{remark*}

Recall that if $f,g\in F[x]$ then $g$ is called \emph{$f$-stable} if $g\circ f^{(n)}$ is irreducible for every $n\geq 1$. If $f$ is $f$-stable then we just call it \emph{stable}.

\begin{theoremA} \label{theoremA}
 Let $F$ be a field of characteristic $\neq 2$, $f=(x-\gamma)^2-\delta\in F[x]$ and $\{c_n\}_{n\geq 0}$ be its adjusted post-critical orbit. Let $\rho_f$ be the associated arboreal representation. Finally, let $\underline{a}=(a_i)_{i\geq 0}\in\F_2^{(\Z_{\geq 0})}$ be a non-zero vector.
 \begin{enumerate}
  \item If $\underline{a}\neq (1,0,\ldots,0,\ldots)$, then $\im(\rho_f)=M_{\underline{a}}$ if and only if one of the following two conditions is satisfied:
  \begin{enumerate}[a)]
  \item $a_0=1$, $\prod_{i\geq 0}c_i^{a_i}\in {F^{\times}}^2$ and for every non-zero vector $\underline{a}'=(a_i')_{i\geq 0}\in\F_2^{(\Z_{\geq 0})}$ different from $\underline{a}$, we have that $\prod_{i\geq 0}c_i^{a_i'}\notin {F^{\times}}^2$.
  \item $a_0=0$, $\prod_{i\geq 0}c_i^{a_i}\in {F^{\times}}^2$, for every non-zero vector $\underline{a}'=(a_i')_{i\geq 0}\in\F_2^{(\Z_{\geq 0})}$ different from $\underline{a}$ we have that $\prod_{i\geq 0}c_i^{a_i'}\notin {F^{\times}}^2$ and the element $\widetilde{c}_{\underline{a}}\in F^{\times}$ defined in Proposition \ref{The abelian "additional" extension} does not belong to the span of $\{c_n\}_{n\geq 0}$ in the $\F_2$-vector space $F^{\times}/{F^{\times}}^2$.
  \end{enumerate}
  If one of the two above conditions hold, then in particular $f$ is stable.
  \item If $\underline{a}=(1,0,\ldots,0,\ldots)$, then the following are equivalent:
\begin{enumerate}[i)]
\item $\im(\rho_f)=M_{\underline{a}}$;
\item the set
  $$\{ c_1+\gamma\pm u,c_2-\gamma\pm u,\ldots,c_n-\gamma\pm u,\ldots\}$$
  is linearly independent in $F^{\times}/(F^{\times})^2$, where $u=\sqrt{\delta}\in F$.
\end{enumerate}  
    If one of the two above equivalent condition holds, then $f$ factors into a product of two $f$-stable linear factors.
 \end{enumerate}

\end{theoremA}
The element $\widetilde{c}_{\underline{a}}$ mentioned in case 1b) is the evaluation of a universal polynomial $g_{\underline{a}}\in\Z[X_0,\ldots,X_{n-1},Y]$ in the point $\left(c_0,\ldots,c_{n-1},\sqrt{\prod_{i\geq 0}c_i^{a_i}}\right)$, where $n$ is the largest index such that $a_n\neq 0$. Although the value $\widetilde{c}_{\underline{a}}=g_{\underline{a}}\left(c_0,\ldots,c_{n-1},\sqrt{\prod_{i\geq 0}c_i^{a_i}}\right)$ depends on a choice of the square root, the linear independence of $\widetilde{c}_{\underline{a}}$ from $\{c_n\}_{n\geq 0}$ does not (cf.\ Proposition \ref{The abelian "additional" extension}).

The key ingredient to prove the above theorem is a detailed analysis of the abelianization of subgroups of index two of $\Omega_\infty$. Products of elements of the adjusted post-critical orbit yield quadratic extensions of $F$ contained in the direct limit of the splitting fields of the $f^{(n)}$'s. When $\rho_f$ is surjective, these are all the quadratic subextensions. We show that when $f$ is irreducible and $\im(\rho_f)$ has index two in $\Omega_\infty$, one can describe explicitly all such quadratic extensions, with a dichotomy of cases that is reflected by conditions 1a) and 1b) of Theorem A. It is worth noticing that the post-critical orbit contains yet again all information to decide whether $\rho_f$ has image of index two, because also the element $\widetilde{c}_{\underline{a}}$ mentioned in point 1b) is constructed using only elements of the post-critical orbit.

\vspace{3mm}

The second part of the paper focuses on polynomials with rational coefficients. For a polynomial $f\in \Q(t)[x]$ and a rational number $t_0$, we denote by $f_{t_0}$ the specialized polynomial in $\Q[x]$. Let $\phi=x^2+t\in \Q(t)[x]$. One can ask which values $t_0\in \Q$ yield a specialized polynomial $\phi_{t_0}\in \Q[x]$ whose arboreal representation has index two image. In Proposition \ref{subgroups}, we prove that there exist exactly five maximal subgroups $\mathcal G_1,\ldots,\mathcal G_5$ of $\Omega_{\infty}$ that can appear as $\im(\rho_{\phi_{t_0}})$ for infinitely many $t_0$, and we describe the vectors of $\F_2^{(\Z_{\geq 0})}$ to which these subgroups correspond. It will follow from our results in Section \ref{proofC} that the $\mathcal G_i$'s are pairwise non-isomorphic topological groups.

Let $\psi\coloneqq x^2-1-t^2\in \Q(t)[x]$ and $\displaystyle \vartheta\coloneqq x^2+\frac{1}{t^2-1}\in \Q(t)[x]$. These have the property that $\im(\rho_\psi)\subseteq \mathcal G_1$ and $\im(\rho_\vartheta)\subseteq \mathcal G_2$, and hence $\im(\rho_{\psi_{t_0}})\subseteq \mathcal G_1$ and $\im(\rho_{\vartheta_{t_0}})\subseteq \mathcal G_2$ for every $t_0\in \Q$. The second main result of the paper is the following.
\begin{theoremB}
Let $\phi,\psi,\vartheta$ be as above.
\begin{enumerate}
 \item Let $t_0\in 2\Z\setminus\{0\}$. Then:
 $$\im(\rho_{\psi_{t_0}})=\im\left(\rho_{\phi_{-1-t_0^2}}\right)=\mathcal G_1,$$
 and consequently $\im(\rho_\psi)=\mathcal G_1$.
 \item Let $t_0\in 2\Z\setminus\{0\}$. Then:
 $$\im(\rho_{\vartheta_{t_0}})=\im\left(\rho_{\phi_{\frac{1}{t_0^2-1}}}\right)=\mathcal G_2,$$
 and consequently $\im(\rho_\vartheta)=\mathcal G_2$.
 \item Assume Vojta's conjecture over $\Q$ (see Subsection \ref{sub:vojta} for the precise statement), and let $i\in \{3,4,5\}$. Then there exists an infinite, thin set $E_i\subseteq \Q$ such that if $t_0\in E_i$, then $\im(\rho_{\phi_{t_0}})=\mathcal G_i$.
\end{enumerate}
\end{theoremB}

For a definition of thin set, see Subsection \ref{sub:vojta}. A conjecture of Hindes \cite[Conjecture 1.5]{hindes2} implies, together with Theorem A, that the only subgroups of index two of $\Omega_{\infty}$ that can appear as images of $\rho_{\phi_{t_0}}$ for some $t_0\in\Z$ are $\mathcal G_1$ and $\mathcal G_5$. In particular, we show that there exist infinitely many integers $t_0$ such that $\im(\rho_{\phi_{t_0}})=\mathcal G_1$, and that the same holds for $\mathcal G_5$ under Vojta's conjecture.

Even though all our sets of specializations are thin, there is a difference between $\mathcal G_1,\mathcal G_2$ and $\mathcal G_5$ on one side and $\mathcal G_3,\mathcal G_4$ on the other one. In fact, for the first three we can find a polynomial of the form $x^2+h(t)\in \Q(t)[x]$, where $h(t)$ is non-constant, such that its associated arboreal representation (over $\Q(t)$) is $\mathcal G_i$. It might therefore be reasonable to expect that almost all specializations (in the appropriate sense) yield the same image. On the other hand, polynomials of the form $x^2+t_0$, for some $t_0\in\Q$, having as image $\mathcal G_3$ or $\mathcal G_4$ do not come from specializations of a polynomial in the aforementioned form.

In order to prove (1), we use Theorem A, combined with a suitable generalization of the arguments described in \cite[$\mathsection$ 2]{Stoll} to polynomials with rational coefficients. To prove (2), we generalize an argument of Hindes \cite{hindes1} showing that Vojta's conjecture implies the claim if one can show that all algebraic curves in a certain finite set have finitely many rational points. We will be able to deduce this from Faltings' theorem via some algebraic manipulations.

\vspace{3mm}

Finally, we consider the question of whether the various closed subgroups of index at most two of $\Omega_{\infty}$ are non-isomorphic in the category of profinite groups. To this end we obtain a full description of the derived series of each closed subgroup of index at most two of $\Omega_{\infty}$, upgrading the description of their abelianization used to prove Theorem A. 
In parallel we introduce and make systematic use of the concept of the \emph{graph of commutativity} of a topological group equipped with a set of topological generators. A delicate combination of these two inputs enable us to establish the following. 

\begin{theoremC}
Let $H_1,H_2$ be two closed subgroups of $\Omega_{\infty}$, both having index at most $2$. Then $H_1 \cong_{\emph{top.gr.}} H_2$ if and only if $H_1=H_2$. 
\end{theoremC} 

In particular Theorem C implies that the groups $\mathcal G_1,\ldots,\mathcal G_5$ yield five distinct instances of the infinite Galois inverse problem. We remark that, for multiple reasons, the realization of a closed subgroup of $\Omega_{\infty}$ cannot be inferred from the infinite version of Shafarevich theorem given in \cite[Corollary 9.5.10]{neukirch}. 

Here is a brief outline of the content of the paper. In Sections \ref{basics} and \ref{subgroups of index 2, abstract} we recall the fundamental definitions and obtain a group theoretic criterion for having index two. In Section \ref{translations} we reprove Stoll's criterion for surjectivity for the arboreal representation attached to a quadratic polynomial using our language and we translate our group theoretic criterion in terms of linear dependence relations among elements of the post-critical orbit, establishing Theorem A. Next, in Section \ref{realizations} we prove Theorem B. Finally in Section \ref{proofC} we give an explicit description of the derived series of closed subgroups of index at most two and we use it to prove Theorem C.

\subsection*{Notation and conventions}
For a topological group $G$, we denote, as usual, by $[G,G]$ the subgroup topologically generated by commutators and by $G^{\ab}$ the quotient of $G$ by $[G,G]$.

We denote by $G^\vee$ the dual group of $G$, namely the group $\hom_{\text{top.gr.}}(G,S^1)$, where the topology on $S^1$ is the Euclidean one. Notice that for a pro-$2$-group $G$, the group $G^\vee$ coincides with $\hom_{\text{top.gr.}}(G^{\ab},\Q_2/\Z_2)$. Here the topology on $\Q_2/\Z_2$ can be either the $2$-adic one or the one induced by its inclusion in $S^1$; the corresponding $\hom_{\text{top.gr.}}$-sets are in fact equal.

The set $\F_2^{\Z_{\geq 0}}$ is the $\F_2$-vector space of functions $\Z_{\geq 0}\to \F_2$, while the set $\F_2^{(\Z_{\geq 0})}$ is the $\F_2$-vector space of functions $\Z_{\geq 0}\to \F_2$ that take non-zero values at finitely many points. For every $N\geq 0$, we denote by $e_N$ the $N$-th vector of the canonical basis of $\F_2^{(\Z_{\geq 0})}$, namely the function $e_N\colon\Z_{\geq 0}\to\F_2$ such that $e_N(m)=1$ if $m=N$ and $e_N(m)=0$ otherwise.

For a field $F$ of characteristic not 2 and a set of non-zero elements $\{c_1,\ldots,c_n\}\subseteq F$, we denote by $\langle c_1,\ldots,c_n\rangle_F$ the $\F_2$-span of $\{c_1{F^{\times}}^2,\ldots,c_n{F^{\times}}^2\}$ inside the $\F_2$-vector space $F^{\times}/{F^{\times}}^2$.

\section*{Acknowledgements}
We are extremely grateful to the anonymous referees of this work for valuable feedback. We are especially grateful to the referee whose careful read of the proof of Theorem \ref{thmC} and insightful suggestions brought us to a shorter and more elegant version of the argument. 

The first two authors are very grateful to the Max Planck Institute for Mathematics in Bonn for its financial support and for the inspiring atmosphere and great work conditions, which helped them to obtain the most significant ideas of this project. The first two authors wish to thank the third author for helping them with some of the calculations of Section \ref{realizations} and for, thereby, accepting to coauthor the paper. The third author acknowledges the Spanish Ministry of Economy and Competitiveness for its financial support through the Severo-Ochoa programme for Centres of Excellence in R\&D (SEV-2015-0554). Finally, we would like to thank Rafe Jones for giving useful feedbacks on an early version of this paper and the anonymous referee for providing many helpful comments.

\section{The tree of roots as a Cayley graph}\label{basics}
In this section, we will explain how to think of the infinite, rooted, regular binary tree as a Cayley graph, and show how this relates to the tree of roots of a quadratic rational function. The language of Cayley graphs will be a more natural one for our arguments.

Let $S\coloneqq\{x,y\}$ and let $F_2$ be the free monoid generated by $S$ . Put
$$T_{\infty}\coloneqq\text{Cayley}(F_2,S),
$$
the Cayley graph on $F_2$ with respect to the generating set $S$: each word $w$ in $F_2$ is connected to $xw$ and $yw$. This is an infinite, regular, binary tree rooted at $e$, the identity element of $F_2$. We put
$$\Omega_{\infty}\coloneqq\text{Aut}_{\text{graph}}(T_{\infty}).
$$
The monoid $F_2$ is naturally a \emph{graded} monoid, where the grading in consideration is simply given by the length of a word: this defines naturally a monoid homomorphism from $F_2$ to $\mathbb{Z}_{\geq 0}$. For each $n \in \mathbb{Z}_{\geq 0}$ we put $L_n$ to be the set of elements of $T_{\infty}$ having length $n$. It is clear that $|L_n|=2^{n}$. We have that $\Omega_{\infty}$ preserves each set $L_n$, and so it does preserve
$$T_N\coloneqq\bigcup_{i=0}^{N}L_i,
$$
For every $N\geq 1$ we will let $\Omega_N\coloneqq \text{Aut}_{\text{graph}}(T_N)$. The restriction homomorphism
$$\Omega_{\infty} \to \Omega_N
$$
is surjective. More details on the action of $\Omega_N$ on $T_N$ can be found for example in \cite{ferra,jones2,silverman2}.

The system of $\{\Omega_N\}_{N \in \mathbb{Z}_{\geq 0}}$, together with the restriction homomorphisms, forms an inverse system, and it is a fact that the natural homomorphism
$$\Omega_{\infty} \to \varprojlim_N\Omega_N,
$$
gotten from the above system of restriction maps, is an isomorphism. This naturally endows $\Omega_{\infty}$ with the structure of a profinite group. Furthermore one can show that for each non-negative integer $N$ we have that $\Omega_N$ is a $2$-Sylow of $\text{Sym}(L_N)$. This gives that $\Omega_{\infty}$ is a pro-$2$-group. This description is of course just a rephrasing of the usual wreath product formulation used, for example, in \cite{boston,ferra,jones2,jones6}.

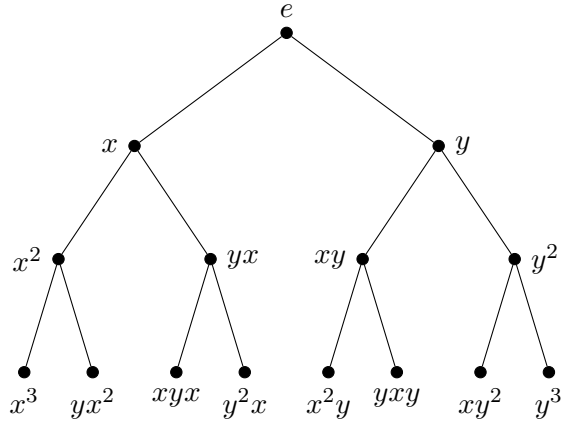
\begin{figure}[h]
\centering

\begin{tikzpicture}
\tikzstyle{solid node}=[circle,draw,inner sep=1.5,fill=black]
\node[solid node,label=above:$e$]{}[sibling distance=40mm]
child{node[solid node,label=left:$x$]{}[sibling distance=20mm]
child{node[solid node,label=left:$x^2$]{}[sibling distance=9mm]
child{node[solid node,label=below:$x^3$]{}}
child{node[solid node,label=below:$yx^2$]{}}}
child{node[solid node,label=right:$yx$]{}[sibling distance=9mm]
child{node[solid node,label=below:$xyx$]{}}
child{node[solid node,label=below:$y^2x$]{}}}
}
child{node[solid node,label=right:$y$]{}[sibling distance=20mm]
child{node[solid node,label=left:$xy$]{}[sibling distance=9mm]
child{node[solid node,label=below:$x^2y$]{}}
child{node[solid node,label=below:$yxy$]{}}}
child{node[solid node,label=right:$y^2$]{}[sibling distance=9mm]
child{node[solid node,label=below:$xy^2$]{}}
child{node[solid node,label=below:$y^3$]{}}}
};
\end{tikzpicture}\caption{The first three levels of $T_\infty$.}\label{fig:tree}
\end{figure}

\vspace{4mm}

Now let $F$ be a field of characteristic different from $2$. Let $f(x)\coloneqq(x-\gamma)^2-\delta$ in $F[x]$ be a quadratic polynomial. Fix once and for all $F^{\text{sep}}$ a separable closure of $F$ and put $G_F\coloneqq\gal(F^{\text{sep}}/F)$ the group of $F$-algebra automorphisms of $F^{\text{sep}}$, viewed as a topological group with its natural profinite topology. For $n \in \mathbb{Z}_{\geq 0}$, the $n$-th iterate of $f$ will be denoted by $f^{(n)}$, where we set $f^{(0)}\coloneqq x$ by convention. Assume that $f^{(n)}$ is a separable polynomial for every $n \in \mathbb{Z}_{\geq 0}$. The infinite, regular, rooted binary tree, $T_{\infty}(f)$ is constructed in the following way: for every non-negative integer $n$, the nodes of the tree at level $n$ are labeled by the roots of $f^{(n)}$ in $F^{\text{sep}}$. If $n \in \mathbb{Z}_{\geq 1}$ then a node $\alpha$ at level $n$ is connected to a node $\beta$ at level $n-1$ if and only if $f(\alpha)=\beta$.

In this manner $T_{\infty}(f)$ is isomorphic, as a graph, to $T_{\infty}$. Observe that $T_{\infty}(f)$ is a $G_F$-set and furthermore, since $f \in F[x]$, we have that $G_F$ preserves the tree structure on $T_{\infty}(f)$. As such we have a continuous homomorphism 
$$\rho_f:G_F \to \Omega_{\infty}(f),
$$
where $\Omega_{\infty}(f)\coloneqq\text{Aut}_{\text{graph}}(T_{\infty}(f))$, which is, as a topological group, isomorphic to $\Omega_{\infty}$. The map $\rho_f$ is the \emph{arboreal Galois representation} associated to $f$.

We denote by $\text{Isom}_{\text{graph}}(T_{\infty}(f),T_{\infty})$ the set of graph isomorphisms. Observe that $\Omega_{\infty}$ acts freely and transitively on $\text{Isom}_{\text{graph}}(T_{\infty}(f),T_{\infty})$. Every element $\iota \in \text{Isom}_{\text{graph}}(T_{\infty}(f),T_{\infty})$ induces an isomorphism of profinite groups
$$\iota_{*}:\Omega_{\infty}(f) \to \Omega_{\infty}.
$$
Observe that after applying $\text{Isom}_{\text{graph}}(T_{\infty}(f),T_{\infty})$ we obtain  that $\rho(f)$ gives a unique $\Omega_{\infty}$-conjugacy class of continuous homomorphisms, $[\rho_f]$, from $G_F$ to $\Omega_{\infty}$. \footnote{This situation is analogous to the one of $l$-adic Galois representations attached to the Tate module of an elliptic curve $E$ where the representation $\rho_{E,l}$ is canonically attached to $E$, but if one identifies $T_l(E)$ with $\mathbb{Z}_l^2$, then one obtains a conjugacy class of representations in $\text{GL}_2(\mathbb{Z}_l)$.}

\begin{remark}
 In particular, in the case in which $\text{Im}(\rho_f)$ has index two in $\Omega_{\infty}(f)$, all elements of $[\rho_f]$ have the same image in $\Omega_{\infty}$, since subgroups of index two are always normal and therefore in this case $\text{Im}(\iota_{*} \circ \rho_f)$ does not depend on the choice of $\iota \in \text{Isom}_{\text{graph}}(T_{\infty}(f),T_{\infty})$. For this reason it makes sense to say that $\text{Im}(\rho(f))$ is contained in a given maximal closed subgroup of $\Omega_\infty$. 
\end{remark}

\section{Maximal subgroups of \texorpdfstring{$\Omega_{\infty}$}{} and their abelianizations} \label{subgroups of index 2, abstract}

The purpose of this section is to describe the abelianization of the maximal closed subgroups of $\Omega_{\infty}$. We will use the notation of Section \ref{basics}, so recall that $T_\infty$ is the Cayley graph of the free monoid $F_2$ generated by $S=\{x,y\}$ and $\Omega_{\infty}$ is the group of graph automorphisms of $T_{\infty}$.

We have an evident involution $\sigma_e \in \Omega_{\infty}$ induced by the unique non-trivial monoid automorphism of $F_2$, which consists of exchanging $x$ and $y$. Furthermore for each word $w \in F_2$ we can define an involution
$$\sigma_w \in \Omega_{\infty},
$$
defined by saying that $\sigma_w(v)=v$ if $v \neq v'w$ for each $v' \in F_2$ and $\sigma_w(v)=\sigma_e(v')w$ if $v=v'w$ for some (and hence unique) $v' \in F_2$. Using these involutions one sees easily that for every non-negative integer $N$ we have that $\Omega_{N+1}$ equals the semidirect product 
$$\mathbb{F}_2^{L_{N}} \rtimes \Omega_{N},
$$ 
where $\mathbb{F}_2^{L_{N}}$ can be naturally identified with the subspace spanned by the set of involutions $\{\sigma_w\}_{w \in L_{N}}$, and the implicit action of $\Omega_{N}$ on $\mathbb{F}_2^{L_{N}}$ is simply the $\mathbb{F}_2$-linearization of the action of $\Omega_N$ on $L_{N}$. In particular we shall view $\mathbb{F}_2^{L_{N}}$ as a $\mathbb{F}_2[\Omega_N]$-module. 

Therefore, by means of the involutions $\sigma_w$ we have an identification:
\begin{equation}\label{semidirect}
\Omega_{N+1}=\mathbb{F}_2^{L_{N}} \rtimes (\mathbb{F}_2^{L_{N-1}} \rtimes( \ldots \mathbb{F}_2^{L_0})).
\end{equation}
This gives for each element $\sigma \in \Omega_{N+1}$ a \emph{unique} decomposition as
$$\sigma_{N} \cdot \sigma_{N-1} \ldots \cdot \sigma_0
$$ 
where each $\sigma_i$ belongs to the subspace $\mathbb{F}_2^{L_i}$ generated by the $\sigma_{w}$ with $w \in L_i$.
\begin{definition}\label{digital_representation}
We will refer to this representation as the \emph{digital representation} of $\Omega_{\infty}$.
\end{definition}

 Pushing this representation to the limit $N \to \infty$, provides us with an isomorphism, as \emph{profinite spaces}, between $\Omega_{\infty}$ and $\mathbb{F}_2^{T_{\infty}}$: the formulas for the group law so obtained on $\mathbb{F}_2^{T_{\infty}}$, by transport of structure, are quite involved, reflecting the fact that $\Omega_{\infty}$ is a substantially more complicated group than $\mathbb{F}_2^{T_{\infty}}$.

For each non-negative integer $i$ denote by 
$$\phi_i: \mathbb{F}_2^{L_i} \to \mathbb{F}_2
$$
the unique non-trivial homomorphism as $\mathbb{F}_2[\Omega_{i}]$-modules, which is obtained by summing all coordinates. From the iterated semi-direct description of $\Omega_{\infty}$, given in \eqref{semidirect}, we see that each of the $\phi_i$ extends to a continuous homomorphism from $\Omega_{\infty}$ to $\mathbb{F}_2$. In formulas, using the digital representation, the extended map, which we shall denote still as $\phi_i$, is simply gotten by defining
$$\phi_i(\sigma)\coloneqq\phi_i(\sigma_i).
$$
It is precisely due to the fact that $\phi_i$ preserves the structure as an $\mathbb{F}_2[\Omega_{i}]$-module that one can show that it gives a group homomorphism from $\Omega_{\infty}$ to $\mathbb{F}_2$. Furthermore one has that
$$[\Omega_{\infty},\Omega_{\infty}]=\Phi(\Omega_{\infty})=\bigcap_{i \in \mathbb{Z}_{\geq 0}}\ker(\phi_i)
$$
(see also Proposition \ref{reducing descending central to augmentational} for a more general statement). Here the first group denotes the commutator subgroup and the second the Frattini subgroup (i.e. the intersection of all maximal proper closed subgroups). As observed for example in \cite[p.\ 241]{Stoll}, it follows that $\Omega_\infty^{\ab}\cong \F_2^{\Z_{\geq 0}}$.

Hence we can parametrize closed subgroups of index two of $\Omega_{\infty}$ simply using non-zero vectors in the direct sum $\mathbb{F}_2^{(\mathbb{Z}_{\geq 0})}$, via the assignment sending $\underline{a}\coloneqq(a_n)_{n \in \mathbb{Z}_{\geq 0}} \in \mathbb{F}_2^{(\mathbb{Z}_{\geq 0})}$ into
$$M_{\underline{a}}\coloneqq\ker\left(\sum_{i \in \mathbb{Z}_{\geq 0}}a_i\phi_i\right).
$$

We will make crucial use of some additional maps from $\Omega_{\infty}$ into $\mathbb{F}_2$, which are ``close'' to be homomorphisms. They will be used to describe the map
$$M_{\underline{a}}^{\text{ab}} \to \Omega_{\infty}^{\text{ab}},
$$
for each $\underline{a}  \in \mathbb{F}_2^{(\mathbb{Z}_{\geq 0})}$. For every positive integer $i \in \mathbb{Z}_{\geq 1}$ and $s \in S=\{x,y\}$, we put
$$\widetilde{\phi}_i(s):\mathbb{F}_2^{L_i} \to \mathbb{F}_2,
$$
consisting of summing only the coordinates of words $w$ of the form $w\coloneqq v's$ (observe this happens half of the time because $i>0$). Observe that these two maps are \emph{not} $\mathbb{F}_2[\Omega_i]$-linear but they are $\mathbb{F}_2[\ker(\phi_0) \cap \Omega_i]$-linear. We extend $\widetilde{\phi}_i(s)$ to a set theoretic map from $\Omega_{\infty}$ to $\mathbb{F}_2$ using the digital representation, in the same way we did for $\phi_i$. The maps $\widetilde{\phi}_i(s)$ fail to be additive in the following manner. Take $\sigma_1,\sigma_2 \in \Omega_{\infty}$, then we have that:
\begin{equation}\label{uncertain_additivity}
\widetilde{\phi}_i(x)(\sigma_1\sigma_2)=\begin{cases}\widetilde{\phi}_i(x)(\sigma_1)+\widetilde{\phi}_i(x)(\sigma_2) & \mbox{if }\phi_0(\sigma_1)=0,\\
\widetilde{\phi}_i(x)(\sigma_1)+\widetilde{\phi}_i(y)(\sigma_2) & \mbox{if } \phi_0(\sigma_1)=1.
\end{cases}
\end{equation}
The $\mathbb{F}_2[\ker(\phi_0) \cap \Omega_i]$-linearity of $\widetilde{\phi}_i(s)$ is clear from the formula above. We shall refer to this property as to the \emph{uncertain additivity} of $\widetilde{\phi}_i(s)$. 
 
Before stating and proving the next theorem, which describes the abelianization of $M_{\underline{a}}$, let us recall that the set $\mathbb{F}_2 \times \mathbb{F}_2$ can be equipped with the group law
\begin{equation}\label{cohomological_group_structure}
 (x_1,y_1)\star(x_2,y_2)\coloneqq(x_1+x_2+y_1 \cdot y_2,y_1+y_2),
\end{equation}

where the $\cdot$ is the ordinary product in $\mathbb{F}_2$, with respect to the field structure. This is a way to represent the cyclic group on $4$ elements: it comes from writing explicitly the unique non-trivial element of $H^2(\mathbb{F}_2,\mathbb{F}_2)$ given by $\chi \cup \chi$, where $\chi$ is the unique non-trivial character of $\mathbb{F}_2$.
 \begin{theorem} \label{Describing abelianized Ma}
Assume $\underline{a} \in \mathbb{F}_2^{(\mathbb{Z}_{\geq 0})}$ is non-zero and different from $(1,0, \ldots, 0, \ldots)$. Let 
$$\varphi\colon M_{\underline{a}}^{\emph{ab}} \to \Omega_{\infty}^{\text{ab}}
$$
be the natural map induced by the inclusion. Then $|\ker\varphi|=2$. Moreover, the following hold.
\begin{enumerate}
\item If $a_0=1$, then $M_{\underline{a}}^{\emph{ab}}$ has a cyclic direct summand of order $4$. A surjective homomorphism from $M_{\underline{a}}$ to $\mathbb{Z}/4\mathbb{Z}$ (represented as $\mathbb{F}_2 \times \mathbb{F}_2$ with the group law \eqref{cohomological_group_structure}), unique up to restriction of characters from $\Omega_{\infty}$, is given by
$$\psi'(\underline{a})\coloneqq\left(\sum_{i \in \mathbb{Z}_{\geq 1}}a_i\widetilde{\phi}_i(x),\phi_0\right).
$$
Its double is the restriction of $\sum_{i \in \mathbb{Z}_{\geq 1}}a_i\phi_i=\phi_0$, which is the unique non-trivial element of $2(M_{\underline{a}}^{\emph{ab}})^{\vee}$.

\item If $a_0=0$, then $M_{\underline{a}}^{\emph{ab}}$ is an $\mathbb{F}_2$-vector space, and a character from $M_{\underline{a}}$ to $\mathbb{F}_2$ (unique modulo restriction of characters from $\Omega_{\infty}$) generating
$$\frac{(M_{\underline{a}}^{\emph{ab}})^{\vee}}{(\Omega_{\infty}^{\emph{ab}})^{\vee}_{|M_{\underline{a}}}}
$$
is given by: 
$$\psi(\underline{a})\coloneqq\sum_{i \in \mathbb{Z}_{\geq 1}}a_i\widetilde{\phi}_i(x).
$$ 
\end{enumerate}

\end{theorem}
\begin{proof}

 To start, notice that we have a chain of subgroups:
 $$[M_{\underline{a}},M_{\underline{a}}]\subseteq [\Omega_\infty,\Omega_\infty]\subseteq M_{\underline{a}}\subseteq\Omega_\infty.$$
 Hence, $\displaystyle\ker\varphi= \frac{[\Omega_\infty,\Omega_\infty]}{[M_{\underline{a}},M_{\underline{a}}]}$. Let $N$ be the largest non-negative integer with $a_{N} \neq 0$, and let $I_{\Omega_{N}}$ be the augmentation ideal in $\mathbb{F}_2[\Omega_N]$. We will first prove the following claim:
 \begin{equation}\label{claim}
   \mbox{every class in $\ker\varphi$ can be represented by an element of $I_{\Omega_N}\F_2^{L_N}$.}   
 \end{equation}
 Here we are thinking of $\F_2^{L_N}$ as the subgroup $(\ldots,0,\ldots,0,\F_2^{L_N},0,\ldots,0)\subseteq\Omega_\infty$, using the iterated semiproduct description of $\Omega_\infty$. We will establish a more general version of the above claim in Proposition \ref{index at most 2 derived}. However, the reader only interested in the present theorem will find the proof self-contained.
 
  To prove \eqref{claim}, first notice that using the definition of $M_{\underline{a}}$ one easily has that:
 \begin{equation}\label{augmentation1}
  I_{\Omega_{N}} \mathbb{F}_2^{L_{N}}=\mathbb{F}_2^{L_{N}} \cap M_{\underline{a}}=\text{ker}(\phi_N),
 \end{equation}
 since $I_{\Omega_N}\F_2^{L_N}$ is generated by the set $\{\lambda_{t,w}\colon t,w\in L_N, t\neq w\}$, where $\lambda_{t,w}$ is the element of $\F_2^{L_N}$ defined by $(\varepsilon_v)_{v\in L_N}$ with $\varepsilon_v=1$ if and only if $v\in \{t,w\}$.

 Consider a commutator $x'=[\sigma,\tau]$, where $\sigma,\tau\in\Omega_\infty$. One can write $\sigma=(\sigma',\sigma_N)$ and $\tau=(\tau',\tau_N)$ where $\sigma_N,\tau_N\in\Omega_N$ and $\sigma',\tau'\in \ker(\Omega_{\infty}\to\Omega_N)$. Thus $x'=(\lambda,[\sigma_N,\tau_N])$ for some $\lambda\in \ker(\Omega_{\infty}\to\Omega_N)$. Since the natural projection $M_{\underline{a}}\to \Omega_N$ is surjective, there exist $\sigma'',\tau''\in\ker(M_{\underline{a}}\to \Omega_N)$ such that $(\sigma'',\sigma_N),(\tau'',\tau_N)\in M_{\underline{a}}$. Let $y'=[(\sigma'',\sigma_N),(\tau'',\tau_N)]\in [M_{\underline{a}},M_{\underline{a}}]$. It is then clear that $x'{y'}^{-1}\in \ker(\Omega_\infty\to\Omega_N)\cap [\Omega_\infty,\Omega_\infty]$. Now we claim that:
 \begin{equation}\label{kernels}
 [M_{\underline{a}},M_{\underline{a}}]\cap\ker(\Omega_\infty\to\Omega_{N+1})=[\Omega_\infty,\Omega_\infty]\cap\ker(\Omega_\infty\to\Omega_{N+1}).
 \end{equation}
We remark that this claim could be proved by an elementary ad-hoc argument relying on the fact that $M_{\underline{a}}=\ker(\Omega_\infty\to\Omega_{N+1}) \rtimes \text{Im}(M_{\underline{a}} \to \Omega_{N+1})$ combined with the fact that for $\underline{a} \neq (1,0,\ldots,0,\ldots)$ and each $N \in \mathbb{Z}_{\geq 0}$, the group $M_{\underline{a}}$ acts transitively on $L_N$. Here instead we refer to Proposition \ref{control at infinite level}, with $j=1$. 

It follows by \eqref{kernels} that there exists $z\in [M_{\underline{a}},M_{\underline{a}}]\cap\ker(\Omega_\infty\to\Omega_{N+1})$ such that $xy^{-1}z\in \F_2^{L_N}\cap M_{\underline{a}}$, and by \eqref{augmentation1} this proves \eqref{claim}.
 
 Now consider $[M_{\underline{a}},I_{\Omega_N}\F_2^{L_N}]$. On the one hand, by \eqref{augmentation1} this is a subset of $[M_{\underline{a}},M_{\underline{a}}]$. On the other hand, by a direct computation (see also, more in general, Proposition \ref{reducing descending central to augmentational}) one can check that it contains $I^2_{\Omega_N}\F_2^{L_N}$, because the natural projection $M_{\underline{a}}\to \Omega_N$ is surjective. Hence, it follows that:
 \begin{equation}\label{augmentation2}
  I_{\Omega_{N}}^2 \mathbb{F}_2^{L_{N}} \subseteq [M_{\underline{a}},M_{\underline{a}}].
 \end{equation}

Next, we claim that $I_{\Omega_N}^2 \mathbb{F}_2^{L_N}$ consists of all the elements $v \in \mathbb{F}_2^{L_N}$ such that $\widetilde{\phi}_N(x)(v)=\widetilde{\phi}_N(y)(v)=0$. In fact, first observe that the description of $I_{\Omega_N}\F_2^{L_N}$ given below \eqref{augmentation1} shows that $I_{\Omega_N}^2 \mathbb{F}_2^{L_N}$ is generated by all elements of the form $(1-\sigma)\lambda_{t,w}$, as $\sigma$ runs over $\Omega_N$ and $t,w$ over all distinct pairs of words in $L_N$. But it is clear that $\widetilde{\phi}_N(x)((1-\sigma)\lambda_{t,w})=\widetilde{\phi}_N(x)(\lambda_{t,w})+\widetilde{\phi}_N(x)(\sigma\lambda_{t,w})$, and the latter coincides with:
$$\begin{cases} 2\widetilde{\phi}_N(x)(\lambda_{t,w})=0 & \mbox{if } \phi_0(\sigma)=0\\
				    \widetilde{\phi}_N(x)(\lambda_{t,w})+\widetilde{\phi}_N(y)(\lambda_{t,w})=\phi_N(\lambda_{t,w})=0 & \mbox{if } \phi_0(\sigma)=1\\
                                                                                                                        
                                                                                                                       \end{cases}.
$$
Thus, every element $v\in I_{\Omega_N}^2 \mathbb{F}_2^{L_N}$ satisfies $\widetilde{\phi}_N(x)(v)=\widetilde{\phi}_N(y)(v)=0$. On the other hand, the space of all such elements is clearly generated by all $\lambda_{t,w}$'s such that $t,w$ are of the form $t'x,w'x$ or  of the form $t'y,w'y$ for some $t',w'\in L_{N-1}$ (notice that $N\geq 1$ thanks to our assumption on $\underline{a}$). To see that any of these $\lambda_{t,w}$'s lies in $I_{\Omega_N}^2 \mathbb{F}_2^{L_N}$, let without loss of generality $t=t'x$ and $w=w'x$. Consider now the element $\lambda_{t'y,w}\in I_{\Omega_N} \mathbb{F}_2^{L_N}$. It is immediate to check that there exists $\sigma\in \Omega_N$ such that $\sigma(t'y)=t'x$ and $\sigma(w)=t$; this shows that $(1-\sigma)(\lambda_{t'y,w})=\lambda_{t,w}$ and therefore implies the claim.

The discussion above proves that $\dim_{\mathbb{F}_2}\left(\frac{I_{\Omega_{N}} \mathbb{F}_2^{L_{N}}}{I_{\Omega_{N}}^2 \mathbb{F}_2^{L_{N}}}\right)=1$. This, together with \eqref{claim} and \eqref{augmentation2}, shows immediately that $|\ker\varphi| \leq 2$. 

We will show that $|\text{ker}(\phi)|= 2$, by producing an homomorphism $M_{\underline{a}} \to \Q_2/\Z_2$ that is not the restriction of any of the characters of $\Omega_{\infty}$. As we shall see if $a_0=0$ this new homomorphism has order $2$, while instead if $a_0=1$ it has order $4$. 

Suppose first that $a_0=0$. Then a calculation based on the uncertain additivity \eqref{uncertain_additivity} of $\widetilde{\phi}_s(x)$ gives us that
$$\psi(\underline{a})=\sum_{s \in \mathbb{Z}_{\geq 1}}a_s\widetilde{\phi}_s(x) \in (M_{\underline{a}}^{\text{ab}})^{\vee}.
$$

Observe that choosing $y$ instead of $x$ makes literally no difference: by definition of $M_{\underline{a}}$ they are the same map. It is very easy to see that $\psi(\underline{a})$ cannot be the restriction of a character from $(\Omega_{\infty}^{\text{ab}})^{\vee}$. Indeed it is enough to pick one $i \in \mathbb{Z}_{\geq 1}$ with $a_i=1$ (which is possible since $\underline{a}\neq e_0$) and an element $\tau\coloneqq(\tau_h)_{h \in \mathbb{Z}_{\geq 0}} \in \Omega_{\infty}$ satisfying the following three properties:
\begin{itemize}
\item $\tau_j=0$ for every $j \neq i$;
\item $\tau_i\coloneqq(\varepsilon_w)_{w \in \mathbb{F}_2^{L_N}} \in \mathbb{F}_2^{L_i}$ with $\varepsilon_{wx}=\varepsilon_{wy}$ for each $w \in L_{i-1}$;
\item $\sum_{w \in L_{i-1}}\varepsilon_{wx}=1$.
\end{itemize}
 One has that $\tau \in [\Omega_{\infty},\Omega_{\infty}]\cap M_{\underline{a}}$, but $\psi(\underline{a})(\tau)=1$ by construction. 

Finally, suppose that $a_0=1$. Let
$$\psi'(\underline{a}): M_{\underline{a}} \to \F_2\times\F_2$$
be defined by the assignment
$$\sigma \mapsto \left(\sum_{i \in \mathbb{Z}_{\geq 1}}a_i\widetilde{\phi}_i(x)(\sigma),\phi_0(\sigma)\right).
$$
Here we are using the group structure \eqref{cohomological_group_structure} on $\F_2\times\F_2$; this yields an isomorphism $(\F_2\times\F_2;\star)\to (\Z/4\Z;+)$. As we show next, the uncertain additivity \eqref{uncertain_additivity} of $\widetilde{\phi}_i(x)$ implies that $\psi'(\underline{a})$ is a group homomorphism that surjects onto $\mathbb{Z}/4\mathbb{Z}$. In fact, let $\sigma_1,\sigma_2$ be in $M_{\underline{a}}$. 
From $(2)$ we deduce that:

$$\psi'(\underline{a})(\sigma_1\sigma_2)=\begin{cases}
                                          \left(\sum_{i \in \mathbb{Z}_{\geq 1}}a_i(\widetilde{\phi}_i(x)(\sigma_1)+\widetilde{\phi}_i(y)(\sigma_2)),\phi_0(\sigma_1)+\phi_0(\sigma_2)\right)  & \mbox{if } \phi_0(\sigma_1)=1 \\
                                          \left(\sum_{i \in \mathbb{Z}_{\geq 1}}a_i(\widetilde{\phi}_i(x)(\sigma_1)+\widetilde{\phi}_i(x)(\sigma_2)),\phi_0(\sigma_1)+\phi_0(\sigma_2)\right)  & \mbox{if } \phi_0(\sigma_1)=0\\ 
                                         \end{cases}.
$$
Notice that, since $a_0=1$, for $j\in\{1,2\}$ we have that:
\begin{equation}\label{useful_relation}
 \sum_{i \in \mathbb{Z}_{\geq 1}}a_i(\tilde{\phi_i}(x)(\sigma_j)+\tilde{\phi_i}(y)(\sigma_j))=\sum_{i \in \mathbb{Z}_{\geq 1}}a_i\phi_i(\sigma_j)=\phi_0(\sigma_j).
\end{equation}

Now it is just a matter of checking, using \eqref{useful_relation}, that for any value of $\phi_0(\sigma_1),\phi_0(\sigma_2)\in\F_2$, we have that $\psi'(\underline{a})(\sigma_1\sigma_2)=\psi'(\underline{a})(\sigma_1)\star\psi'(\underline{a})(\sigma_2)$.
So we have shown that $\psi'(\underline{a})$ is indeed an homomorphism. It is clearly surjective, because in $M_{\underline{a}}$ there are $\sigma$ with $\phi_0(\sigma) \neq 0$, since $\underline{a} \neq (1,0,\ldots,0,\ldots)$: for any such $\sigma$ we have that $\psi'(\underline{a})(\sigma)$ has order $4$. In this case switching $x$ with $y$ gives rise to the \emph{opposite} homomorphism, trivially by the definition of $M_{\underline{a}}$. Moreover the double of this homomorphism is the restriction of the character $\sum_{i \in \mathbb{Z}_{\geq 1}}a_i\phi_i$ to $M_{\underline{a}}$ which also equals the restriction of $\phi_0$ to $M_{\underline{a}}$ thanks to the definition of $M_{\underline{a}}$.
\end{proof}

\begin{remark}
 When $\underline{a}=(1,0, \ldots ,0, \ldots)$, we have that $M_{\underline{a}} \cong \Omega_{\infty}^2$, and thus its abelianization is easily obtained from that of $\Omega_{\infty}$.
\end{remark}

\section{Criteria for arboreal representations of index two} \label{translations}
In this section we shall use the material of Section \ref{subgroups of index 2, abstract} to deduce necessary and sufficient criteria for a quadratic polynomial to have an arboreal representation of index two. From now on, we let $f=(x-\gamma)^2-\delta$ have coefficients in a field $F$ of characteristic $\neq 2$.

We shall begin translating the maps $\phi_n,\widetilde{\phi}_n(x),\widetilde{\phi}_n(y)$, introduced in Section \ref{subgroups of index 2, abstract}, in terms of the arithmetic of the adjusted post-critical orbit of $f$, which we next define. 

\begin{definition}
 The \emph{adjusted post-critical orbit} of $f$ is the sequence defined by:
 $$c_0\coloneqq -f(\gamma),\quad c_n\coloneqq f^{(n+1)}(\gamma) \mbox{ for } n\geq 1.$$
\end{definition}
Next, let $\iota \in \text{Isom}_{\text{graph}}(T_{\infty}(f),T_{\infty})$ and let $n \in \mathbb{Z}_{\geq 1}$. Let $s \in \{x,y\}$. We put
$$\widetilde{c}_n(s,\iota)\coloneqq c_{n-1}-\iota^{-1}(s).
$$
Notice that $\widetilde{c}_n(x,\iota)\widetilde{c}_n(y,\iota)=c_n$ and $\widetilde{c}_n(x,\iota)+\widetilde{c}_n(y,\iota)=2(c_{n-1}-\gamma)$. The set $\{\widetilde{c}_n(x,\iota),\widetilde{c}_n(y,\iota)\} \subseteq F(\sqrt{c_0})$ does not depend on $\iota$, and choosing a different $\iota'$ will swap the two elements if and only if $\phi_0(\iota'\circ \iota^{-1})=1$. 

For a finite extension $E/F$ inside $F^{\text{sep}}$ we denote by $G_E$ the closed subgroup of $G_F$ corresponding to it by Galois theory. For a $t$ in $E^{\times}$ we denote by 
$$\chi_t: G_E \to \mathbb{F}_2,
$$
the quadratic character satisfying the formula $\sigma(r)=(-1)^{\chi_t(\sigma)}r$ for every $\sigma\in G_E$ and each $r \in F^{\text{sep}}$ with $r^2=t$.

For each $n \in \mathbb{Z}_{\geq 0}$ we put $K_n$ to be the splitting field of $f^{(n)}$ in $F^{\text{sep}}$ and we denote by 
$$G_n\coloneqq\gal(K_n/F),
$$ 
the Galois group of the polynomial $f^{(n)}$. Let $\rho_f\colon G_F\to \Omega_{\infty}(f)$ be the associated arboreal representation. The next proposition is of crucial importance, since it is the tool that allows to relate the algebraic structure of $\Omega_\infty$ to the arithmetic of the adjusted post-critical orbit. We remark that part (1), in a less general form, had already been noticed in the arguments of \cite[Section 5]{odoni2}.

\begin{proposition} \label{field theoretic description of phi,tildephi}
Fix $\iota$ in $\emph{Isom}_{\emph{graph}}(T_{\infty}(f),T_{\infty})$. Then following hold.
\begin{enumerate}
 \item Let $n \in \mathbb{Z}_{\geq 0}$. Then we have that :
$$\phi_n \circ \iota_{*} \circ \rho_f=\chi_{c_n}.
$$
\item Let $n \in \mathbb{Z}_{\geq 1}$ and $s \in \{x,y\}$. Then we have that:
$${(\widetilde{\phi}_n(s) \circ \iota_{*} \circ \rho_f)}_{|G_{K_1}}=\chi_{\widetilde{c}_n(s,\iota)}.
$$
\end{enumerate}

\end{proposition}

\begin{proof}
 (1) Fix a square root $r\coloneqq \sqrt{c_n}$. We have to show that for every $\sigma\in G_F$ one has that $\sigma(r)=(-1)^{(\phi_n \circ \iota_{*} \circ \rho_f)(\sigma)}r$. The very definition of $\phi_n$ immediately implies that $(\phi_n \circ \iota_{*} \circ \rho_f)(\sigma)=\phi_n  ((\iota_*\circ \rho_f(\sigma))_n)$, where the subscript $n$ denotes the $n$-th entry in the digital representation. Such digital representation, and the corresponding map $\phi_n$, can of course be transferred, via $\iota^{-1}$, to $\Omega_\infty(f)$, and we will denote the $n$-th entry of this representation again by a subscript $n$. The element $\rho_f(\sigma)_n$ acts on the $n+1$-th level of the tree $T_\infty(f)$ and fixes the first $n$ levels. This means that it can only swap between them nodes which have the same parent. Hence $\phi_n(\rho_f(\sigma))$ counts the parity of the number of pairs of nodes at level $n+1$ that are swapped by $\rho_f(\sigma)$. Let now $\alpha_1,\ldots,\alpha_{2^{n+1}}\in T_\infty(f)$ be the nodes at level $n+1$, ordered so that $f(\alpha_i)=f(\alpha_{i+1})$ for every odd $i\in\{1,\ldots,2^{n+1}-1\}$.\footnote{This is equivalent to asking that $\alpha_i$ and $\alpha_{i+1}$ have the same parent for every odd $i$, by definition of $T_\infty(f)$.} Recalling that $\gamma$ is the finite critical point of $f$, one checks easily that: 
 $$\prod_{i=1}^{2^{n+1}}(\alpha_i-\gamma)=\begin{cases}c_0=-r^2 & \mbox{ if }n=0\\
 														c_n=r^2 & \mbox{ if }n\geq 1\\ \end{cases}.$$
 On the other hand having ordered the nodes the way we did, we have that:
 \begin{equation}\label{ordering}
  \alpha_{i+1}-\gamma=-(\alpha_i-\gamma) \mbox{ for every odd } i\in\{1,\ldots,2^{n+1}-1\}.
 \end{equation}
 Hence, we can assume without loss of generality that $r=\gamma-\alpha_1$ when $n=0$ and $r=\prod_{i \text{ odd}}(\alpha_i-\gamma)$ when $n\geq 1$. Since of course $\rho_f(\sigma)(\alpha_i-\gamma)=\rho_f(\sigma)(\alpha_i)-\gamma$, and the latter is $\alpha_i-\gamma$ if $\rho_f(\sigma)$ does not swap $(\alpha_i,\alpha_{i+1})$ and is $\alpha_{i+1}-\gamma$ otherwise, from \eqref{ordering} it follows immediately that:
 $$\sigma(r)=(-1)^{\phi_n(\rho_f(\sigma))}r.$$
 
 (2) The proof is exactly the same as the one of point (1), but considering only the half of the tree $T_{\infty}(f)$ that corresponds to $T_\infty s$ via $\iota$. Clearly this time the key relation will be that:
 $$\prod_{i=1}^{2^{n+1}}(\alpha_i-\gamma)=c_n-\iota^{-1}(s)=\widetilde{c}_n(s,\iota).$$
\end{proof}

For every $\underline{a} \in \mathbb{F}_2^{(\mathbb{Z}_{\geq 0})}$, define: 
$$c_{\underline{a}}\coloneqq\prod_{i \in \mathbb{Z}_{\geq 0}}c_i^{a_i}.$$

\begin{corollary} \label{when you are in Ma}
Let $\underline{a}$ be in $\mathbb{F}_2^{(\mathbb{Z}_{\geq 0})}$. Then $\emph{Im}(\rho_f) \subseteq M_{\underline{a}}$ if and only if $c_{\underline{a}} \in {F^{\times}}^2$.
\end{corollary}
\begin{proof}
By definition of $M_{\underline{a}}$ we have that $\text{Im}(\rho_f) \subseteq M_{\underline{a}}$ if and only if for one (equivalently any) $\iota$ in $\text{Isom}_{\text{graph}}(T_{\infty}(f),T_{\infty})$ we have that $\sum_{i \in \mathbb{Z}_{\geq 0}}a_i(\phi_i \circ \iota_{*} \circ \rho_f)=0$. Thanks to part (1) of Proposition \ref{field theoretic description of phi,tildephi} we have that this is equivalent to  $c_{\underline{a}}=\prod_{i \in \mathbb{Z}_{\geq 0}}c_i^{a_i} \in {F^{\times}}^2$. 
\end{proof}

Corollary \ref{when you are in Ma} is the key, together with Theorem \ref{Describing abelianized Ma}, to relate maximal subgroups of $\Omega_\infty$ to post-critical orbits of quadratic polynomials. To show its strength let us explain how one can use it to immediately prove that if $f\in F[x]$ is a quadratic polynomial and $f$ is post-critically finite or $F^{\times}/(F^{\times})^2$ is a finite group, then $\im(\rho_f)$ has infinite index in $\Omega_\infty(f)$. Notice that it was already known that post-critically finite rational functions over global field have arboreal representations of infinite index (see \cite[Theorem 3.1]{jones2}). However, the proof involves the use of a delicate theorem of Ihara on the Galois group of the maximal extension of a global field unramified outside a finite set. Our proof instead, although valid only for quadratic polynomials, does not depend in any way on the arithmetic of the ground field. The second condition, namely the fact that $F^{\times}/(F^{\times})^2$ is a finite group, is satisfied for example by finite extensions of $\Q_p$. Thus, we recover in a few lines part of a result of Anderson et al.\ \cite{anderson}.

\begin{corollary}\label{PCF_polynomials}
 Let $f\in F[x]$ be monic and quadratic. If $\dim\langle c_1,\ldots,c_n,\ldots\rangle_F<\infty$, then $[\aut(T_\infty(f)):\im(\rho_f)]=\infty$. Therefore, if one of the following two holds:
 \begin{enumerate}[a)]
  \item $f$ is post-critically finite;
  \item $F^{\times}/(F^{\times})^2$ is a finite group
 \end{enumerate}
 then $[\aut(T_\infty(f)):\im(\rho_f)]=\infty$.
\end{corollary}
\begin{proof}
 Since $\dim\langle c_1,\ldots,c_n,\ldots\rangle_F<\infty$, then by Corollary \ref{when you are in Ma} there exists an infinite, linearly independent set $\{\underline{a}_n\}_{n\in \N}\subseteq\mathbb{F}_2^{(\mathbb{Z}_{\geq 0})}$ such that $\im(\rho(f))\subseteq M_{\underline{a}_n}$ for every $n$. Thus, $\im(\rho_f)\subseteq \bigcap_{n\in\N}M_{\underline{a}_n}$, and the latter clearly has infinite index in $\Omega_\infty$.
\end{proof}

In the next two subsections we respectively review Stoll's criterion for surjectivity, under our point of view, and introduce a criterion for representations of index two. 
\subsection{Surjective Arboreal representations for quadratic polynomials}\label{stoll_revisited}
We recall that if $G$ is a pro-$2$-group then its maximal closed subgroups are precisely the kernels of the non-trivial continuous homomorphisms
$$\chi:G \to \mathbb{F}_2,
$$
and that every closed subgroup $H$ of $G$ is contained in some maximal closed subgroup of $G$. It follows that for a closed subgroup $H$ of $G$ it is equivalent to say that $H=G$ and to say that for every non-trivial continuous homomorphism $\chi:G \to \mathbb{F}_2$, we have that $\chi(H) \neq \{0\}$. Since the space of continuous characters from $\Omega_{\infty}$ to $\mathbb{F}_2$ is precisely the span of the set $\{\phi_n\}_{n \in \mathbb{Z}_{\geq 0}}$ one sees at once that for a closed subgroup $H$ of $\Omega_{\infty}$:
$$H=\Omega_{\infty} \mbox{ if and only if }\{(\phi_n)_{|H}\}_{n \in \mathbb{Z}_{\geq 0}} \mbox{ is linearly independent}.$$

Now if $H=\text{Im}(\rho_f)$ we see, through part (1) of Proposition \ref{field theoretic description of phi,tildephi}, that:
$$\text{Im}(\rho_f)=\Omega_{\infty}\mbox{ if and only if }\{c_n\}_{n \in \mathbb{Z}_{\geq 0}} \mbox{ is linearly independent in } F^{\times}/{F^{\times}}^2.$$
By the same logic one obtains that $G_N=\Omega_N$ if and only if $\{c_n\}_{0 \leq n \leq N-1}$ forms a linearly independent set in $F^{\times}/{F^{\times}}^2$. This fact was established in \cite{Stoll}.
\begin{theorem}[{{\cite{Stoll}}}] \label{Stolltheorem}
	Let $n\in \mathbb{Z}_{\geq 0}$ and $G_n$ be the Galois group of $f^{(n)}$. Then $G_n\cong\Omega_n$ if and only if:
	\[
		\dim \langle c_{0}, \dots, c_{n-1} \rangle_F = n.
	\]
\end{theorem}
As we remarked above this implies that $\im(\rho_f)=\Omega_{\infty}$ if and only if the set $\{c_n\}_{n\in \mathbb{Z}_{\geq 0}}$ is linearly independent modulo squares.

\subsection{Arboreal representations of index two} 

Let us start by recalling a general, standard fact that will be useful later.
\begin{proposition} \label{quadratic on quadratic}
Let $E$ be a field of characteristic different from $2$. Let $L=E(\sqrt{a})$ be a quadratic extension, denote by $\sigma$ the unique non-trivial element of $\gal(L/E)$. Let $t\in L^{\times}$. Then the extension $L(\sqrt{t})$ remains Galois over $E$ if and only $\emph{Nm}_{L/E}(t) \in {L^{\times}}^2$. If that is the case, then either $\emph{Nm}_{L/E}(t)=ay^2$ for some $y \in E^{\times}$ and then $\emph{Gal}(L(\sqrt{t})/E) \cong \mathbb{Z}/4\mathbb{Z}$ or $\emph{Nm}_{L/E}(t)={y'}^2$ for some $y' \in E^{\times}$ and then $\emph{Gal}(L(\sqrt{t})/E)$ is of exponent two and $L(\sqrt{t})=E(\sqrt{a},\sqrt{2y'+t+\sigma(t)})$.
\end{proposition}
\begin{proof}
By Kummer theory, we have that $L(\sqrt{t})$ remains Galois over $E$ if and only if $t$ is in $(L^{\times}/{L^{\times}}^2)^{\gal(L/E)}$, which is equivalent to say that $t\sigma(t) \in {L^{\times}}^2$. Since $t\sigma(t)\in E^{\times}$ as well, one checks immediately that it is either in ${E^{\times}}^2$ or in $a{E^{\times}}^2$. Assume the latter holds. If $L(\sqrt{t})/E$ is not cyclic then, again by Kummer theory, $t$ is equivalent modulo ${L^{\times}}^2$ to some $h \in E^{\times}$, hence taking norm yields that $a \in {E^{\times}}^2$ which contradicts that $L/E$ is quadratic. So in this case $L(\sqrt{t})/E$ must be cyclic and again it must be of degree $4$, otherwise one concludes that $a \in {E^{\times}}^2$. Hence $\gal(L(\sqrt{t})/E) \cong\mathbb{Z}/4\mathbb{Z}$. Now suppose that $\text{Nm}_{L/E}(t)=y'^2$ with $y' \in E$. Then we have that:
$$\frac{t}{y'}=\frac{1+\frac{t}{y'}}{1+\sigma\left(\frac{t}{y'}\right)}=\left(1+\frac{t}{y'}\right)\left(1+\sigma\left(\frac{t}{y'}\right)\right) \cdot \frac{1}{\left(1+\sigma\left(\frac{t}{y'}\right)\right)^2}.
$$
This yields that up to squares in $L^{\times}$ the element $t$ equals $2y'+t+\sigma(t)\in E$. \footnote{This calculation consists of applying the proof of Hilbert's $90$ to $\frac{t}{y'}$.}
\end{proof}

The next step is to understand how the dichotomy between $a_0=0$ and $a_0=1$ in Theorem \ref{Describing abelianized Ma} is reflected in the two cases of Proposition \ref{quadratic on quadratic}. When $a_0=0$, fix $\iota \in \text{Isom}_{\text{graph}}(T_{\infty}(f),T_{\infty})$ and $s \in \{x,y\}$. Define
$$\widetilde{c}_{\underline{a}}(s,\iota)\coloneqq\prod_{i \in \mathbb{Z}_{\geq 1}}{\widetilde{c}_i(s,\iota)}^{a_i}=\prod_{i\in\Z_{\geq 1}}(c_{i-1}-\iota^{-1}(s))^{a_i}.
$$

Notice that:
\begin{equation}\label{ca_relation}
 c_{\underline{a}}=c_0^{a_0}\cdot\prod_{i\geq 1}c_i^{a_i}=c_0^{a_0}\cdot\widetilde{c}_{\underline{a}}(x,\iota)\widetilde{c}_{\underline{a}}(y,\iota).
\end{equation}
Moreover, if $f$ is irreducible then the elements $\widetilde{c}_{\underline{a}}(x,\iota)$ and $\widetilde{c}_{\underline{a}}(y,\iota)$ are $\gal(K_1/F)$-conjugates, and hence the element $\widetilde{c}_{\underline{a}}(x,\iota)+\widetilde{c}_{\underline{a}}(y,\iota)$ is in $F$.

Recall from Theorem \ref{Describing abelianized Ma} that when $a_0=1$ there is a character $M_{\underline{a}}\to (\F_2\times\F_2;\star)$, where $(\F_2\times\F_2;\star)\cong (\Z/4\Z;+)$, given by $\psi'(\underline{a})=\left(\sum_{i\geq 1}a_i\widetilde{\phi}_i(x),\phi_0\right)$. When $a_0=0$, there is a character $M_{\underline{a}}\to \F_2$ given by $\psi(\underline{a})=\sum_{i\geq 0}a_i\widetilde{\phi}(x)$.

\begin{proposition} \label{The abelian "additional" extension}
Let $\underline{a}$ be in $\mathbb{F}_2^{(\mathbb{Z}_{\geq 0})}$ different from $(1,0,\ldots,0,\ldots)$. Assume that $c_{\underline{a}}$ is in ${F^{\times}}^2$ and that $c_0\notin {F^{\times}}^2$.
\begin{enumerate}
 \item Suppose that $a_0=1$. Then the extension $K_1(\sqrt{\widetilde{c}_{\underline{a}}(s,\iota)})/F$ is independent of the choice of $s$ in $\{x,y\}$ and it is a cyclic extension of degree $4$. Furthermore we have that:
$$K_1\left(\sqrt{\widetilde{c}_{\underline{a}}(s,\iota)}\right)=(F^{\emph{sep}})^{\emph{ker}(\psi'(\underline{a}) \circ \iota_{*} \circ \rho_f)}.$$
\item  Suppose that $a_0=0$. Then there exists $d \in F^{\times}$ such that $\widetilde{c}_{\underline{a}}(x,\iota)\widetilde{c}_{\underline{a}}(y,\iota)=d^2$. Hence the extension $K_1(\sqrt{\widetilde{c}_{\underline{a}}(s,\iota)})/F$ is independent of the choice of $s \in \{x,y\}$ and is equal to the extension
$$F\left( \sqrt{c_0},\sqrt{2d+\widetilde{c}_{\underline{a}}(x,\iota)+\widetilde{c}_{\underline{a}}(y,\iota)}\right)/F,$$
which is a Galois extension of exponent $2$. Finally, denoting $\widetilde{c}_{\underline{a}}\coloneqq 2d+\widetilde{c}_{\underline{a}}(x,\iota)+\widetilde{c}_{\underline{a}}(y,\iota)$, we have:
$$K_1\left(\sqrt{\widetilde{c}_{\underline{a}}}\right)=(F^{\emph{sep}})^{\ker(\psi(\underline{a}) \circ \iota_{*} \circ \rho_f)_{|G_{K_1}}}.
$$
\end{enumerate}
\end{proposition}
\begin{proof}
The fact that the extension is cyclic of degree 4 in (1) and of exponent $2$ in (2) follows immediately from \eqref{ca_relation} and Proposition \ref{quadratic on quadratic}. The fact that in case (1) we have $K_1\left(\sqrt{\widetilde{c}_{\underline{a}}(s,\iota)}\right)=(F^{\text{sep}})^{\text{ker}(\psi'(\underline{a}) \circ \iota_{*} \circ \rho_f)}$ and in case (2) we have $K_1(\sqrt{\widetilde{c}_{\underline{a}}})=(F^{\text{sep}})^{\ker(\psi(\underline{a}) \circ \iota_{*} \circ \rho_f)_{|G_{K_1}}}$ follows immediately from Proposition \ref{field theoretic description of phi,tildephi}. 
\end{proof}

Let now $\underline{a}$ be an element of $\mathbb{F}_2^{(\mathbb{Z}_{\geq 0})}$, and suppose that any of the equivalent conditions of Corollary \ref{when you are in Ma} is satisfied. When $a_0=0$, denote by $d$ the element given by part (2) of Proposition \ref{The abelian "additional" extension}. Recall that we denote
$$\widetilde{c}_{\underline{a}}\coloneqq 2d+\widetilde{c}_{\underline{a}}(x,\iota)+\widetilde{c}_{\underline{a}}(y,\iota).
$$
\begin{example}\label{ctilda}
To clarify the construction of $\widetilde{c}_{\underline{a}}$, let us see an example that will become useful later. Let $f=x^2-\delta\in F[x]$ and $\underline{a}=(0,1,1,0,\ldots,0,\ldots)$, so that $c_1c_2\in F^2$. Fix $\iota \in \text{Isom}_{\text{graph}}(T_{\infty}(f),T_{\infty})$, so that $\iota(\alpha)=x$ where $\alpha$ is a fixed root of $f$. Then $\widetilde{c}_{\underline{a}}(x,\iota)=(c_0-\alpha)(c_1-\alpha)$ and $\widetilde{c}_{\underline{a}}(y,\iota)=(c_0+\alpha)(c_1+\alpha)$. The element $d$ described in Proposition \ref{The abelian "additional" extension} is defined by $d^2=\widetilde{c}_{\underline{a}}(x,\iota)\widetilde{c}_{\underline{a}}(y,\iota)=c_1c_2$. Hence, we get that:
$$\widetilde{c}_{\underline{a}}=2(c_0+c_0c_1+\sqrt{c_1c_2}).$$
\end{example}
The following theorem yields necessary and sufficient conditions for $\im(\rho_f)$ to have index two in $\Omega_\infty$. 
\begin{theorem} \label{Classification of realizing stable Ma}
Suppose that $\underline{a}\in \mathbb{F}_2^{(\mathbb{Z}_{\geq 0})}$ is different from $(1,0,\ldots,0,\ldots)$.
\begin{enumerate}
 \item Suppose that $a_0=1$. Then we have that $\emph{Im}(\rho_f)=M_{\underline{a}}$ if and only if $c_{\underline{a}} \in {F^{\times}}^2$ and for each non-zero $\underline{b} \in \mathbb{F}_2^{(\mathbb{Z}_{\geq 0})}$ different from $\underline{a}$ we have that $c_{\underline{b}} \notin {F^{\times}}^2$.
 \item Suppose that $a_0=0$. Then we have that $\emph{Im}(\rho_f)=M_{\underline{a}}$ if and only if $c_{\underline{a}} \in {F^{\times}}^2$, $c_{\underline{b}} \notin {F^{\times}}^2$ for each non-zero $\underline{b}$ in $\mathbb{F}_2^{(\mathbb{Z}_{\geq 0})}$ different from $\underline{a}$ and $\widetilde{c}_{\underline{a}}$ is outside the span of the set $\{c_n\}_{n \in \mathbb{Z}_{\geq 0}}$ in $F^{\times}/{F^{\times}}^2$. 
\end{enumerate}
\end{theorem}
\begin{proof}
Thanks to Corollary \ref{when you are in Ma} we have that $\text{Im}(\rho_f)\subseteq M_{\underline{a}}$ if and only if $c_{\underline{a}} \in {F^{\times}}^2$. Let $\iota$ be in $\text{Isom}_{\text{graph}}(T_{\infty}(f),T_{\infty})$. To have  $\text{Im}(\rho_f)= M_{\underline{a}}$ it is equivalent to ask that $\text{Im}(\iota_{*} \circ \rho_f)$ is not contained in any index two subgroup of $M_{\underline{a}}$. Since by assumption $\underline{a} \neq (1,0,\ldots,0,\ldots)$, we are in position to apply Theorem \ref{Describing abelianized Ma}. In case (1), Theorem \ref{Describing abelianized Ma} tells us that index two closed subgroups of $M_{\underline{a}}$ are precisely the kernels of non-zero characters living in the span of the characters
$$\{{\phi_i}|_{M_{\underline{a}}}\}_{i \in \mathbb{Z}_{\geq 0}}.$$
Furthermore, Theorem \ref{Describing abelianized Ma} tells us also that the only relation among these characters is the trivial one:
$$\sum_{i \in \mathbb{Z}_{\geq 0}}a_i{\phi_i}|_{M_{\underline{a}}}=0.$$
In case (2) Theorem \ref{Describing abelianized Ma} tells us instead that index two subgroups of $M_{\underline{a}}$ are precisely the kernels of non-zero characters living in the span of 
$$\{\phi_i\}_{i \in \mathbb{Z}_{\geq 0}} \cup \left\{\sum_{i \in \mathbb{Z}_{\geq 0}}a_i\widetilde{\phi}_i(x)\right\},
$$
and again the only relation among them is the obvious one:
$$\sum_{i \in \mathbb{Z}_{\geq 0}}a_i{\phi_i}_{|M_{\underline{a}}}=0.
$$
Now the conclusion follows at once from Proposition \ref{field theoretic description of phi,tildephi}.
\end{proof}

We can refine the conclusion of Theorem \ref{Classification of realizing stable Ma} to a conclusion at each finite level.  
 \begin{corollary}\label{big_cor}
Suppose that $\underline{a} \neq (1,0,\ldots,0,\ldots)$. Let $n$ be the largest integer such that $a_n=1$. Suppose that $G_{n-1}\cong\Omega_{n-1}$ and that $c_{\underline{a}} \in {F^{\times}}^2$. Then we have the following:
  \begin{enumerate}[i)]
   \item If $a_0=1$, then $G_n=\im(M_{\underline{a}}\to \Omega_n)$.
   \item If $a_0=0$, then the following two conditions are equivalent:
    \begin{enumerate}
    \item $G_n=\im(M_{\underline{a}}\to \Omega_n)$;
    \item $\dim\langle c_0,\ldots,c_{n-2}, \widetilde{c}_{\underline{a}}\rangle_F=n$.
    \end{enumerate}
  \end{enumerate}
 \end{corollary}

It will be clear from the next subsection that if $f$ satisfies either of the conditions of Theorem \ref{Classification of realizing stable Ma}, then it is \emph{stable}, i.e.\ $f^{(n)}$ is irreducible for every $n \geq 1$.

\subsection{The non-stable case}
Let us start by proving a lemma that shows that $M_{(1,0,\ldots,0,\ldots)}$ plays a unique role among index two closed subgroups of $\Omega_{\infty}$: it is the only one that can appear as image of the representation attached to a non-stable polynomial. 

\begin{lemma}\label{nontransitive}
Let $\underline{a}\in \F_2^{(\mathbb{Z}_{\geq 0})}$. Let $n \geq 1$. Then $M_{\underline{a}}$ acts non-transitively on the set $L_n$ if and only if $\underline{a}=(1,0,\ldots,0)$.
\end{lemma}
\begin{proof}
The subgroup $M_{(1,0,\ldots,0,\ldots)}$ can be naturally identified with $\Omega_{\infty}^2$, since it is the subgroup preserving both subtrees $T_{\infty}x$ and $T_{\infty}y$. This certainly implies that $M_{(1,0,\ldots,0,\ldots)}$ acts non-transitively on each $L_n$ with $n \geq 1$.

Conversely, notice that $\Omega_{\infty}$ embeds diagonally in $\Omega_{\infty}^2$. It is clear from the fact that $[\Omega_{\infty},\Omega_{\infty}]=\cap_{i \in \mathbb{Z}_{\geq 0}}\ker(\phi_i)$ that this diagonal embedding of $\Omega_{\infty}$ ends up in $[\Omega_{\infty},\Omega_{\infty}]$ and therefore is contained in each $M_{\underline{a}}$. On the one hand this diagonal embedding acts transitively on both $T_{\infty}x$ and $T_{\infty}y$. On the other hand if $\underline{a} \neq (1,0,\ldots,0,\ldots)$ then there is $\sigma \in M_{\underline{a}}$ with $\phi_0(\sigma)=1$. Therefore $\sigma(xL_{n-1})=yL_{n-1}$ and hence all elements of $L_n$ lie in the same $M_{\underline{a}}$-orbit. 
\end{proof}
The next theorem gives necessary and sufficient conditions for this to hold.
\begin{theorem}\label{nonstable}
 Let $f=(x-\gamma)^2-\delta\in F[x]$ be non-stable, and let $G_{\infty}\subseteq \Omega_{\infty}$ be the image of the associated arboreal representation. Then the following are equivalent:
 \begin{enumerate}[i)]
  \item $[\Omega_{\infty}:G_{\infty}]=2$;
  \item $G_\infty=M_{(1,0,0,\ldots,0,\ldots)}$;
  \item There exists $u\in F$ such that $\delta=u^2$, and for all $n\geq 2$:
	$$\dim\langle c_0+\gamma\pm u,c_1-\gamma\pm u,\ldots, c_{n-2}-\gamma\pm u\rangle_F=2(n-1),$$
 \end{enumerate}
\end{theorem}
\begin{proof}
 Since $f$ is not stable, there exists $n$ such that $G_n$ acts non-transitively on $L_n$. Thus by Lemma \ref{nontransitive}, it follows immediately that i) and ii) are equivalent.
 
 Next, notice that ii) holds if and only if $G_1$ is the trivial group (i.e.\ $\delta=u^2$ for some $u\in F$) and $\gal(K_n/K_1)\cong \Omega_{n-1}^2$ for every $n\geq 2$, since we have that
 $$\gal(K_n/K_1)=\gal\left((f^{(n-1)}-\gamma-u)\cdot (f^{(n-1)}-\gamma+u)\right)\subseteq$$
 $$\subseteq\gal\left(f^{(n-1)}-\gamma-u\right)\times \gal\left(f^{(n-1)}-\gamma+u\right)\subseteq \Omega_{n-1}^2.$$
 Thus ii) holds if and only if all the above containments are equalities for all $n\geq 2$. This is equivalent to asking that for all $n\geq 2$ the following two conditions hold:
 \begin{enumerate}[a)]
  \item $\gal\left(f^{(n-1)}-\gamma-u\right)\cong \gal\left(f^{(n-1)}-\gamma+u\right)\cong \Omega_{n-1}$;
  \item let $K_n',K_n''$ be the splitting fields of $\gal\left(f^{(n-1)}-\gamma-u\right)$ and $\gal\left(f^{(n-1)}-\gamma+u\right)$, respectively. Then $K_n'\cap K_n''=F$.
 \end{enumerate}
 Let
 $$V'\coloneqq\langle c_0+\gamma+u,c_2-\gamma-u,\ldots,c_{n-2}-\gamma-u\rangle_F$$
 and
 $$V''\coloneqq\langle c_0+\gamma-u,c_2-\gamma+u,\ldots,c_{n-2}-\gamma+u\rangle_F.$$
 Theorem \ref{Stolltheorem} applied to the polynomials $f^{(n-1)}-\gamma-u$ and $f^{(n-1)}-\gamma+u$ shows that a) is equivalent to having $\dim V'=\dim V''=n-1$ for every $n\geq 2$.\footnote{As mentioned in the final remark of the introduction, Theorem \ref{Stolltheorem} is valid for any basepoint.}
 
 Let now $L\coloneqq K_n'\cap K_n''$. Since $\gal(L/F)$ is a 2-group, $L=F$ if and only if $\gal(L/F)$ contains no subgroups of index two, i.e.\ if and only if there is no quadratic extension of $F$ that is contained in both $K_n'$ and $K_n''$. If a) holds, then $\gal\left(f^{(n-1)}-\gamma-u\right)^{\ab}\cong \gal\left(f^{(n-1)}-\gamma+u\right)^{\ab}\cong \F_2^{n-1}$, and by the same logic of Subsection \ref{stoll_revisited}, this time applied to $\Omega_\infty\times\Omega_\infty$, one gets that if a) holds then b) holds if and only if $V'\cap V''$ is zero-dimensional, i.e.\ if and only if $\dim\langle c_0 +\gamma\pm u,\ldots, c_{n-2}-\gamma\pm u\rangle_F=2(n-1)$.
 
 Conversely, if $\dim\langle c_0+\gamma\pm u,\ldots, c_{n-2}-\gamma\pm u\rangle_F=2(n-1)$ then in particular a) holds, and therefore also b) holds.
\end{proof}

\section{Realizing representations of index two over \texorpdfstring{$\Q$}{}}\label{realizations}

Let $t$ be transcendental over $\Q$, and let $\phi=x^2+t\in \Q(t)[x]$. In this section, we will focus on index two subgroups of $\Omega_{\infty}$ that can appear as images of $\rho_{\phi_{t_0}}$, where $t_0\in\Q$, $\phi_{t_0}$ is the specialized polynomial and $\rho_{\phi_{t_0}}$ is the associated arboreal representation. As a first step, we will show that there exist exactly five index two subgroups of $\Omega_{\infty}$ that can appear as $\im(\rho_{\phi_{t_0}})$ for infinitely many $t_0$. Afterwards, we will prove that two of these subgroups do indeed appear infinitely often, by providing explicit examples, and finally we will show that if Vojta's conjecture over $\Q$ holds true, then so do the remaining three. 

\begin{remark}
 Every closed subgroup of $\Omega_{\infty}$ is the image of the arboreal representation attached to a quadratic polynomial over an algebraic extension of $\Q$. In fact, let $f\in \Q[x]$ be a quadratic polynomial such that $G_{\infty}\coloneqq\im(\rho_f)\cong\Omega_{\infty}$, let $K_n$ be the splitting field of $f^{(n)}$ and $K\coloneqq \varinjlim_nK_n$. Let now $G\subseteq G_{\infty}$ be a closed subgroup and fix $L\coloneqq K^{G}$; by Galois theory we have that $\gal(K/L)\cong G$. On the other hand $\gal(K/L)$ is clearly isomorphic to the image of the arboreal representation associated to $f$, when the latter is considered as an element of $L[x]$. Notice that this phenomenon is analogous to what happens with the classical inverse Galois problem: if one does not fix the base field, then every finite group appears as the Galois group of \emph{some} field extension, while the same question over the base field $\Q$ it is still a wide open problem. When $G=M_{\underline{a}}$ for some non-zero $\underline{a}\in\F_2^{(\Z_{\geq 0})}$, Corollary \ref{when you are in Ma} shows explicitly what is $L$ in the construction above: we have $\displaystyle L=\Q\left(\sqrt{\prod_{i\geq 0} c_i^{a_i}}\right)$. Notice that there exist infinitely many examples of quadratic polynomials over $\Q$ satisfying $G_{\infty}=\Omega_{\infty}$, e.g.\ $x^2+a\in \Z[x]$ with $a\equiv 1\bmod 4$ (see \cite{Stoll}).
\end{remark}
 
 In order to prove the results of this section in a clean way, we need to switch back to the usual indexing for the post-critical orbit, i.e.\ we need it to start from 1. However, since we do not want to create confusion in the reader, we will use a slightly different symbol; thus, we let $\mathfrak{c}_1(t)\coloneqq -\phi(0)=-t$ and $\mathfrak c_{n+1}(t)\coloneqq \phi(\mathfrak c_n(t))$ for $n\geq 1$ be the adjusted post-critical orbit of $\phi$. For every $n\geq 1$ we define:
 $$b_n(t)\coloneqq \prod_{d\mid n}\mathfrak c_d(t)^{\mu(n/d)},$$
 where $\mu$ is the M\"obius function. It is proven in \cite[Proposition 6.2]{jones4} that $b_n(t)\in\Q[t]$ for every $n\geq 1$.
 \begin{proposition}\label{separability}
  For every $n\geq 1$, $b_n(t)\in\Q[t]$ is separable, and $b_n(t),b_m(t)$ are coprime if $m\neq n$.
 \end{proposition}
\begin{proof}
 The polynomial $\mathfrak c_1(t)$ is obviously separable, and for $n\geq 2$ we have that $\mathfrak c_n(t)=\mathfrak c_{n-1}(t)^2+t$. Thus $\mathfrak c_n(t)'=2\mathfrak c_{n-1}(t)\mathfrak c_{n-1}(t)'+1$, and reducing modulo 2 it follows immediately that the discriminant of $\mathfrak c_n(t)$ is non-zero, and hence the $\mathfrak c_n(t)$'s are separable. Now the claim follows by M\"obius inversion from the fact that $b_n(t)\in\Q[t]$ for every $n$.
\end{proof}

\begin{proposition}\label{subgroups}
Let $v_1,\ldots,v_5\in \F_2^{\Z_{\geq 0}}$ be defined as follows:
$$v_1=(1,1,0,\ldots,0,\ldots),\quad v_2=(0,1,0,\ldots,0,\ldots),\quad v_3=(1,0,1,0,\ldots,0,\ldots),$$
 $$v_4=(0,1,1,0,\ldots,0\ldots), \quad v_5=(1,0,\ldots,0,\ldots).$$
 Then for every $v\in \F_2^{\Z_{\geq 0}}\setminus\{\underline{0},v_1,\ldots,v_5\}$ there exist only finitely many $t_0\in \Q$ such that $\im(\rho_{\phi_{t_0}})=M_v$. 
\end{proposition}

\begin{proof}
 By Theorem \ref{Stolltheorem}, in order for $\im(\rho_{\phi_{t_0}})$ to have index two in $\Omega_{\infty}$, it is necessary for $t_0$ to be the $t$-coordinate of a point on a curve of the form:
 $$y^2=\prod_{i\in I}\mathfrak c_i(t),$$
 where $I$ is a finite, non-empty subset of $\N$. Thus, our claim follows from the fact that there exist exactly five curves in the above form which have infinitely many rational points. Since $\deg \mathfrak c_n(t)=2^{n-1}$ for every $n$, we have that:
 $$\deg b_n(t)=\sum_{d\mid n}2^{d-1}\mu(n/d).$$
 Thus, as long as $n\geq 4$, we have:
 $$\deg b_n(t)\geq \frac{1}{2}\left(2^n-\sum_{d\leq n-3}2^d\right)\geq\frac{1}{2}(2^n-2^{n-2})\geq 6.$$
 
 Now take a curve $C_I\colon y^2=\prod_{i\in I}\mathfrak c_i(t)$, and let $n\coloneqq \max\{i\in I\}$. By M\"obius inversion, we have that $\displaystyle \mathfrak c_i(t)=\prod_{j\mid i}b_j(t)$ for every $i\in I$. Use this to write $C_I\colon y^2=r(t)\cdot b_n(t)$, where $\displaystyle r(t)=\prod_{\substack{i\in I\\i<n}}\prod_{j\mid i}b_j(t)\cdot\prod_{\substack{d\mid n\\d<n}}b_d(t)$. By Proposition \ref{separability} $r(t)$ and $b_n(t)$ are coprime. Thus, if we write $r(t)\cdot b_n(t)$ as $d(t)^2\cdot s(t)$ where $s(t)$ is separable, then $b_n(t)\mid s(t)$ by Proposition \ref{separability} again. The curve $C_I$ admits a non-constant map to the curve $y^2=s(t)$. This is a smooth curve because $s(t)$ is separable, and as long as $n\geq 4$ we have that $\deg s(t)\geq 6$. Therefore, its genus is at least 2, and it has finitely many rational points by Faltings' theorem. In turn, this implies that $C_I$ has finitely many rational points.
 
 The above argument shows that if $C_I$ has infinitely many rational points, then $n\leq 3$. Now it is just a matter of checking finitely many curves. A brief computation with Magma \cite{magma} shows that there are exactly six curves with genus at most 1, and exactly five of them have infinitely many rational points. Specifically, they correspond to the following subsets: $\{1\}$, $\{1,2\}$, $\{2\}$, $\{1,3\}$, and $\{2,3\}$ (the sixth one is the elliptic curve $y^2=\mathfrak c_3(t)$, which has rank 0).
\end{proof}
  
  From now on, we will denote by $\mathcal G_i$ the subgroup corresponding to the vector $v_i$ of Proposition \ref{subgroups}, for $i\in\{1,\ldots,5\}$. We will show in Section \ref{proofC} that the $\mathcal G_i$'s are pairwise non-isomorphic as topological groups.
  
  We remark that by searching for rational points of small height with Magma \cite{magma} we could not find any other curve with rational points with $t$-coordinate different from $0,-1,-2$. These values of $t_0$ yield post-critically finite polynomials, whose arboreal representation has infinite index in $\Omega_{\infty}$ by Corollary \ref{PCF_polynomials}. However, since running an extensive search for rational points on curves of the form $y^2=\prod_{i\in I}\mathfrak c_i(t)$ is beyond the scope of this paper, we do not want to conjecture that the five groups of Proposition \ref{subgroups} are the only ones that can possibly appear.
  
  Notice that among the five curves listed at the end of the proof of Proposition \ref{subgroups}, the only ones with non-trivial integral points are $C_{\{1\}}$ and $C_{\{1,2\}}$. In fact, a conjecture of Hindes \cite[Conjecture 1.5]{hindes2} implies, together with Theorems \ref{Classification of realizing stable Ma} and \ref{nonstable}, that when restricting to integral specializations, the only index two subgroups that can appear as images of the arboreal representation are $\mathcal G_1$ and $\mathcal G_5$ (see \cite{hindes1},\cite{hindes2} for more on this topic). We will show that in fact there exist infinitely many integral specializations of $\phi$ that yield $\mathcal G_1$, and that the same happens for $\mathcal G_5$ under Vojta's conjecture for $\Q$ (cf.\ Proposition \ref{group1} and Remark \ref{group5r}).
 
 \subsection{Explicit families for index two}\label{stoll_trick}
 In this section, we will show that if $i\in\{1,2\}$, there exist infinitely many $t_0\in\Q$ such that $\im(\rho_{\phi_{t_0}})=\mathcal G_i$.
 
 In \cite{Stoll} the author proved a conjecture of Cremona \cite{cremona}, by constructing an infinite family of polynomials of the form $x^2+a\in\Z[x]$ having surjective representation. His proof makes use of the following idea, that we will describe in a more general form.
 
 Let $a\in \Q$ and $f=x^2+a\in \Q[x]$. Let $\{\mathfrak c_n\}_{n\in\N}$ be the adjusted post-critical orbit of $f$. Now define
 $$b_n\coloneqq \prod_{d\mid n}\mathfrak c_d^{\mu(n/d)},$$
 where $\mu$ is the M\"obius function. Since by M\"obius inversion one has $\displaystyle \mathfrak c_n=\prod_{d\mid n}b_d$, it follows that:
 \begin{equation}\label{dimensions}
  \dim \langle \mathfrak c_1,\ldots,\mathfrak c_n\rangle_\Q=\dim \langle b_1,\ldots,b_n\rangle_\Q.
 \end{equation}
 
 The key observation is now the following: if $p$ is a prime such that $v_p(\mathfrak c_i)>0$ for some $i\in \N$, and $i$ is the minimal index with this property, then $v_p(b_i)>0$ and $v_p(b_j)=0$ for every $j\neq i$ (cf.\ \cite[Lemma 1.1]{Stoll}). It follows that if $S\subseteq \Z$ is a set of primes, $\Z_S$ is the localization of $\Z$ with respect to the multiplicative system generated by $S$ and $a\in \Z_S$, then the $b_i$'s are relative coprime $S$-integers.
 
 From now on, we will denote by $S$ the set of primes $p\in \Z$ such that $v_p(\mathfrak c_1)=v_p(-a)<0$. Notice that $\mathfrak c_i\in \Z_{S}$ for every $i$ and consequently $b_i\in \Z_{S}$ for every $i$. Write $b_i=\overline{b}_i/d_i$ where $\overline{b}_i\in\Z$, $d_i\in \N$ and $\gcd (\overline{b}_i,d_i)=1$. Then we have that:
 \begin{equation}\label{dimensions_2}
  \dim \langle\overline{b}_1,\ldots,\overline{b}_n\rangle_{\Q}\leq \dim \langle b_1,\ldots,b_n\rangle_{\Q}.
 \end{equation}
 Notice that if for every $p\in S$ one has that $v_p(\mathfrak c_1)\equiv 0 \bmod 2$, then obviously $\dim \langle\overline{b}_1,\ldots,\overline{b}_n\rangle_{\Q}= \dim \langle b_1,\ldots,b_n\rangle_{\Q}$.
 
 Let now $g\coloneqq |a|x^2+\text{sgn}(a)\in \Z_{S}[x]$. Let $\gamma_1\coloneqq 1$ and $\gamma_{n+1}=g(\gamma_n)$ for every $n\geq 1$. Define, as above, $\beta_n\coloneqq \prod_{d\mid n}\gamma_d^{\mu(n/d)}$. Finally, for every $n$ write $\gamma_n=\overline{\gamma}_n/r_n$ and $\beta_n=\overline{\beta}_n/e_n$ where $\overline{\gamma}_n,\overline{\beta}_n\in\Z$, $r_n,e_n\in\N$ and $\gcd(\overline{\gamma}_n,r_n)=\gcd(\overline{\beta}_n,e_n)=1$. The next lemma allows us to establish a criterion to produce polynomials with arboreal representations of index two. The proof follows closely the ideas of \cite[Lemma 2.1]{Stoll} and \cite[Lemma 3.3]{li}. However, our result is slightly more general, as it allows to work with quadratic polynomials in $\Q[x]$ rather than only in $\Z[x]$. This is of crucial importance to construct the family of examples of Proposition \ref{second_class}.
 
 \begin{lemma}\label{trick}
  Assume that $\overline{\gamma}_2\in \Z$ is a square. Suppose that for every $n\geq 2$ there exists $m_n\in \Z_{S}$ such that $m_n\mid \gamma_n+\gamma_{n+1}$, $\gcd (m_n,\gamma_n)=1$ and $-1$ is not a square modulo $m_n$. Then $\overline{\beta}_n$ is not a square in $\Z$ for every $n\geq 3$.
 \end{lemma}
 \begin{proof}
  	Let $n\geq 3$, let $n'$ be the radical of $n$ and set $k\coloneqq n/n'$. Notice that an easy induction proves the following claim:
\begin{equation}\label{induction}
\mbox{if $u\in \Z_S$ is such that $u \mid \gamma_k+\gamma_{k+1}$, then $u\mid \gamma_k+\gamma_{k+\ell}$ for every $\ell\geq 1$}.
\end{equation}  	
  	 
  	First, let $k>1$ and let $m_k\mid \gamma_k+\gamma_{k+1}$ be as in the hypotheses; by \eqref{induction} we have that $\gamma_k\equiv -\gamma_{2k}\bmod m_k$, while one sees immediately that $\gamma_{\ell k}\equiv \gamma_{2k}\bmod m_k$ for every $\ell\geq 2$. Notice that $\gamma_{2k}\not\equiv 0\bmod m_k$. Then we have that:
  	$$\beta_n=\prod_{d\mid n}\gamma_{d}^{\mu(n/d)}=\prod_{\ell\mid n'}\gamma_{k\ell}^{\mu(n'/\ell)}\equiv (-1)^{\mu(n')}\prod_{\ell\mid n'}(\gamma_{2k})^{\mu(n'/\ell)}\equiv -1\bmod m_k.$$
  	Since $-1$ is not a square modulo $m_k$, then $\beta_n$ cannot be a square in $\Z_{S}$. Moreover, for every $p\in S$ one has $v_p(\beta_n)\equiv 0 \bmod 2$, because $n'$ has an even number of divisors and $v_p(\gamma_\ell)\equiv v_p(\gamma_2)\bmod 2$ for every $\ell\geq 2$. Therefore $\beta_n=\overline{\beta}_n/u^2$ for some $u\in \Z$, proving that $\overline{\beta}_n$ is not a square in $\Z$.
  	
  	If $k=1$ (i.e.\ $n$ is squarefree), we need to treat separately even $n$'s and odd $n$'s. Let $r$ be the number of distinct prime factors of $n$.

	Let $n$ be odd. Let $m_2$ be a divisor of $\gamma_2+\gamma_3$ as in the hypotheses of the lemma. Since $m_2\mid \gamma_2+\gamma_\ell$ for all $\ell\geq 3$, we get that:
	\[
		\beta_n = \prod_{\substack{d\mid n\\d\ne 1}} \gamma_d^{\mu(n/d)} \equiv \prod_{\substack{d\mid n\\d\ne 1}} (-\gamma_2)^{\mu(n/d)} \equiv (-1)^{2^r-1} \prod_{\substack{d\mid n\\d\ne 1}} (\gamma_2)^{\mu(n/d)} \equiv   -\gamma_2^{\mu(n)}   \bmod m_2.
	\]
	For every $p\in S$ one has that $v_p(\beta_n)\equiv v_p(\gamma_2)\bmod 2$ (notice also that for such $p$'s one has $v_p(\gamma_2)=v_p(a)$). Thus, if $\pi$ is the product of all primes in $S$ with $v_p(\beta_n)\equiv 1 \bmod 2$, we see that $\pi\beta_n\equiv -\pi\gamma_2^{\mu(n)}\bmod m_2$. But now $\pi\gamma_2^{\mu(n)}$ is clearly a rational square because $\overline{\gamma}_2$ is a square in $\Z$ by assumption, while $\pi\beta_n=\frac{\overline{\beta}_n}{d^2}$ for some $d\in \N$ with $(\overline{\beta}_n,d)=1$. It follows that $\overline{\beta}_n$ is congruent to minus a square modulo $m_2$; hence it cannot be a square in $\Z$.
	
	Finally, let $n\ne 2$ be even, let $p$ be the smallest odd prime dividing $n$ and $m_p$ a divisor of $\gamma_p+\gamma_{p+1}$ as in the hypotheses of the lemma. Again, we have that $m_p\mid \gamma_p+\gamma_q$, where $q$ is any divisor of $n$ different from $1,2$ and $p$. We then have that:
	$$
		\beta_n =  \prod_{\substack{d\mid n\\d\ne 1}} \gamma_d^{\mu(n/d)} \equiv   \gamma_2\gamma_p\prod_{\substack{d\mid n\\d\neq 1,2,p}} (-\gamma_p)^{\mu(n/d)}\equiv (-1)^{2^r-3}\gamma_2(\gamma_p)^{-(2^r-2)} \bmod m_p,
	$$
	and the same argument of the odd case applies, proving that $\overline{\beta}_n$ is not a square.
 \end{proof}

 The above lemma is useful in the following sense: since $\gamma_n\cdot |a|=\mathfrak c_n$ for every $n\geq 2$, we have that $|b_n|=\beta_n$, and thus also $|\overline{b}_n|=\overline{\beta}_n$, for every $n\geq 2$. If the hypotheses of the lemma are satisfied for some $a\in \Q$, we clearly get that $\dim \langle \overline{\beta}_3,\ldots,\overline{\beta}_n\rangle_{\Q}=n-2$ and hence:
 $$\dim \langle \overline{b}_3,\ldots,\overline{b}_n\rangle_{\Q}=\dim \langle b_3,\ldots,b_n\rangle_{\Q}=n-2.$$
 This is crucial to prove that the conditions of Theorem \ref{Classification of realizing stable Ma} are satisfied.

\subsubsection{Polynomials with \texorpdfstring{$\im(\rho_f)=\mathcal G_1$}{}}

The specializations of $\phi$ whose arboreal representation has image $\mathcal G_1$ must be of the form $x^2-(u^2+1)$, with $u\in \Q$, as one easily sees by parametrizing rational points on the curve $y^2=\mathfrak c_1(t)\mathfrak c_2(t)$.

\begin{proposition}\label{group1}
	Let $u\in 2\Z\setminus\{0\}$ and let $f = x^2-(u^2+1)$. Then $\im(\rho_f)=\mathcal G_1$.
\end{proposition}

\begin{proof}
	By Theorem \ref{Classification of realizing stable Ma}, we need to verify that $\dim \langle \mathfrak c_1,\ldots,\mathfrak c_n\rangle_{\Q}=n-1$ for every $n\geq 2$. Since $\mathfrak c_i>0$ for every $i$, $b_2=u^2$ and $b_1=1+u^2\notin (\Q^{\times})^2$, by \eqref{dimensions} it is enough to show that $b_n\notin (\Q^{\times})^2$ for every $n\geq 3$. To do this, we can simply show that the hypotheses of Lemma \ref{trick} are satisfied. Here we have $g = (u^2+1)x^2-1$, so of course $(\gamma_n, \gamma_{n+1})=1$ for all $n$. It follows that any $m$ dividing $\gamma_n+\gamma_{n+1}$ is coprime to $\gamma_n$. Notice that $\gamma_{2n}\equiv 0\pmod{4}$ while $\gamma_{2n+1} \equiv -1 \bmod 4$, for all $n \geq 1$. This shows that for every fixed $n\geq 2$ we can just choose $m_n\coloneqq\gamma_n+\gamma_{n+1}\equiv -1 \bmod 4$.
\end{proof}

It follows immediately from the above proposition that if $\psi=x^2-(1+t^2)\in \Q(t)[x]$, then $\im(\rho_{\psi})=\mathcal G_1$.

\subsubsection{Polynomials with \texorpdfstring{$\im(\rho_f)=\mathcal G_2$}{}}

The specializations of $\phi$ whose arboreal representation has image $\mathcal G_2$ must be of the form $\displaystyle x^2+\frac{1}{u^2-1}$, where $u\in \Q\setminus\{\pm 1\}$ and $1/\mathfrak c_1=1-u^2\notin (\Q^{\times})^2$.

\begin{proposition}\label{second_class}
 Let $u\in 2\Z\setminus\{0\}$ and let $\displaystyle f=x^2+\frac{1}{u^2-1}$. Then $\im(\rho_f)=\mathcal G_2$.
\end{proposition}

\begin{proof}
 Since $\displaystyle \mathfrak c_2=\frac{u^2}{(u^2-1)^2}\in (\Q^{\times})^2$, Theorem \ref{Classification of realizing stable Ma} shows that we need to verify that, for every $n\geq 2$, one has $\dim \langle \mathfrak c_1,\widetilde{\mathfrak c}_2,\mathfrak c_3,\ldots,\mathfrak c_n\rangle_{\Q}=n$. Here $\displaystyle\widetilde{\mathfrak c}_2=2(\mathfrak c_1+\sqrt{\mathfrak c_2})=\frac{2}{u-1}$ by Proposition \ref{The abelian "additional" extension}. Let us show that the hypotheses of Lemma \ref{trick} apply: here $S$ is the set of the prime divisors of $u^2-1$ and we have that $\displaystyle g=\frac{1}{u^2-1}x^2+1\in \Z_S[x]$, so $\displaystyle\gamma_2=\frac{u^2}{u^2-1}$ and thus clearly $\overline{\gamma}_2=u^2\in (\Q^{\times})^2$. Moreover, $\overline{\gamma}_{2n}\equiv 0 \bmod 4$ and $\overline{\gamma}_{2n+1}\equiv -1 \bmod 4$ for every $n\geq 1$, and since $u^2-1\equiv -1\bmod 4$, one verifies that the numerator $m_n$ of $\gamma_{n}+\gamma_{n+1}$ is $\equiv -1 \bmod 4$ for $n\geq 2$. It follows from Lemma \ref{trick} that $|\overline{b}_i|\notin (\Q^{\times})^2$ for every $n\geq 3$, and thus $\dim \langle |\overline{b}_3|,\ldots,|\overline{b}_n|\rangle_{\Q}=n-2$. Since $\overline{b}_1=-1$ and none of the $\overline{b}_i$'s, with $i\geq 3$, belongs to $-(\Q^{\times})^2$, while $\overline{b_2}=-u^2$ then $\dim \langle \overline{b_1},\overline{b}_2,\overline{b}_3,\ldots,\overline{b}_n\rangle_{\Q}=n-1$. Thus, $\dim \langle \mathfrak c_1,\ldots,\mathfrak c_n\rangle_{\Q}=n-1$ by \eqref{dimensions} and \eqref{dimensions_2}. It remains to show that for every $n\geq 1$ we have that $\widetilde{\mathfrak c}_2\notin \langle \mathfrak c_1,\ldots,\mathfrak c_n\rangle_{\Q}$. But this is easy to check, because once we multiply $\widetilde{\mathfrak c}_2$ and each $\mathfrak c_i$ by the square of its denominator, obtaining quantities that we denote by $\widetilde{\mathfrak c}_2'$ and $\mathfrak c_i'$, respectively, we see easily that $\mathfrak c_{2i}'\equiv 0\bmod 4$ while $\mathfrak c_{2i+1}'\equiv 1\bmod 4$ for every $i\geq 0$. On the other hand, $\widetilde{\mathfrak c}_2'\equiv 2 \bmod 4$, so it cannot belong to the space generated by the $\mathfrak c_i$'s.
\end{proof}

Again, it follows immediately from the above proposition that if $\displaystyle \psi=x^2+\frac{1}{t^2-1}\in \Q(t)[x]$, then $\im(\rho_{\psi})=\mathcal G_2$.

\subsection{Vojta's conjecture, primitive divisors, and index two specializations}\label{sub:vojta}
The goal of this section is to show that if $i\in\{3,4,5\}$ and Vojta's conjecture over $\Q$ holds true, then there exist infinitely many $t_0\in\Q$ such that $\im(\rho_{\phi_{t_0}})=\mathcal G_i$. To do this, we will borrow some ideas from \cite{hindes1}, and combine them with our Theorem \ref{Classification of realizing stable Ma}. From now on, we denote by $h\colon \overline{\Q}\to \R_{\geq 0}$ the absolute logarithmic height (cf.\ \cite[VIII.5]{silverman1}).

Let us start by recalling the following conjecture, which in degree at least $5$ is a consequence of Vojta's conjecture (see for example \cite{ih} and \cite[Conjecture 4]{stoll2}).
\begin{conjecture}\label{vojta}
For all $d \geq 5$ there exist constants $C_1=C_1(d)$ and $C_2=C_2(d)$ so that for all
$f\in\Q[x]$ of degree $d$ with non-zero discriminant, if $x_0,y_0\in\Q$ satisfy $y_0^2=f(x_0)$, then
$$h(x_0)\leq C_1\cdot h(f)+C_2,$$
where $h(f)$ is the maximum among the logarithmic heights of the coefficients of $f$.
\end{conjecture}

Let $\gamma(t),\mathfrak c(t),r(t)\in \Z[t]$ be such that $\deg(\gamma(t)-\mathfrak c(t))\neq 0$. We define $\psi\coloneqq (x-\gamma(t))^2+\mathfrak c(t)$ and $g\coloneqq x+r(t)$. For $a\in \Q$, we will denote by $\psi_a$ and $g_a$ the specialized polynomials in $\Q[x]$. It is shown in \cite{hindes1} that there exists a computable, positive constant $B_{1,\psi}$ such that:
$$\deg(\mathfrak c(t)-\gamma(t))\cdot h(a)-B_{1,\psi}\leq h(\mathfrak c(a)-\gamma(a)) \quad \mbox{ for all }a \in \overline{\Q}.$$

\begin{lemma}{{\cite[Lemma 1.1]{hindes1}}}\label{primitive_divisors}
 Assume Vojta's conjecture over $\Q$. Then there exists $n_\psi>0$ such that for every $a\in \Q$ satisfying the following properties:
 \begin{enumerate}
  \item $g_a\circ\psi_a(\gamma(a))\cdot g_a\circ\psi_a^{(2)}(\gamma(a))\neq 0$,
  \item $\gamma(a)$ is not preperiodic for $\psi_a$,
  \item $\deg(\mathfrak c(t)-\gamma(t))\cdot h(a)-B_{1,\psi}>0$
 \end{enumerate}
 and every $n\geq n_{\psi}$, there exists an odd prime $p_n$ such that:
 $$v_{p_n}(g_a\circ\psi_a^{(n)}(\gamma(a)))\not\equiv 0 \bmod 2 \mbox{ and } v_{p_n}(g_a\circ\psi_a^{(j)}(\gamma(a)))=0 \mbox{ for all } 1\leq j\leq n-1,$$
 where $v_p$ is the usual $p$-adic valuation.
\end{lemma}

\begin{proof}
 The proof is essentially that of \cite{hindes1}. For every $n$, one writes $g_a\circ\psi_a^{(n-1)}(\gamma(a))=2^{e_n}\cdot d_n\cdot y_n^2$, where $e_n\in \{0,1\}$ and $d_n$ is an odd, squarefree integer. For $n\geq 4$, let
 $$C_{a}^{(d_n)}\colon Y^2=2^{e_n}\cdot d_n\cdot (X-\mathfrak c(a))(g_a\circ\psi_a^{(2)}(X)).$$
 One checks easily that assumption (1) implies that $C_a^{(d_n)}$ is smooth (for example by using \cite[Lemma 2.6]{jones3}), and that $(\psi_a^{n-3}(\gamma(a)),2^{e_n}\cdot d_n\cdot y_n\cdot(\psi_a^{(n-4)}(\gamma(a))-\gamma(a)))$ is a rational point of $C_a^{(d_n)}$; Conjecture \ref{vojta} implies the existence of absolute constants $\kappa_1,\kappa_2,\kappa_3$, such that:
 $$h(\psi_a^{n-3}(\gamma(a)))\leq \kappa_1\cdot h(d_n)+\kappa_2\cdot h(a)+\kappa_3.$$

 From this point on, the arguments of \cite{hindes1} apply verbatim, except for replacing $n-1$ with $n-3$, which just yields a weaker bound on $n$ at the end of the proof.
\end{proof}
\begin{remark}\label{finiteness}
 Notice that the set of $a\in\Q$ not fulfilling the three conditions of Lemma \ref{primitive_divisors} is finite. In fact, first it is clear that every $a$ with large enough height satisfies (1) and (3). Moreover, it is a well-known fact that there exist only finitely many post-critically finite polynomials of the form $x^2+u$, where $u\in \Z$ (see for example \cite{ingram}). On the other hand if $\gamma(a)$ is preperiodic for $\psi_a$ then 0 is preperiodic for $x^2+\mathfrak c(a)-\gamma(a)$. Since by hypothesis $\deg(\mathfrak c(t)-\gamma(t))>0$, clearly there exist only finitely many $a$'s such that $\gamma(a)$ is preperiodic for $\psi_a$.
\end{remark}

We remark that in \cite{hindes1}, Lemma \ref{primitive_divisors} is stated with $r=0$ and $a\in \Z$.

Primes as the ones appearing in Lemma \ref{primitive_divisors} are called \emph{primitive prime divisors}.

Recall that a subset $E\subseteq \mathbb P^1(\Q)$ is called \emph{thin} if it is contained in a finite union of finite sets and sets of the form $\pi(C(\Q))$, where $C/\Q$ is an irreducible algebraic curve and $\pi\colon C\to \mathbb P^1$ is a morphism of degree $\geq 2$. Moreover, \cite[Proposition 3.4.2]{serre} shows that the complement of a thin set is infinite.

\begin{theorem}
 Assume Vojta's conjecture over $\Q$, and let $i\in \{3,4\}$. Then the set of $t_0\in \Q$ such that $\im(\rho_{\phi_{t_0}})=\mathcal G_i$ is infinite, but thin.
\end{theorem}
\begin{proof}
 As usual, let $\mathfrak c_1(t)=-t$ and $\mathfrak c_n(t)=\mathfrak c_{n-1}(t)^2+\mathfrak c_1(t)$ for every $n\geq 2$, let $b_n(t)=\prod_{d\mid n}\mathfrak c_d(t)^{\mu(n/d)}\in\Q[t]$ and let $n=n_{\phi}$ be the positive integer determined by Lemma \ref{primitive_divisors} (with $g=x$ and $\psi=\phi$).
 
 First let $i=3$. Let $C_{\{1,3\}}\colon y^2=b_3(t)$. One can check with Magma that $C_{\{1,3\}}$ is an elliptic curve of rank 1. Let $T_{\{1,3\}}\subseteq \Q$ be the infinite set of $t$-coordinates of rational points on $C_{\{1,3\}}$. Now we claim that for all $t_0\in T_{\{1,3\}}$, except at most for finitely many, we have that
 $$\dim\langle \mathfrak c_1(t_0),\ldots,\mathfrak c_n(t_0)\rangle_{\Q}=n-1.$$
 Notice that if this holds, then Lemma \ref{primitive_divisors} and Remark \ref{finiteness} imply that for all $t_0\in T_{\{1,3\}}$ but finitely many we have that $\dim\langle \mathfrak c_1(t_0),\ldots,\mathfrak c_m(t_0)\rangle_{\Q}=m-1$ for every $m\geq n$, and thus $\im(\rho_{\phi_{t_0}})=\mathcal G_3$ by Theorem \ref{Classification of realizing stable Ma}.
 
 Values $t_0\in T_{\{1,3\}}$ such that $\dim\langle \mathfrak c_1(t_0),\ldots,\mathfrak c_n(t_0)\rangle_{\Q(t)}<n-1$ must satisfy, by Corollary \ref{big_cor}, the following relations:
 $$\begin{cases}
    y_1^2=b_3(t_0) & \\
    y_2^2=\prod_{i\in I}\mathfrak c_i(t_0) & 
   \end{cases}
$$
  for some non-empty subset $I\subseteq \{1,\ldots,n\}$ such that $I\neq \{1,3\}$ and some $y_1,y_2\in\Q$. Thus, they have to be $t$-coordinates of rational points on the curve $y^2=b_3(t)\cdot\prod_{i\in I}\mathfrak c_i(t)$. The proof of Proposition \ref{subgroups} shows that as long as $\max\{i\in I\}\geq 4$, such curve has a finite number of points. It remains to check by hand the other cases. If $I\in\{\{1\},\{2\},\{3\}\}$, then $t_0$ must be the $t$-coordinate of a rational point on one of the curves $y^2=\mathfrak c_3(t)$ or $y^2=\mathfrak c_1(t)\mathfrak c_2(t)\mathfrak c_3(t)$, both of which have finitely many rational points (the first one is an elliptic curve of rank 0, the second one has genus 2). The cases $I=\{1,2\}$ or $\{2,3\}$ both imply, up to multiplying the two equations, factoring out squares and swapping $y_1$ and $y_2$ if necessary, that $t_0$ is the $t$-coordinate of a rational point on the curve:
  $$\begin{cases}
    y_1^2=b_2(t) & \\
    y_2^2=b_3(t) & 
   \end{cases}.
$$
  It is easy to check that the projection of such curve on the $(y_1,y_2)$-plane is the curve $y_2^2=y_1^6+y_1^4-1$, which is smooth and has genus 2, and consequently only finitely many rational points.
  
  Now let $i=4$. Let $C_{\{2,3\}}\colon y^2=b_2(t)b_3(t)$. Again, one can check with Magma that this is an elliptic curve of rank 1. Let $T_{\{2,3\}}\subseteq \Q$ be the infinite set of $t$-coordinates of rational points on $C_{\{2,3\}}$. The strategy is the same we used for $i=3$: one wants to prove that for all $t_0\in T_{\{2,3\}}$ except at most for finitely many we have that $\dim\langle \mathfrak c_1(t_0),\mathfrak c_2(t_0),\widetilde{\mathfrak c}_3(t_0),\mathfrak c_4(t_0)\ldots,\mathfrak c_n(t_0)\rangle_{\Q}=n$, and then use Theorem \ref{Classification of realizing stable Ma} and Lemma \ref{primitive_divisors}. Recall that by Example \ref{ctilda} here we have $\widetilde{\mathfrak c}_3(t_0)=2(\mathfrak c_1(t_0)+\mathfrak c_1(t_0)\mathfrak c_2(t_0)+\sqrt{\mathfrak c_2(t_0)\mathfrak c_3(t_0)})$. One needs to pay extra attention to the following fact: Lemma \ref{primitive_divisors} only ensures that for $m\geq n$ there exists an odd prime $p_m$ such that $v_{p_m}(\mathfrak c_m(t_0))\equiv 1 \bmod 2$ and $v_{p_m}(\mathfrak c_i(t_0))=0$ for every $i<m$, but a priori it is not clear that $v_{p_m}(\widetilde{\mathfrak c}_3(t_0))=0$. However, this is easily settled: first set $\widetilde{\mathfrak c}_3'(t_0)=2(\mathfrak c_1(t_0)+\mathfrak c_1(t_0)\mathfrak c_2(t_0)-\sqrt{\mathfrak c_2(t_0)\mathfrak c_3(t_0)})$. The key observation is that:

\begin{equation}\label{valuation_5}
  \widetilde{\mathfrak c_3}(t_0)\cdot \widetilde{\mathfrak c_3}'(t_0)=4\mathfrak c_1(t_0)^5=-4t_0^5 \mbox{ for every } t_0\in T_{\{2,3\}}.
  \end{equation}

  This implies that if $p$ is an odd prime then:
  
  \begin{equation}\label{negative_valuation}
   v_p(\widetilde{\mathfrak c}_3(t_0))<0 \Longleftrightarrow v_p(\widetilde{\mathfrak c}_3'(t_0))<0 \Longleftrightarrow v_p(\mathfrak c_1(t_0))<0
   \end{equation}
  and
  \begin{equation}\label{positive_valuation}
   v_p(\widetilde{\mathfrak c}_3(t_0))>0 \Longleftrightarrow v_p(\widetilde{\mathfrak c}_3'(t_0))>0 \Longleftrightarrow v_p(\mathfrak c_1(t_0))>0.
   \end{equation}
  To see why \eqref{negative_valuation} holds, suppose first that $v_p(\widetilde{\mathfrak c_3}(t_0))<0$. Then necessarily $v_p(\mathfrak c_1(t_0))<0$, and hence $v_p(\mathfrak c_n(t_0))=2^{n-1}v_p(\mathfrak c_1(t_0))$ for every $n\geq 1$. Therefore 
  $$v_p(\widetilde{\mathfrak c}_3(t_0)),v_p(\widetilde{\mathfrak c}'_3(t_0))\geq 3v_p(\mathfrak c_1(t_0)).$$
  By \eqref{valuation_5} it follows immediately that $v_p(\widetilde{\mathfrak c}_3'(t_0))<0$. Of course, symmetrically the converse implication holds. The second double implication of \eqref{negative_valuation} follows directly from the first double implication together with \eqref{valuation_5}.
  
  To prove \eqref{positive_valuation}, suppose $v_p(\widetilde{\mathfrak c}_3(t_0))>0$. By \eqref{negative_valuation}, it must be $v_p(\widetilde{\mathfrak c}'_3(t_0))\geq 0$. But then \eqref{valuation_5} implies that $v_p(\mathfrak c_1(t_0)>0$, and the definition of $\widetilde{\mathfrak c}'_3(t_0)$ shows immediately that $v_p(\widetilde{\mathfrak c}'_3(t_0))>0$. Again, the converse implication holds by symmetry and the second double implication follows from the first one together with \eqref{valuation_5}.

  Relations \eqref{negative_valuation} and \eqref{positive_valuation} immediately prove that if $m\geq n$ and $p_m$ is a primitive prime divisor of $\mathfrak c_m(t_0)$, then $v_{p_m}(\widetilde{\mathfrak c_3}(t_0))=0$.
  
  That said, it remains to prove that for all $t_0\in T_{\{2,3\}}$ except at most for finitely many we have that $\dim\langle \mathfrak c_1(t_0),\mathfrak c_2(t_0),\widetilde{\mathfrak c}_3(t_0),\mathfrak c_4(t_0)\ldots,\mathfrak c_n(t_0)\rangle_{\Q}=n$. The same computations we did for $i=3$ show that there for all $t_0\in T_{\{2,3\}}$ but at most finitely many we have $\dim\langle \mathfrak c_1(t_0),\mathfrak c_2(t_0),\mathfrak c_4(t_0)\ldots,\mathfrak c_n(t_0)\rangle_{\Q}=n-1$. 
  To conclude the proof, it is enough to show that for every finite (possibly empty) subset $I$ of $\{1,2,4,\ldots,n\}$ the curves (one for each choice of the sign in front of $y_1$) defined by:
 $$\begin{cases}
    y_1^2=b_2(t)b_3(t) & \\
    y_2^2=\prod_{i\in I}\mathfrak c_i(t)\cdot 2(\mathfrak c_1(t)+\mathfrak c_1(t)\mathfrak c_2(t)\pm ty_1) & 
   \end{cases}
$$
 has finitely many rational points. By using M\"obius inversion on $\prod_{i\in I}\mathfrak c_i(t)$, factoring out squares and noticing that $\mathfrak c_1(t)+\mathfrak c_1(t)\mathfrak c_2(t)=-t-t^2-t^3$, we can reduce to the curves:
 $$\begin{cases}
    y_1^2=b_2(t)b_3(t) & \\
    y_2^2=2B(t)(t^2+t+1\pm y_1) & 
   \end{cases},
$$
 where $B(t)=\prod_{i\in I'}b_i(t)$ for some $I'\subseteq \{1,2,4,\ldots,n\}$.
 
  Isolating $y_1$ in the second equation, we see that the projection of the above curves on the $(t,y_2)$-plane is the curve $C$ given by $h(t,y_2)=0$, where:
 $$h(t,y_2)\coloneqq (y_2^2+2(t+t^2+t^3)B(t))^2-4b_2(t)b_3(t)B(t)^2.$$
An easy computation shows that $h(t,y_2)$ is irreducible as a polynomial in the single variable $y_2$: in general if $K$ is a field of characteristic not 2 and $f,g\in K[x]$ are monic and quadratic then $f\circ g$ is irreducible if $f$ is irreducible and $f(g(\gamma))$ is not a square in $K$, for $\gamma$ the critical point of $g$. It follows from Gauss' lemma that $C$ is an irreducible curve.

 To conclude the proof of the theorem, it is enough to show that $C$ has finitely many rational points. To prove it, let $E$ be the rational elliptic curve $y^2=b_2(t)b_3(t)$ (as we already mentioned in the proof of Proposition \ref{subgroups}, $E$ has rank $1$). There is an obvious rational map $\pi\colon C\to E$ of degree $2$, defined over $\Q$.  Let $\widetilde{C}$ be the normalization of $C$. This is a smooth projective curve defined over $\Q$ with the same geometric genus as $C$, and comes equipped with a rational map to the projective closure of $C$ that is an isomorphism outside the singular locus of $C$ (which is a finite set). Hence, it is enough to prove that $\widetilde{C}$ has finitely many rational points. The normalization map, composed with $\pi$, yields a rational map of degree $2$, again defined over $\Q$, $\widetilde{\pi}\colon \widetilde{C}\to E$, and since both curves are smooth this is in fact a morphism. Now, clearly the geometric genus of $\widetilde{C}$ must be at least 1. If it is $\geq 2$, we are done by Faltings' theorem. Assume it is exactly 1. If $\widetilde{C}$ has no rational points there is nothing to prove. If on the other hand $\widetilde{C}$ has a rational point, then it is an elliptic cuve, and the map $\widetilde{\pi}$ is a degree $2$ morphism of elliptic curves. It is well-known that such a map must be the composition of a rational isogeny $\widetilde{C}\to E$ of degree $2$ with a translation. But if there exists a rational isogeny $\widetilde{C}\to E$ of degree $2$, then there exists a rational dual isogeny $E\to \widetilde{C}$ of degree $2$. However, one can check that $E[2]$ is an irreducible $G_\Q$-module, and thus such an isogeny cannot exist.
 \end{proof}

Finally, we shall consider $\mathcal G_5$. We first need the following preliminary lemma.

\begin{lemma}\label{inter}
 Let $u\in \Q^{\times}$ and $f=x^2-u^2$. Let $\{\mathfrak c_i\}_{i\in\N}$ be the adjusted post-critical orbit of $f$. Let $a\in \Z$ be a squarefree integer such that $a\in \langle \mathfrak c_1-u,\mathfrak c_2+u,\ldots,\mathfrak c_n+u\rangle_\Q\cap \langle \mathfrak c_1+u,\mathfrak c_2-u,\ldots,\mathfrak c_n-u\rangle_\Q$ and let $p$ be a prime such that $p\mid a$. Then $v_p(2\mathfrak c_1)\neq 0$.
\end{lemma}
\begin{proof}
 Let $\mathfrak c_1^+\coloneqq \mathfrak c_1-u$ and $\mathfrak c_i^+\coloneqq \mathfrak c_i+u$ for $i\in\{2,\ldots,n\}$ and symmetrically let $\mathfrak c_1^-\coloneqq \mathfrak c_1+u$ and $\mathfrak c_j^-\coloneqq \mathfrak c_j-u$ for $j\in \{2,\ldots,n\}$. The construction of $p$ implies the existence of indexes $i,j$ such that $v_p(\mathfrak c_i^+),v_p(\mathfrak c_j^-)\neq 0$. Notice that, for any $k\in \{1,\ldots,n\}$, if $q$ is a prime then $v_q(\mathfrak c_k^{\pm})<0$ if and only if $v_q(u)<0$, so that in particular $v_q(\mathfrak c_1)\neq 0$. Therefore, if $a,p$ are as in the statement, we can assume that $v_p(u)\geq 0$ and $v_p(\mathfrak c_i^+),v_p(\mathfrak c_j^-)>0$ for some $i,j\in \{1,\ldots,n\}$.
 
 Thus in order to prove the lemma it is enough to prove the following more general claim:
 
 \begin{center}
   if there exists a prime $q$ and indices $i,j$ such that $v_q(\mathfrak c_i^+),v_q(\mathfrak c_j^-)>0$, then $v_q(2\mathfrak c_1)>0$.
 \end{center}
 Assume without loss of generality that $i\leq j$ and $i,j$ are minimal with the above property. Notice that for every $k\geq 1$ we have that $\mathfrak c_k^+\cdot \mathfrak c_k^-=\mathfrak c_{k+1}$. It follows that $v_q(\mathfrak c_{i+1}),v_q(\mathfrak c_{j+1})>0$. By the minimality of $i,j$ we have that $v_q(\mathfrak c_k)=0$ for every $k<i+1$ and thus $i+1\mid j+1$ (cf.\ \cite[Lemma 1.1]{Stoll}). Let $k\geq 1$ be such that $j+1=k(i+1)$, and rewrite the relation as $j=i+(k-1)(i+1)$. Now it is enough to show that for every $\ell\geq 0$ one has $v_q(\mathfrak c_{i+\ell(i+1)}^+)>0$, because this implies that $v_q(\mathfrak c_j^+)>0$, and thus $v_q(\mathfrak c_j^--\mathfrak c_j^+)=v_q(2u)>0$, so in particular $v_q(2\mathfrak c_1)>0$. One proves this by an easy induction, having a little extra care for the case $i=1$. For $\ell=0$ there is nothing to prove. Suppose the claim is true for $\ell-1$. If $i>1$, we have the following (all terms live in $\Z$ localized at the set of primes at which $u$ has negative valuation, and the congruence is taken modulo $q$):
 $$\mathfrak c_{i+\ell(i+1)}^+=\mathfrak c_{i+\ell(i+1)}+u=f^{((\ell-1)(i+1)+i+1)}(\mathfrak c_i)+u\equiv f^{((\ell-1)(i+1)+i+1)}(\pm u)+u.$$
 Now when $i=1$ and $\ell=1$, the right hand side of the congruence above is $-\mathfrak c_1^+$, which is $0$ modulo $q$ by assumption. In every other case, it coincides with $\mathfrak c_{i+(\ell-1)(i+1)}+u=\mathfrak c_{i+(\ell-1)(i+1)}^+$, which is 0 modulo $q$ by the inductive hypothesis. 
\end{proof}

\begin{theorem}
 Assume Vojta's conjecture over $\Q$ and let $\psi\coloneqq x^2-t^2\in\Q(t)[x]$. Then there exists a thin set $E\subseteq \Q$ such that for every $t_0\in \Q\setminus E$ we have $\im(\rho_{\psi_{t_0}})=\mathcal G_5$.
\end{theorem}
\begin{proof}
 Let us start by noticing that $\im(\rho_{\psi})=\mathcal G_5$. In fact, since the arboreal representation of $\phi=x^2-t$ has image $\Omega_{\infty}$, then so does the one of $x^2-t^2$ seen as a polynomial with coefficients in $\Q(t^2)$. It follows that $\im(\rho_{\psi})$ is an index two subgroup of $\Omega_{\infty}$ which is contained in $\mathcal G_5$, and so it actually coincides with it. By Theorem \ref{nonstable} we therefore have that for every $m\in\N$:
 \begin{equation}\label{fullsize}
  \dim \langle \mathfrak c_1(t)\pm t,\ldots,\mathfrak c_m(t)\pm t\rangle_{\Q(t)}=2m.
 \end{equation}
Now let $g_1=x+t$ and $g_2=x-t$. Let $n_1$ and $n_2$ be the positive integers determined by Lemma \ref{primitive_divisors} for $g_1\circ\psi$ and $g_2\circ\psi$, let $F_1$ and $F_2$ be the finite sets of exceptions (cf.\ Remark \ref{finiteness}) and let $n\coloneqq \max\{n_1,n_2\}$. The set of $t_0\in \Q$ such that
 $$\dim \langle \mathfrak c_1(t_0)\pm t_0,\ldots,\mathfrak c_n(t_0)\pm t_0\rangle_{\Q}<2n$$
 coincides with the set of $t_0$'s that appear as $t$-coordinate of at least a curve of the form $y^2=\prod_{i=1}^n(\mathfrak c_i(t)-t)^{e_i}(\mathfrak c_i(t)+t)^{f_i}$, where $e_i,f_i\in\{0,1\}$ are not all 0. On the other hand these curves are all irreducible because of \eqref{fullsize}. Let $E$ be the subset of all $t_0\in\Q$ such that at least one of the aforementioned curves has a rational point with $t$-coordinate $t_0$. Notice that $E$ is a thin set by definition.
 
  Now let $t_0\in \Q\setminus (E\cup F_1\cup F_2)$. From now on, for every $r\in \N$, let us set $V_r^+\coloneqq \langle \mathfrak c_1(t_0)-t_0,\mathfrak c_2(t_0)+t_0,\ldots,\mathfrak c_r(t_0)+t_0\rangle_{\Q}$ and $V_r^-\coloneqq \langle \mathfrak c_1(t_0)+t_0,\mathfrak c_2(t_0)-t_0,\ldots,\mathfrak c_r(t_0)-t_0\rangle_{\Q}$. Let $m\geq n$. Clearly, $\dim V_m^+=\dim V_m^-=m$, because since $a\notin F_1\cup F_2$ then for every $n'\geq n$ there exists a primitive prime divisor of $\mathfrak c_{n'}(t_0)+t_0=g_1(\psi^{n'-1}(0))$ (resp.\ $\mathfrak c_{n'}(t_0)-t_0=g_2(\psi^{n'-1}(0))$). To conclude the proof it is enough, by Theorem \ref{nonstable}, to show that if $b\in V_m^+\cap V_m^-$, where $b$ is a squarefree integer, then $b=1$. This is done by an easy induction. For $m=n$, the claim is true by construction. Let it be true for $m-1$ and pick $b$ as above. Since $V_{m-1}^+\cap V_{m-1}^-=\{1\}$, it follows that, wlog, $|b|=|(\mathfrak c_m(t_0)+t_0)\cdot (\mathfrak c_1(t_0)-t_0)^{r_1}\prod_{i=2}^{m-1}(\mathfrak c_i(t_0)+t_0)^{r_i}\cdot d^2|$, where $r_i\in \{0,1\}$ for all $i$, $d\in \Q^{\times}$ and $|\cdot |$ denotes the standard absolute value. Now let $p$ be a primitive prime divisor of $\mathfrak c_m(t_0)+t_0$. Then clearly $v_p(b)\equiv 1 \bmod 2$, and since $b\in V_m^-$ as well, by Lemma \ref{inter} we must have that $v_p(2\mathfrak c_1(t_0))\neq 0$. Since $p$ is odd, it must be $v_p(\mathfrak c_1(t_0))=v_p(t_0^2)\neq 0$, so that in particular $v_p(t_0)\neq 0$. But $\mathfrak c_1(t_0)+t_0=t_0(t_0+1)$ and therefore $v_p(\mathfrak c_1(t_0)+t_0)\neq 0$, contradicting the primitivity of $p$.
\end{proof}

\begin{remark}\label{group5r}
 Notice if $E\subseteq \Q$ is thin, then there are infinitely many integers outside of $E$ (see for example \cite[Theorem 3.4.4]{serre}). Thus, the above theorem shows that, under Vojta's conjecture, there exist infinitely many integral specializations $t_0$ of $t$ yielding $\mathcal G_5$ as image of $\rho_{\psi_{t_0}}$.
\end{remark}
\section{Proof of Theorem C}\label{proofC}
The goal of this section is to prove the following. 
\begin{theorem} \label{thmC}
Let $\underline{a},\underline{a}^{'}$ be two vectors in $\mathbb{F}_2^{(\mathbb{Z}_{\geq 0})}$. Then 
$$M_{\underline{a}} \cong_{\textup{top.gr.}} M_{\underline{a}^{'}}
$$
if and only if 
$$\underline{a}=\underline{a}^{'}.
$$
\end{theorem}
Since the proof is rather involved, before going into its technical implementation we offer here an high level outline overviewing the main ideas.

\textbf{Overview of the proof:} The first basic observation at the root of this proof, is that vectors $\underline{a}$ with $a_0=1$ and with $a_n=1$ for some $n>0$ are distinguished from others, from Theorem \ref{Describing abelianized Ma}, as those providing groups whose abelianization has an element of order $4$. The second one, is that the vector $(1,0, \ldots, 0,\ldots)$ is also distinguished by having a set of generators that can be put in two infinite blocks $B_1,B_2$, where every element of $B_1$ commutes with every element of $B_2$. This does not happen for any other vector. We formalize this type of property using the \emph{graph of commutativity}, introduced below. 

These two observations are already sufficient to distinguish a few of the vectors, but not quite all of them, since as soon as both vectors start with $0$, or when they both start with $1$ and they have both at least two $1$'s, then the two observations above do not tell us anything. The next step is to notice that if we ideally had a functor from profinite groups to profinite groups, that has the effect of shifting the relation to the left, then the two observations above would alone suffice after repeatedly applying the functor. In practice we are only able to produce functorially a group that is closely related to the one with the relation shifted to the left. The functor under consideration will be nothing else than taking commutators, and hence, upon iteration, we will be looking at the derived series of each $M_{\underline{a}}$. Recall that for a profinite group $\mathcal{G}$, the derived series is defined by putting $\mathcal{G}^{(0-\text{Fr.})}\coloneqq\mathcal{G}$ and
$$\mathcal{G}^{((i+1)-\text{Fr.})}\coloneqq[\mathcal{G}^{(i-\text{Fr.})},\mathcal{G}^{(i-\text{Fr.})}],
$$
while the lower central series is defined by putting 
$$\mathcal{G}^{(0)}\coloneqq\mathcal{G},
$$
and 
$$\mathcal{G}^{(i+1)}\coloneqq [\mathcal{G},\mathcal{G}^{(i)}].
$$
We will see that the maximal number of connected components of the graphs of commutativity along the derived series is eventually reading off precisely the largest integer $i$ with $a_i=1$ (or whether there is no such integer). The previous $1$'s in the vector are then read off by looking at which terms of the derived series are topologically generated by involutions. 

To study these derived series, we will exploit the self-replicating structure of the modules $\mathbb{F}_2[L_N]$ over the rings $\mathbb{F}_2[\Omega_N]$, which is crystallized in the exact sequence of Proposition \ref{father-son sequence}. This fundamental exact sequence will also allow us to quickly reprove classical theorems along the way, such as Kaloujnine's description of the group of the lower central series $\Omega_{\infty}^{(i)}$ \cite{kaloujnine}, to obtain a precise description of the derived series $\Omega_{\infty}$ (and of its maximal subgroups), as well as the uni-seriality of the $\mathbb{F}_2[\Omega_N]$-module $\mathbb{F}_2[L_N]$ \cite{grigorchuk}. The computation of the lower central series plays a key role in the computation of the derived series. 

We now proceed delving into the actual proof.

One of the two invariants that we will use to distinguish isomorphism classes of profinite groups is the aforementioned \emph{graph of commutativity}, which goes as follows. Let $\mathcal{G}$ be a profinite group and let $S\subseteq \mathcal{G}$ be a set of topological generators.

\begin{definition}
 The \emph{graph of commutativity} of $\mathcal G$ with respect to $S$, denoted by $\Gamma(\mathcal G,S)$, is the graph with vertex set $S$, and such that $g,h\in S$ are connected by an edge if and only if $[g,h]\neq 1$.
\end{definition}

This provides us with the following first invariant of a profinite group. 

\begin{definition}
Let $\mathcal{G}$ be a profinite group. We say that $\mathcal{G}$ \emph{admits a disconnected presentation}, in case there exists a set of topological generators $S$ of $\mathcal{G}$ not containing the identity, such that $\Gamma(\mathcal{G},S)$ is disconnected. In case no such $S$ exists, we say that $\mathcal{G}$ \emph{admits only connected presentations}.     
\end{definition}

Next, we refine the previous definition to get a measure of how disconnected a profinite group can be. 
\begin{definition}
Let $\mathcal{G}$ be a profinite group and $h$ a positive integer. We say that $\mathcal{G}$ has $h$ components in case $h$ is the largest number of connected components of $\Gamma(\mathcal{G},S)$ as $S$ varies among the possible sets of topological generators not containing the identity.    
\end{definition}

Finally, we have a more straightforward invariant. 
\begin{definition}
Let $\mathcal{G}$ be a profinite group. We say that $\mathcal{G}$ can be generated by involutions, in case there exists a set $S$ of topological generators for $\mathcal{G}$ consisting entirely of involutions. In case no such $S$ exists, we say that $\mathcal{G}$ cannot be generated by involutions.    
\end{definition}
For a non-zero element $\underline{a}=(a_n)_{n\ge 0}\in\mathbb{F}_2^{(\mathbb{Z}_{\geq 0})}$ we define $i_{\max}(\underline{a})$ to be the largest integer $i$ with $a_i=1$, and we set $i_{\max}(\underline{0})=+\infty$. By convention we have that $j<+\infty$, for every integer $j$.  

We are now ready to state the first reduction step in the proof of Theorem \ref{thmC}.
\begin{proposition}\label{prop1}
Let $\underline{a}$ be a vector in $\mathbb{F}_2^{(\mathbb{Z}_{\geq 0})}$ and let $j$ be a non-negative integer. Then the following three hold true.
\begin{enumerate}[(a)]
\item The profinite group $\Omega_{\infty}^{(j-\textup{Fr.})}$ admits only connected presentations. 
\item For all $j>i_{\max}(\underline{a})$, the profinite group $M_{\underline{a}}^{(j-\textup{Fr.})}$ has $2^{i_{\max}(\underline{a})+1}$ components.
\item Let $j<i_{\max}(\underline{a})$. If $a_j=1$, the profinite group $M_{\underline{a}}^{(j-\textup{Fr.})}$ cannot be generated by involutions. Instead, if $a_j=0$ then $M_{\underline{a}}^{(j-\textup{Fr.})}$ can be generated by involutions. 
\end{enumerate}
\end{proposition}
Let us now prove that Proposition \ref{prop1} implies Theorem \ref{thmC}.

\subsection{Proposition \ref{prop1} implies Theorem \ref{thmC}} \label{sbs: from 1 to C}

In this subsection we show that we can deduce Theorem \ref{thmC} from Proposition \ref{prop1}. 

\begin{proof} Let us assume Proposition \ref{prop1} holds and let $\underline{a},\underline{a}^{'}$ be two vectors in $\mathbb{F}_2^{(\mathbb{Z}_{\geq 0})}$ such that 
$$M_{\underline{a}} \cong_{\textup{top.gr.}} M_{\underline{a}^{'}}.
$$
For every non-negative integer $j$, we must then have
$$M_{\underline{a}}^{(j-\text{Fr.})} \cong_{\textup{top.gr.}} M_{\underline{a}^{'}}^{(j-\text{Fr.})},
$$
so that in particular both groups have the same the same number of components. Parts $(a)$ and $(b)$ of Proposition \ref{prop1} show then immediately that $\underline{a}=\underline{0}$ if and only if $\underline{a}'=\underline{0}$.

We can then assume, from now on, that both vectors are non zero, so that $i_{\max}(\underline{a}),i_{\max}(\underline{a}')<+\infty$. Part $(b)$ of Proposition \ref{prop1} gives us, choosing $j$ large enough, that $2^{i_{\text{max}}(\underline{a})+1}=2^{i_{\text{max}}(\underline{a}^{'})+1}$, so that 
$$i_{\text{max}}(\underline{a})=i_{\text{max}}(\underline{a}^{'}),
$$
and we can therefore call $i$ such index.
Now, in case there exists $j<i$ with $a_j=1, a_j^{'}=0$, then part $(c)$ of Proposition \ref{prop1} implies that $M_{\underline{a}}^{(j-\text{Fr.})}$ cannot be generated by involutions, while $M_{\underline{a}^{'}}^{(j-\text{Fr.})}$ can be generated by involutions. It follows that that 
$$M_{\underline{a}}^{(j-\text{Fr.})} \not \cong_{\textup{top.gr.}} M_{\underline{a}^{'}}^{(j-\text{Fr.})},
$$
which contradicts 
$$M_{\underline{a}} \cong_{\textup{top.gr.}} M_{\underline{a}^{'}}.
$$
The same argument works with the roles of the two vectors swapped. Therefore we deduce for each $j<i$ that $a_j=1$ if and only if $a_j^{'}=1$. Overall we have exactly proved that
$$\underline{a}=\underline{a}^{'}.
$$
\end{proof}

Our next goal is to prove Proposition \ref{prop1}. To this end we will provide an explicit description of $M_{\underline{a}}^{(j-\text{Fr.})}$ for each non-negative integer $j$. To state this description, we first need a definition. 
\begin{definition}
Let $h$ be in $\mathbb{Z}_{\geq 0}$ and $\underline{a}$ be in $\mathbb{F}_2^{(\mathbb{Z}_{\geq 0})}$. We define
$$(M_{\underline{a}})_{\text{fib}}^h \coloneqq\{(\sigma_1,\ldots,\sigma_j) \in M_{\underline{a}}^{h}: \forall i \in \mathbb{Z}_{\geq 0} \ \phi_i(\sigma_1)=\ldots=\phi_i(\sigma_h)\}.
$$
\end{definition}
Let $\sh$ be the left shift operator, i.e.\ the unique $\mathbb{F}_2$-linear endomorphism of $\mathbb{F}_2^{(\mathbb{Z}_{\geq 0})}$ given by
$$\sh((a_n)_{n\ge 0})\coloneqq (a_{n+1})_{n\ge 0}.$$
We denote by $\sh_n$ the $n$-th iteration of $\sh$. We have the following. 

\begin{proposition}\label{prop2}
Let $j$ be in $\mathbb{Z}_{\geq 0}$ and $\underline{a}$ be in $\mathbb{F}_2^{(\mathbb{Z}_{\geq 0})}$. The following two hold true.

\begin{enumerate}[(a)]
    \item Suppose that $j \leq i_{\textup{max}}(\underline{a})$. Then:
$$M_{\underline{a}}^{(j-\textup{Fr.})} \cong_{\textup{top.gr.}}(M_{\textup{sh}_j(\underline{a})})_{\textup{fib}}^{2^j}.
$$
\item Suppose that $j>i_{\textup{max}}(\underline{a})$. Then:
$$M_{\underline{a}}^{(j-\textup{Fr.})} \cong_{\textup{top.gr.}}((\Omega_{\infty})_{\textup{fib}}^{2^{j-i_{\max}(\underline{a})}})^{2^{i_{\max}(\underline{a})+1}}.
$$  
\end{enumerate}
\end{proposition}
The next proposition is the second ingredient needed to prove Proposition \ref{prop1}.   
\begin{proposition} \label{prop:generatedbyinvolutions}
Let $\underline{a}$ be in $\mathbb{F}_2^{(\mathbb{Z}_{\geq 0})}$ and let $h$ be a positive integer. Then
$$(M_{\underline{a}})_{\textup{fib}}^{h}
$$
cannot be generated by involutions if and only if $a_0=1$ and $i_{\max}(\underline{a})>0$.  
\end{proposition}

Let us now prove that Propositions \ref{prop2} and \ref{prop:generatedbyinvolutions} imply Proposition \ref{prop1}.

\subsection{Propositions \ref{prop2} and \ref{prop:generatedbyinvolutions} imply Proposition \ref{prop1}} \label{sbs:from 2 to 1}

In this subsection we show that we can deduce Proposition \ref{prop1} from Propositions \ref{prop2} and \ref{prop:generatedbyinvolutions}. To this end we will need four auxiliary facts. The first one will enable us to control graphs of commutativity of the relevant groups. 

\begin{proposition} \label{stuff do not commute often}
Let $\tau_1,\tau_2$ be in $\Omega_{\infty}$ with $\phi_0(\tau_1)=1$ and $\tau_2$ is not in $\{\tau_1,\text{id}\}$ modulo commutators. Then $\tau_1$ and $\tau_2$ do not commute. 
\end{proposition}
\begin{proof}
First observe that we can assume that $\phi_0(\tau_2)=0$. Indeed if this is not the case, then $\tau_1\tau_2$ is also not in $\{\tau_1,\text{id}\}$ modulo commutators and if $\tau_1$ does not commute with $\tau_1\tau_2$ then it does not commute with $\tau_2$ as well. 

Next, notice that there must exist $i \in \mathbb{Z}_{\geq 1}$ such that $\phi_i(\tau_2)=1$, otherwise $\tau_2\in [\Omega_\infty,\Omega_\infty]$. On the other hand observe that, thanks to the formula of uncertain additivity \eqref{uncertain_additivity} given in Section \ref{subgroups of index 2, abstract} and the fact that $\phi_{0}(\tau_1)=1$ and $\phi_0(\tau_2)=0$, we have that
$$\widetilde{\phi}_i(x)(\tau_2\tau_1)=\widetilde{\phi}_i(y)(\tau_2)+\widetilde{\phi}_i(x)(\tau_1),
$$
and
$$\widetilde{\phi}_i(x)(\tau_1\tau_2)=\widetilde{\phi}_i(x)(\tau_1)+\widetilde{\phi}_i(x)(\tau_2).
$$
Adding term to term, we obtain that 
$$\widetilde{\phi}_i(x)(\tau_2\tau_1) \neq \widetilde{\phi}_i(x)(\tau_1\tau_2),
$$
so that in particular $\tau_2\tau_1 \neq \tau_1\tau_2$. 
\end{proof}

The second auxiliary fact is a basic property of $\Omega_{\infty}$.

\begin{proposition} \label{infinite conjugacy classes}
Every element of $\Omega_{\infty}$ different from the identity has an infinite conjugacy class.
\begin{proof}
First, recall that the conjugacy class of an element $g$ of a group $G$ is infinite if and only if the centralizer of $g$ in $G$, denoted by $C_G(g)$, has infinite index in $G$.

Let $\sigma\in \Omega_\infty$ be such that $\phi_0(\sigma) \neq 0$. Proposition \ref{stuff do not commute often} implies that $C_{\Omega_{\infty}}(\sigma)$ maps inside a $1$-dimensional vector space of $\Omega_{\infty}^{\text{ab}}$, namely the one generated by $\sigma$. This has infinite index in $\Omega_{\infty}^{\text{ab}}$, since this group is infinite. In particular $C_{\Omega_{\infty}}(\sigma)$ has infinite index in $\Omega_{\infty}$. 

Let now $\sigma\in \Omega_\infty$ be different from the identity, and let $i$ be the first level such that $\sigma$ acts non-trivially on $L_i$. This means that we can review naturally $\sigma$ as an element of $\Omega_{\infty}^{2^{i-1}} \subseteq \Omega_{\infty}$. In particular one of the $2^{i-1}$ coordinates must be a $\tau\in \Omega_\infty$ with $\phi_0(\tau) \neq 0$: if not, $\sigma$ would be trivial also on $L_i$. But then it means that $\tau$ has infinite conjugacy class with respect to the subgroup $\Omega_{\infty}^{2^{i-1}} \subseteq \Omega_{\infty}$, which implies certainly that also $\sigma$ does with respect to the ambient group $\Omega_{\infty}$.    
\end{proof}
\end{proposition}

Next, we prove a proposition that guarantees the connectedness of $(\Omega_{\infty})_{\text{fib}}^{h}$. 

\begin{proposition} \label{prop: graphsofcommutativity}
Let $h$ be a positive integer. Then $(\Omega_{\infty})_{\textup{fib}}^{h}$ admits only connected presentations.      
\end{proposition}
\begin{proof}
Let $S$ be a set of topological generators of $(\Omega_{\infty})_{\textup{fib}}^{h}$ not containing the identity. We will prove that $\Gamma((\Omega_\infty)_{\textup{fib}}^{h},S)$ is starlike, in the following sense. We show how to produce a subset $S_0\subseteq S$ such that:
\begin{itemize}
    \item There exists a node $\sigma_0$ of $S_0$ that is connected to all other nodes of $S_0$.
    \item Every $\tau\in S\setminus S_0$ is connected to some node of $S_0$.  
\end{itemize}

This obviously shows that $\Gamma((\Omega_\infty)_{\textup{fib}}^{h},S)$ is connected.

To achieve this, pick $S_0$ a subset of $S$ that is sent bijectively into a minimal set of generators through the vector of characters $(\phi_s)_{s \geq 0}$ (notice that in principle for each $i$ we have $h$ such characters, one for each projection, but they all coincide thanks to the definition of $(\Omega_{\infty})_{\textup{fib}}^{h}$).


Let now $\sigma_0\in S_0$ be such that $\phi_0(\sigma_0) \neq 0$. Since any two distinct elements of $S_0$ are linearly independent in $\F_2^{\Z_{\ge 0}}$ by construction, it follows that every node of $S_0$ is connected to $\sigma_0$ thanks to Proposition \ref{stuff do not commute often}. 

Finally, take an element $\tau\in S\setminus S_0$. It cannot be the identity by assumption. Therefore there exists some $s\in \{1,\ldots,h\}$ such that the $s$-th coordinate projection of $\tau$ is non-trivial. But then, the $s$-th coordinate projection of $\tau$ cannot commute with the entire $s$-th coordinate projection of $S$. Indeed if it does, then it would commute with the entire $\Omega_{\infty}$ (observe that by construction the $s$-th coordinate projection of $S_0$ is a minimal set of topological generators for $\Omega_{\infty}$). However Proposition \ref{infinite conjugacy classes} shows in particular that the center of $\Omega_{\infty}$ is trivial (and we know that the $s$-th coordinate projection of $\tau$ is not). Hence we have shown that every element of $S$ is connected to an element of $S_0$.
\end{proof}

We conclude the list of auxiliary propositions with the following elementary one. 

\begin{proposition} \label{it has h components}
Let $\mathcal{G}$ be a profinite group and let $h$ be a positive integer. Suppose that $\mathcal{G}$ admits only connected presentations. Then $\mathcal{G}^h$ has $h$ components.     
\begin{proof}
Let $S$ be a set of topological generators of $\mathcal{G}^h$ not containing the identity. For every integer $k\in \{1,\ldots,h\}$, we denote by $S_k$ the subset of elements of $S$ having non-trivial $k$-th coordinate projection. The $k$-th projection of $S_k$ is a set of topological generators of $\mathcal{G}$ not containing the identity. Therefore it must give a connected graph by assumption. It follows that $S_k$ is entirely contained in a connected component of $\Gamma(\mathcal{G}^h,S)$. It follows that the number of connected components does not exceed $h$. 

Conversely, if we start with a set of topological generators $S'$ of $\mathcal{G}$ not containing the identity, we can construct a set $S$ of topological generators of $\mathcal{G}^h$ by taking $h$ copies of $S'$, one for each coordinate. Clearly $\Gamma(\mathcal{G},S)$ has at least $h$ connected components. 

Since the number of connected components cannot exceed $h$ and can reach $h$, we arrive precisely at the desired conclusion. 
\end{proof}
\end{proposition}

\emph{Conclusion of the proof that Proposition \ref{prop2} and Proposition \ref{prop:generatedbyinvolutions} imply Proposition \ref{prop1}.}
\begin{proof} Let us assume that Proposition \ref{prop2} and \ref{prop:generatedbyinvolutions} hold.

Part $(a)$ of Proposition \ref{prop1} follows from part $(a)$ of Proposition \ref{prop2} together with Proposition \ref{prop: graphsofcommutativity}.

Let now $\underline{a}\in \mathbb{F}_2^{(\mathbb{Z}_{\geq 0})}$ be non-zero, so that $i_{\max}(\underline{a})<+\infty$.

To prove part $(b)$ of Proposition \ref{prop1}, notice that if $j>i_{\max(\underline{a})}$ then by part $(b)$ of Proposition \ref{prop2} we get that:
$$M_{\underline{a}}^{(j-\textup{Fr.})} \cong_{\textup{top.gr.}}((\Omega_{\infty})_{\textup{fib}}^{2^{j-i_{\max}(\underline{a})}})^{2^{i_{\max}(\underline{a})+1}}.
$$  

By Proposition \ref{prop: graphsofcommutativity}, the group $(\Omega_{\infty})_{\textup{fib}}^{2^{j-i_{\max}(\underline{a})}}$ admits only connected presentations, and Proposition \ref{it has h components} concludes the proof.

We are left with proving part $(c)$ of Proposition \ref{prop1}. Let then $j<i_{\max}(\underline{a})$ be a non-negative integer. By part $(a)$ of Proposition \ref{prop2} we get that 
$$M_{\underline{a}}^{(j-\textup{Fr.})} \cong_{\textup{top.gr.}}(M_{\textup{sh}_j(\underline{a})})_{\textup{fib}}^{2^j},$$
and since $a_j=1$ if and only if $\sh_j(\underline{a})_0=1$, Proposition \ref{prop:generatedbyinvolutions} implies that $M_{\underline{a}}^{(j-\text{Fr.})}
$
cannot be generated by involutions if and only if $a_j=1$.
\end{proof}

We are now left with proving Proposition \ref{prop2} and Proposition \ref{prop:generatedbyinvolutions}. To this end we need an intermezzo on the derived and lower central series of $\Omega_{\infty}$. 
\subsection{Intermezzo: the derived and lower central series of \texorpdfstring{$\Omega_{\infty}$}{}}
If $G$ is a group, we denote by $I_{G}$ the augmentation ideal in the group ring $\mathbb{F}_2[G]$. Let us start with an easy fact. 
\begin{proposition} \label{derandcent in semidir}
Let $G$ be a finite group and $A$ be an $\mathbb{F}_2[G]$-module. Let $i$ be a non-negative integer. Then
$$(A \rtimes G)^{(i)}=I_G^i\cdot A \rtimes G^{(i)}
$$
and if $i>0$ then
$$(A \rtimes G)^{(i-\textup{Fr.})}=I_{G^{(i-1)-\textup{Fr.}}}\cdot I_{G^{(i-2)-\textup{Fr.}}}\cdot \ldots \cdot I_G \cdot A \rtimes G^{(i-\textup{Fr.})}
$$
\begin{proof}
Both identities follows immediately from applying repeatedly the case $i=1$. This single case is a straightforward computation.     
\end{proof}
\end{proposition}

Recall the digital representation $\sigma=(\sigma_N)_{N \in \mathbb{Z}_{\geq 1}}$ introduced in Definition \ref{digital_representation}. Proposition \ref{derandcent in semidir} immediately implies the following thanks to the iterated semi-direct product description of $\Omega_{\infty}$.  

\begin{corollary} \label{reducing descending central to augmentational}
Let $i$ be a positive integer. Then:
$$
\Omega_{\infty}^{(i)}=\{(\sigma_N)_{N \in \mathbb{Z}_{\geq 0}}: \sigma_N \in I_{\Omega_{N}}^{i} \cdot \mathbb{F}_2[L_N] \ \text{for each} \ N \in \mathbb{Z}_{\geq 0} \}.
$$
\end{corollary}
Thanks to Corollary \ref{reducing descending central to augmentational} we are reduced to examine the filtration 
$$\{I_{\Omega_N}^{i} \cdot\mathbb{F}_2[L_N]\}_{i \geq 0}
$$
for every $N\geq 0$.  To this end the following proves to be crucial. Observe that for a positive integer $N$ we have two natural maps of $\mathbb{F}_2[\Omega_N]$-modules\footnote{Here $\mathbb{F}_2[L_{N-1}]$ is considered as a $\mathbb{F}_2[\Omega_N]$-module via the natural projection $\Omega_N \twoheadrightarrow \Omega_{N-1}$.}
$$s_N: \mathbb{F}_2[L_{N-1}] \to \mathbb{F}_2[L_N]
$$
and
$$f_N: \mathbb{F}_2[L_N] \to \mathbb{F}_2[L_{N-1}],
$$
defined as follows. The map $s_N$ sends each vertex of $L_{N-1}$ into the sum of his two neighbors in $L_N$. The map $f_N$ sends each vertex of $L_N$ into its neighbor in $L_{N-1}$. This assignment uniquely extends to an $\mathbb{F}_2$-linear map, which is clearly $\Omega_N$-linear from the way it is defined. 
\begin{proposition} \label{father-son sequence}
Let $N$ be a positive integer. The maps $s_N,f_N$ induce an exact sequence of $\mathbb{F}_2[\Omega_N]$-modules
$$0 \to \mathbb{F}_2[L_{N-1}] \to \mathbb{F}_2[L_N] \to \mathbb{F}_2[L_{N-1}] \to 0.
$$
Furthermore $s_N(\mathbb{F}_2[L_{N-1}])=I_{\Omega_{N}}^{2^{N-1}} \cdot \mathbb{F}_2[L_N]$.
\end{proposition}
\begin{proof}
Checking the exactness of the sequence is elementary; we shall prove the second part of the statement. 

Observe that $I_{\Omega_{N-1}}^{2^{N-1}}\cdot\mathbb{F}_2[L_{N-1}]=0$ because $\text{dim}_{\mathbb{F}_2}(\mathbb{F}_2^{L_{N-1}})=2^{N-1}$ and the augmentation ideal is nilpotent. This already implies $s_N(\mathbb{F}_2[L_{N-1}]) \supseteq I_{\Omega_N}^{2^{N-1}} \cdot \mathbb{F}_2[L_N]$ by the exactness of the sequence. On the other hand we have that a basis of $s_N(\mathbb{F}_2[L_{N-1}])$ is given by the set
$$\{xw+yw\}_{w \in \mathbb{F}_2[L_{N-1}]} \subseteq \mathbb{F}_2[L_N].
$$
Recall that $\Omega_{N}$ is the $2$-Sylow of $\text{Sym}(L_N)$. As such it is equipped with at least one cycle $\rho$ of order $2^N$. Combining the fact that the group generated by $\rho$ acts simply transitively on $L_N$, along with the fact that $\rho^{2^{N-1}}$ acts trivially on $L_{N-1}$, we see that $\rho^{2^{N-1}}(xw)=yw$ and $\rho^{2^{N-1}}(yw)=xw$ for each $w \in L_{N-1}$. Therefore
$$xw+yw=(\rho^{2^{N-1}}+\text{Id})(xw)=(\rho+\text{Id})^{2^{N-1}}(xw),
$$
which gives precisely the desired conclusion, since $\rho+\text{Id} \in I_{\Omega_{N}}$. 
\end{proof}

Thanks to Proposition \ref{father-son sequence} we see that essentially the filtration $I_{\Omega_{N}}^{\bullet} \cdot \mathbb{F}_2[L_N]$ consists in repeating two times the filtration $I_{\Omega_{N-1}}^{\bullet} \cdot \mathbb{F}_2[L_{N-1}]$. This has two important consequences.

The first one is that it puts the basis for a recursive description of $I_{\Omega_N}^{i} \cdot \mathbb{F}_2[L_N]$. It is already clear from Proposition \ref{father-son sequence} that the resulting criterion, which we next explain, will depend in a natural manner on the base $2$-expansion of $i$. To this end it will be convenient to pause and review the space $\mathbb{F}[L_N]$ as the space of functions
$$\text{Fun}(L_N,\mathbb{F}_2),
$$
where each word $w$ in $L_N$ is identified with the delta-function $\delta_w$ that vanishes at all words $w'$ in $L_N$ except from $w$. In this manner, we have that the equality of Proposition \ref{father-son sequence}
$$I_{\Omega_N}^{2^{N-1}}\cdot \mathbb{F}_2[L_N]=s_N(\mathbb{F}_2[L_{N-1}])
$$
tells us that $I_{\Omega_N}^{2^{N-1}}\cdot \mathbb{F}_2[L_N]$ consists precisely of those functions $f \in \text{Fun}(L_N,\mathbb{F}_2)$ that do not depend on the value of the first symbol in the word, more formally
$$I_{\Omega_N}^{2^{N-1}}\cdot \mathbb{F}_2[L_N]=\{f:L_N \to \mathbb{F}_2: f(xw)=f(yw), \ \text{for all} \ w\in L_{N-1}\}.
$$
Of course this can be naturally identified with $\text{Fun}(L_{N-1},\mathbb{F}_2)$, and the quotient map is precisely the map $f_N$ sending a function $f$ into the sum on any pair of descendants. This is a function-theoretic reinterpretation of the maps $s_N$ and $f_N$.

It will now be convenient to identify each word $w$ in $L_N$ with a vector $v(w)$ in $\mathbb{F}_2^N$, following the insight of \cite{kaloujnine}. If $x$ appears in position $k$ from left to right of $w$, then we place $0$ in position $N+1-k$ in the vector $v(w)$, while if $y$ appears in position $k$ we place a $1$ in position $N+1-k$. Of course this gives a bijection between $L_N$ and $\mathbb{F}_2^{N}$ and we now wish to translate our description of Proposition \ref{father-son sequence} under this bijection. To this end, let $x_1, \ldots, x_N$ be the coordinate projection functions from $\mathbb{F}_2^N$ to $\mathbb{F}_2$ and recall that any function $f\colon \mathbb{F}_2^N \to \mathbb{F}_2$ can be realized as a polynomial map in $x_1, \ldots,x_N$. More precisely, the natural map
$$\frac{\mathbb{F}_2[x_1, \ldots,x_N]}{(x_1^2-x_1,\ldots,x_N^2-x_N)} \to \text{Fun}(\mathbb{F}_2^N,\mathbb{F}_2)
$$
$$p(x_1,\ldots,x_N)\mapsto (
   (y_0,\ldots,y_N)\mapsto p(y_1,\ldots,y_n))$$
is a ring isomorphism. An $\F_2$-basis for $\frac{\mathbb{F}_2[x_1, \ldots,x_N]}{(x_1^2-x_1,\ldots,x_N^2-x_N)}$ consists of
$$\{x_T\}_{T \subseteq \{1,\ldots,N\}},
$$
where 
$$x_T\coloneqq\prod_{i \in T}x_i.
$$
We are going to describe the spaces $I_{\Omega_N}^i\cdot \mathbb{F}_2[L_N]$, through Proposition \ref{father-son sequence}, in this space of polynomials. Before delving in the auxiliary important notion of \emph{height} of a polynomial and into a formal argument, we give below the essential gist of the argument.  

\emph{Informal overview:} As observed above, the space $s_N(\mathbb{F}_2[L_{N-1}])=I_{\Omega_N}^{2^{N-1}}\cdot \mathbb{F}_2[L_N]$ is nothing else than the space of polynomials that do not depend on the last variable $x_N$, i.e.\ $\frac{\mathbb{F}_2[x_1, \ldots,x_{N-1}]}{(x_1^2-x_1,\ldots,x_{N-1}^2-x_{N-1})}$. Proposition \ref{father-son sequence} says that the filtration past this point is precisely a repetition of the filtration obtained on polynomials in $N-1$ variables. If we quotient out by this space we get nothing else than another copy of $\frac{\mathbb{F}_2[x_1, \ldots,x_{N-1}]}{(x_1^2-x_1,\ldots,x_{N-1}^2-x_{N-1})}$, namely $x_N \cdot \frac{\mathbb{F}_2[x_1, \ldots,x_{N-1}]}{(x_1^2-x_1,\ldots,x_{N-1}^2-x_{N-1})}$. Proposition \ref{father-son sequence} says precisely that the filtration previous level $2^{N-1}$ is a repetition of the filtration obtained on polynomials in $N-1$ variables. This ends our informal overview (which contains the main idea of the proof of Proposition \ref{Explicit description for desc. centr. series of omegainfty}), and suggests the following definition (see also \cite{kaloujnine}). 

\begin{definition} \label{def:height polynomials}
Let $T$ be a subset of $\{1, \ldots, N\}$. We define the \emph{height} of $x_T$ to be
$$h(x_T)\coloneqq\sum_{h \in T}2^{h-1}.
$$
In this way the height of the monomials cover each number between $0$ (attained at $1=x_{\emptyset}$) and $2^N-1$ (attained at $x_{\{1,\ldots,N\}}$) precisely once. We define the height $h(f)$ of any polynomial $f \in \frac{\mathbb{F}_2[x_1, \ldots,x_{N}]}{(x_1^2-x_1,\ldots,x_{N}^2-x_{N})},$ decomposed uniquely as
$$f\coloneqq\sum_{T \subseteq \{1, \ldots,N\}}\lambda_Tx_T,
$$
to be the largest height of a monomial $x_T$ with $\lambda_T=1$, if $f$ is non-zero, and we put $h(0)=0$.  

We put for each $i$ between $0$ and $2^N$:
$$V_i(N)\coloneqq \left\{f \in \frac{\mathbb{F}_2[x_1, \ldots,x_{N}]}{(x_1^2-x_1,\ldots,x_{N}^2-x_{N})}\colon h(f) \leq 2^N-1-i\right\}.
$$
\end{definition}
We are now ready to give the desired description. 
\begin{proposition}\label{Explicit description for desc. centr. series of omegainfty}
For every $i \in \{0,\ldots,2^N-1\}$ we have that under the above identifications
$$V_i(N)=I_{\Omega_{N}}^i \cdot \mathbb{F}_2[L_N].
$$
\begin{proof}
We proceed by induction on $N$. For $N=0$ the claim is trivial. Let now $N$ be a positive integer and suppose to have the desired conclusion for $N-1$. Let us distinguish two cases.

\emph{Case $1$:} Suppose first that $i \geq 2^{N-1}$. Then Proposition \ref{father-son sequence} tells us that
$$I_{\Omega_{N}}^i \cdot \mathbb{F}_2[L_N]=s_N(I_{\Omega_{N}}^{i-2^{N-1}} \cdot \mathbb{F}_2[L_{N-1}]).$$
However the map $s_N$ at the level of polynomials is simply the inclusion map of $\frac{\mathbb{F}_2[x_1, \ldots,x_{N-1}]}{(x_1^2-x_1,\ldots,x_{N-1}^2-x_{N-1})}$ in $\frac{\mathbb{F}_2[x_1, \ldots,x_{N}]}{(x_1^2-x_1,\ldots,x_{N}^2-x_{N})}$. Hence we get by induction that $I_{\Omega_{N}}^i \cdot \mathbb{F}_2[L_N]$ is the space of polynomials involving only the first $N-1$ variables and of height at most $2^{N-1}-1-(i-2^{N-1})=2^N-1-i$, which is precisely the space $V_i(N)$, as desired (notice that the fact that $x_N$ does not appear is the empty condition, once the height is strictly bounded by $2^N-1$).

\emph{Case $2$:} Suppose now that $i<2^{N-1}$. Then Proposition \ref{father-son sequence} implies that
\begin{equation}\label{f_s}
    I_{\Omega_{N}}^i \cdot \mathbb{F}_2[L_N]=f_N^{-1}(I_{\Omega_{N-1}}^i\cdot \mathbb{F}_2[L_{N-1}]).
\end{equation}

However at the level of polynomials, if $S \subseteq \frac{\mathbb{F}_2[x_1, \ldots,x_{N-1}]}{(x_1^2-x_1,\ldots,x_{N-1}^2-x_{N-1})}$ the map $f_N^{-1}$ sends $S$ into
$$\frac{\mathbb{F}_2[x_1, \ldots,x_{N-1}]}{(x_1^2-x_1,\ldots,x_{N-1}^2-x_{N-1})} \oplus x_N \cdot S.
$$
Since by the inductive assumption we have that $V_i(N-1)=I_{\Omega_{N-1}}^i\cdot \F_2^{L_{N-1}}$, equation \eqref{f_s} implies that
$$I_{\Omega_N}^i\cdot \F_2^{L_N}=f_N^{-1}(V_i(N-1))=\frac{\mathbb{F}_2[x_1, \ldots,x_{N-1}]}{(x_1^2-x_1,\ldots,x_{N-1}^2-x_{N-1})} \oplus x_N \cdot V_i(N-1).
$$
Hence we get that $I_{\Omega_{N}}^i \cdot \mathbb{F}_2[L_N]$ consists of the set of polynomials of height at most $2^{N-1}+(2^{N-1}-i-1)=2^N-i-1$, which is precisely $V_i(N)$, as desired. 
\end{proof}
\end{proposition}
Combining Corollary \ref{reducing descending central to augmentational} and Proposition \ref{Explicit description for desc. centr. series of omegainfty} we obtain the following corollary, which reproves the main result of \cite{kaloujnine}.

\begin{corollary} \label{total explicit description}
Let $i$ be a non-negative integer. We have that

$$\Omega_{\infty}^{(i)}=\{ (\sigma_N)_{N \in \mathbb{Z}_{\geq 0}}: \sigma_N \in V_i(N),\,\, \forall N\geq 0\}.$$
\end{corollary}

In particular, we deduce the following fact.
\begin{corollary} \label{at powers of 2-1}
Let $h$ be a non-negative integer. Then 
$$\Omega_{\infty}^{(2^h-1)}\cong(\Omega_{\infty})_{\textup{fib}}^{2^h}.
$$

\begin{proof}
For $h=0$ the claim is a tautology since it equates a group with itself. Suppose now that $h>0$. Observe that $I_{\Omega_h}^{2^h-1}\cdot \mathbb{F}_2[L_h]$ consists of the $1$-dimensional space generated by 
$$\sum_{w \in L_h}w.
$$
In fact, by Proposition \ref{Explicit description for desc. centr. series of omegainfty} this is the space of polynomials of height $2^h-1-(2^h-1)=0$. These are precisely the constant polynomials, hence the constant functions, which is precisely the sum above and $0$. 

Let now $N\ge h$ be an integer. To describe $I_{\Omega_N}^{2^h-1} \cdot \mathbb{F}_2[L_N]$, we can apply 
$$f_{h+1} \circ \ldots \circ f_N,
$$
and have that
$$I_{\Omega_N}^{2^h-1} \cdot \mathbb{F}_2[L_N]=(f_{h+1} \circ \ldots \circ f_N)^{-1}V_{2^h-1}(h)=$$
$$=(f_{h+1} \circ \ldots \circ f_N)^{-1}\left(\mathbb{F}_2\cdot \left(\sum_{w \in L_h}w\right)\right).
$$
In other words, $I_{\Omega_N}^{2^h-1}\cdot \mathbb{F}_2[L_N]$ consists precisely of the functions $f: L_N \to \mathbb{F}_2$ such that the induced function $L_h \to \mathbb{F}_2$ given by
$$w' \mapsto \sum_{w_0 \in L_{N-h}}f(w_0w')
$$
is constant. A moment of reflection shows that therefore:
$$\{(f_n)_{n\in \Z_{\ge 0}}\in \Omega_\infty\colon f_n\in I_{\Omega_n}^{2^h-1}\cdot \F_2[L_n] \mbox{ for every } n\}\cong(\Omega_\infty)^{2^h}_{\text{fib}},$$
and invoking Corollary \ref{reducing descending central to augmentational}, we get the claim
\end{proof}
\end{corollary}
We now turn to the second, also well-known, consequence of Proposition \ref{father-son sequence}, which is the \emph{uni-seriality} of the module $\mathbb{F}_2[L_N]$ over the ring $\mathbb{F}_2[\Omega_N]$, see for instance \cite{grigorchuk}. In a rather precise sense, it says that $\mathbb{F}_2[L_N]$ as $\mathbb{F}_2[\Omega_N]$-module, behaves like a cyclic module over a DVR.

\begin{definition} \label{defweight}
Let $N$ be a non-negative integer and let $x$ be a non-zero element of $\mathbb{F}_2[L_N]$. We define the \emph{weight} of $x$, denoted with $w(x)$, to be the largest non-negative integer $n$ such that $x$ belongs to $I_{\Omega_N}^n\cdot \mathbb{F}_2[L_N]$.     
\end{definition}

The following proposition proves the aforementioned uni-seriality. Recall that a \emph{maximal cycle} is an element $\rho\in \Omega_\infty$ with $\phi_i(\rho)=1$ for every $i\ge 0$. 

\begin{proposition} \label{Uniseriality}
The following hold true.
\begin{enumerate}[(a)]
    \item The filtration
$$\mathbb{F}_2[L_N] \supseteq I_{\Omega_N} \cdot\mathbb{F}_2[L_N] \supseteq I_{\Omega_N}^2 \cdot \mathbb{F}_2[L_N] \supseteq \ldots\supseteq I_{\Omega_N}^{2^{N}} \cdot \mathbb{F}_2[L_N]=0,
$$
coincides with the filtration
$$\mathbb{F}_2[L_N] \supseteq (\rho-1) \cdot\mathbb{F}_2[L_N] \supseteq (\rho-1)^2 \cdot \mathbb{F}_2[L_N] \supseteq \ldots\supseteq (\rho-1)^{2^N} \cdot \mathbb{F}_2[L_N]=0,
$$
where $\rho$ is any maximal cycle.
 \item Every submodule of the filtration has index $2$ in the previous one.
 \item Every submodule of $\mathbb{F}_2[L_N]$ has the form $I_{\Omega_N}^i\cdot \F_2[L_N]$ for some $i\in \{0,\ldots,2^N\}$.
 \item For every non-zero $x \in \mathbb{F}_2[L_N]$, we have that
$$\mathbb{F}_2[\Omega_N]\cdot x=I_{\Omega_N}^{w(x)}\cdot \mathbb{F}_2[L_N].
$$
\end{enumerate}
\begin{proof}
To prove $(a)$, notice that if $\rho$ is a maximal cycle, we have that $\mathbb{F}_2[L_N]$ viewed as a $\mathbb{F}_2[\langle \rho \rangle]$-module is free of rank $1$, and isomorphic to $\frac{\mathbb{F}_2[\epsilon]}{\epsilon^{2^N}}$ as a module on the polynomial ring $\mathbb{F}_2[\epsilon]$, where the formal variable $\epsilon$ is sent into $\rho-1$. Since here also the successive submodules are also $2^N+1$, all of index $2$ in the previous one, and since the latter are contained in the former, the only possibility is that they coincide.

Part $(b)$ follows immediately from repeatedly applying Proposition \ref{father-son sequence}, which reduces the verification to $N=0$ and $i=0$, where it is trivial.

Let now $M$ be a non-trivial submodule of $\F_2[L_N]$. Let $x$ be an element of smallest possible weight in $M$, denote this weight by $j$. It follows that $x$ generates $\frac{I_{\Omega_N}^j\cdot \mathbb{F}_2[L_N]}{I_{\Omega_N}^{j+1}\cdot \mathbb{F}_2[L_N]}$. Hence from the non-commutative version of Nakayama's lemma we have that $x$ generates $I_{\Omega_N}^j\cdot \mathbb{F}_2[L_N]$, which is therefore contained in $M$. On the other hand $M$ is trivially contained in $I_{\Omega_N}^j\cdot \mathbb{F}_2[L_N]$, yielding part $(c)$. 

Part $(d)$ immediately follows from the proof of $(c)$, noticing that the module generated by $x$ contains only elements of weight larger equal than $w(x)$. 
\end{proof}
\end{proposition}

We are now ready to derive the following crucial theorem.

\begin{theorem} \label{derived series}
Let $i$ be a non-negative integer. Then
$$\Omega_{\infty}^{(i-\textup{Fr.})}=\Omega_{\infty}^{(2^i-1)}\cong(\Omega_{\infty})_{\textup{fib}}^{2^i}.
$$
\begin{proof}
The second isomorphism is just Corollary \ref{at powers of 2-1}.

To prove the first equality, start by noticing that we can safely assume that $i>0$, as otherwise the claim is a tautology. Let us fix a non-negative integer $N$ and apply Proposition \ref{derandcent in semidir} to the semidirect product
$$\Omega_{N+1}=\mathbb{F}_2[L_N] \rtimes \Omega_N,
$$
to conclude that 
$$\Omega_{N+1}^{(i-\text{Fr.})}=M \rtimes \Omega_N^{(i-\text{Fr.})},
$$
for $M$ an $\mathbb{F}_2[\Omega_N]$ submodule of $\mathbb{F}_2[L_N]$. Therefore in view of Proposition \ref{Uniseriality}, there exists an integer $h$ such that
$$M=I_{\Omega_N}^{h}\cdot \mathbb{F}_2[L_N].
$$
To conclude the proof, in view of Corollary \ref{total explicit description} it is enough to prove that $h=2^i-1$. Since the inclusion $G^{(i-\text{Fr.})} \subseteq G^{(2^i-1)}$ holds for any profinite group, we only need to prove that $h \leq 2^i-1$. To this end, observe that Proposition \ref{derandcent in semidir} immediately implies that $\Omega_{\infty}^{(h-\textup{Fr.})}$ is topologically generated by involutions for each positive integer $h$, since it is an iterated semidirect product of successive $\mathbb{F}_2$-vector spaces. Hence its abelianization must be a profinite $\mathbb{F}_2$-vector space. It follows that if $\rho$ in $\Omega_{\infty}$ is a maximal cycle, then $\rho^{2^h} \in \Omega_{\infty}^{(h-\text{Fr.})}
$ for every $h$ non-negative integer. Hence
$$(\rho^{2^{i-1}}-1)\cdot \ldots \cdot (\rho-1) \in I_{\Omega_N^{(i-1)-\text{Fr.}}}\cdot \ldots \cdot I_{\Omega_N},
$$
and invoking Proposition \ref{derandcent in semidir} we see that
$M \supseteq (\rho-1)^{2^{i-1}+ \ldots +1}\cdot \mathbb{F}_2[L_N]$. But this last module coincides with $I_{\Omega_N}^{2^i-1} \cdot \mathbb{F}_2[L_N]$, and the proof is complete.
\end{proof}
\end{theorem}

\subsection{Proof of Proposition \ref{prop2} and Proposition \ref{prop:generatedbyinvolutions}} \label{sbs:prop 2}
The previous subsection, by means of Corollary \ref{derived series} has successfully proved part $(a)$ of Proposition \ref{prop2} in the special case $\underline{a}=\underline{0}$. So let us fix $\underline{a}$ a non-zero element of $\mathbb{F}_2^{(\mathbb{Z}_{\geq 0})}$ and write
$$N\coloneqq i_{\text{max}}(\underline{a}),
$$
which is then a non-negative integer. Observe that if $N=0$, then Proposition \ref{prop2} part $(a)$ is a tautology, since it identifies a group with itself. On top of that, notice that for $N=0$ we have that 
$$M_{\underline{a}}\cong\Omega_{\infty}^2.
$$
Since the derived series of a square is the square of the derived series, we see that part $(b)$ is a trivial consequences of Corollary \ref{derived series}. Hence, we have established Proposition \ref{prop2} also in case $N=0$. In the rest of this subsection we assume therefore that $N>0$. We denote by
$$(M_{\underline{a}})_{N+1}\subseteq \Omega_{N+1}=\mathbb{F}_2[L_N] \rtimes \Omega_N,
$$
the image of $M_{\underline{a}}$ in $\Omega_{N+1}$. We begin with the following.  
\begin{proposition} \label{index at most 2 derived}
For each non-negative integer $i$, we have that
$$ [\Omega_{N+1}^{(i-\textup{Fr.})}:{(M_{\underline{a}})}_{N+1}^{(i-\textup{Fr.})}]\leq 2.
$$
\begin{proof}
For $i=0$ the statement is a tautology. So let us assume $i>0$. We divide the proof in steps.

$(1)$ Observe that $(M_{\underline{a}})_{N+1}$ projects surjectively onto $\Omega_N$. It follows that ${(M_{\underline{a}})}_{N+1}^{(i-\textup{Fr.})}$ projects surjectively onto $\Omega_N^{(i-\textup{Fr.})}$. As a consequence every element of the quotient 
$$\frac{\Omega_{N+1}^{(i-\textup{Fr.})}}{(M_{\underline{a}})_{N+1}^{(i-\textup{Fr.})}}
$$
has a representative in $\mathbb{F}_2[L_N] \rtimes \{\text{id}\}$.

$(2)$ Next observe that trivially $I_{\Omega_N}\cdot \mathbb{F}_2[L_N] \rtimes \{\text{id}\} \subseteq (M_{\underline{a}})_{N+1}$. Pick for each non-negative integer $j<i$ an element $\rho_j$ of $(M_{\underline{a}})_{N+1}^{(j-\text{Fr.})}$ whose image in $\Omega_N$ coincides with $\rho^{2^j}$, where $\rho$ is a fixed maximal cycle of $\Omega_N$ (as noticed in the proof of Theorem \ref{derived series}, this is an element of $(\Omega_N)^{(j-\textup{Fr.})}$). Observe that therefore the elements of
\begin{equation}\label{iter_augm}
    (\rho_{i-1}-1) \ldots (\rho_0-1) \cdot I_{\Omega_N} \cdot \mathbb{F}_2[L_N] \rtimes \{\text{id}\},
\end{equation}
are naturally elements of $(M_{\underline{a}})_{N+1}^{(i-\text{Fr.})}$. But the above product only depends on the image of the elements in $\Omega_N$, and hence \eqref{iter_augm} is nothing else than
$$(\rho-1)^{2^i-1}\cdot I_{\Omega_N}\cdot \mathbb{F}_2[L_N] \rtimes \{\text{id}\}.
$$
Thanks to Proposition \ref{Uniseriality}, applied back and forth twice, this last module can be rewritten as
$$(\rho-1)^{2^i-1}I_{\Omega_N}\cdot \mathbb{F}_2[L_N]\rtimes \{\text{id}\}=I_{\Omega_N}^{2^i}\cdot \mathbb{F}_2[L_N]\rtimes \{\text{id}\},
$$
which is therefore a subgroup of $(M_{\underline{a}})_{N+1}^{(i-\text{Fr.})}$.

$(3)$ We know that $\Omega_{N+1}^{(i-\textup{Fr.})} \cap (\mathbb{F}_2[L_N] \rtimes \{\text{id}\})=I_{\Omega_N}^{2^i-1}\cdot \mathbb{F}_2[L_N] \rtimes \{\text{id}\}$, thanks to Theorem \ref{derived series} combined with Proposition \ref{derandcent in semidir}.

$(4)$ Thanks to step $(1)$ we can represent any element of $\frac{\Omega_{N+1}^{(i-\textup{Fr.})}}{(M_{\underline{a}})_{N+1}^{(i-\textup{Fr.})}}$ as an element of $\mathbb{F}_2[L_N] \rtimes \{\text{id}\}$. Thanks to step $(2)$, we have that
$$(\mathbb{F}_2[L_N] \rtimes \{\text{id}\}) \cap (M_{\underline{a}})_{N+1}^{(i-\text{Fr.})} \supseteq I_{\Omega_N}^{2^i}\cdot \mathbb{F}_2[L_N] \rtimes \{\text{id}\},
$$
which thanks to step $(3)$ combined with Proposition \ref{Uniseriality} has index $2$ in
$$(\mathbb{F}_2[L_N] \rtimes \{\text{id}\}) \cap (\Omega_{N+1})^{(i-\text{Fr.})} = I_{\Omega_N}^{2^i-1}\cdot \mathbb{F}_2[L_N] \rtimes \{\text{id}\}.
$$
Hence the quotient $\frac{\Omega_{N+1}^{(i-\textup{Fr.})}}{(M_{\underline{a}})_{N+1}^{(i-\textup{Fr.})}}$ admits a surjection from 
$$\frac{I_{\Omega_N}^{2^i-1}\cdot \mathbb{F}_2[L_N]}{I_{\Omega_N}^{2^i}\cdot \mathbb{F}_2[L_N]},
$$
which is of size $2$ by Proposition \ref{Uniseriality}. Therefore the index is at most $2$, as desired. 
\end{proof}
\end{proposition}

We next show the following inclusion. 
\begin{proposition} \label{also at least 2}
Let $j$ be a non-negative integer. Then
$$(M_{\underline{a}})^{(j-\textup{Fr.})} \subseteq (M_{\textup{sh}_j(\underline{a})})_{\textup{fib}}^{2^j}.
$$
\begin{proof}
We proceed by induction on $j$. For $j=0$ the statement is a tautology. Suppose that we have a non-negative integer $j$ such that  
$$(M_{\underline{a}})^{(j-\textup{Fr.})} \subseteq (M_{\textup{sh}_j(\underline{a})})_{\textup{fib}}^{2^j} \subseteq (\Omega_{\infty})_{\textup{fib}}^{2^j}
$$
It follows that 
$$(M_{\underline{a}})^{((j+1)-\textup{Fr.})} \subseteq [(M_{\textup{sh}_j(\underline{a})})_{\textup{fib}}^{2^j},(M_{\textup{sh}_j(\underline{a})})_{\textup{fib}}^{2^j}] \subseteq [(\Omega_{\infty})_{\textup{fib}}^{2^j},(\Omega_{\infty})_{\textup{fib}}^{2^j}]=(\Omega_{\infty})_{\textup{fib}}^{2^{j+1}}.
$$
On the other hand, trivially, 
$$[(M_{\textup{sh}_j(\underline{a})})_{\textup{fib}}^{2^j},(M_{\textup{sh}_j(\underline{a})})_{\textup{fib}}^{2^j}]\subseteq ([M_{\textup{sh}_j(\underline{a})},M_{\textup{sh}_j(\underline{a})}])^{2^{j}}.
$$
It follows from Theorem \ref{Describing abelianized Ma} that 
$$[M_{\textup{sh}_j(\underline{a})},M_{\textup{sh}_j(\underline{a})}] \subseteq (M_{\text{sh}_{j+1}(\underline{a})})_{\text{fib}}^2.
$$
Hence we obtain the desired conclusion, namely that 
$$(M_{\underline{a}})^{((j+1)-\textup{Fr.})} \subseteq ((M_{\text{sh}_{j+1}(\underline{a})})_{\text{fib}}^2)^{2^j} \cap (\Omega_{\infty})_{\text{fib}}^{2^{j+1}}=(M_{\textup{sh}_{j+1}(\underline{a})})_{\textup{fib}}^{2^{j+1}},$$

where equality holds because the first set in the intersection guarantees that each of the $2^{j+1}$ coordinates has to satisfy the relation imposed by $\text{sh}_{j+1}(\underline{a})$, while the second set guarantees that the resulting vector of $2^{j+1}$ elements has agreeing values of $\phi_i$, for every non-negative integer $i$. This is precisely the definition of $(M_{\textup{sh}_{j+1}(\underline{a})})_{\textup{fib}}^{2^{j+1}}$.
\end{proof}
\end{proposition}
As an immediate consequence, we can deduce a level $N+1$ version of Proposition \ref{prop2}.
\begin{proposition} \label{prop2 at level N}
 Let $j\leq N$ be a non-negative integer. Then
$$(M_{\underline{a}})_{N+1}^{(j-\textup{Fr.})}=\pi_{\Omega_{\infty} \to \Omega_{N+1}}((M_{\textup{sh}_j(\underline{a})})_{\textup{fib}}^{2^j}).
$$   
\begin{proof}
Observe that $\text{sh}_j(\underline{a})$ is not zero for each $j \leq N$. This certainly implies that $(M_{\textup{sh}_j(\underline{a})})_{\textup{fib}}^{2^j}$ is an index $2$ subgroup of $(\Omega_{\infty})_{\text{fib}}^{2^j}$. Furthermore we have the same conclusion when projecting at level $N+1$, since the constraint given by belonging to $(M_{\textup{sh}_j(\underline{a})})_{\textup{fib}}^{2^j})$ concerns still only the first $N+1$ entries of the digital representation, when we view $(\Omega_{\infty})_{\text{fib}}^{2^j}$ as a subgroup of $\Omega_{\infty}$ (more precisely it constraints only digits between level $j$ and level $N+1$). In other words we have that $\pi_{\Omega_{\infty} \to \Omega_{N+1}}((M_{\textup{sh}_j(\underline{a})})_{\textup{fib}}^{2^j}))$ is an index $2$ subgroup of $\pi_{\Omega_{\infty} \to \Omega_{N+1}}((\Omega_{\infty})_{\text{fib}}^{2^j})$. Hence combining Theorem \ref{derived series}, Proposition \ref{index at most 2 derived}, and Proposition \ref{also at least 2} we obtain immediately the desired conclusion.      
\end{proof}
\end{proposition}
We now gain control on what happens after level $N+1$, for $j \leq N$. Let us first recall the following basic fact. 
\begin{proposition} \label{augnormal gives submodule}
Let $G$ be a group, $U$ be a normal subgroup, and $M$ be a $\mathbb{F}_2[G]$-module. Then $I_U\cdot M$ is a $\mathbb{F}_2[G]$-sub-module. 
\begin{proof}
Indeed we have that
$$g(u-1)m=gug^{-1}gm-gm=(gug^{-1}-1)gm,
$$
where the last term is still in $I_U \cdot M$, owing to the fact that $U$ is normal. 
\end{proof}
\end{proposition}
\begin{proposition} \label{control at infinite level}
 Let $j\leq N$ be a non-negative integer. Then
$$(M_{\underline{a}})^{(j-\textup{Fr.})} \cap \textup{ker}(\Omega_{\infty} \to \Omega_{N+1})=\Omega_{\infty}^{(j-\textup{Fr.})} \cap \textup{ker}(\Omega_{\infty} \to \Omega_{N+1}).
$$
\begin{proof}
Let $N_1$ be an integer strictly larger than $N+1$. Write $N_2\coloneqq N_1-1$ and consider
$$\Omega_{N_1}\coloneqq\mathbb{F}_2[L_{N_2}] \rtimes \Omega_{N_2}.
$$
For a non-negative integer $h$, write $(M_{\underline{a}})_{h}$ for the image of $M_{\underline{a}}$ in $\Omega_{h}$. Since $N_1>N+1$ we have that
$$(M_{\underline{a}})_{N_1}=\mathbb{F}_2[L_{N_2}] \rtimes (M_{\underline{a}})_{N_2}.
$$
Keeping in mind that $M_{\underline{a}}$ is normal in $\Omega_{\infty}$, we obtain by Proposition \ref{derandcent in semidir} combined with Proposition \ref{augnormal gives submodule} that
$$(M_{\underline{a}})_{N_1}^{(j-\text{Fr.})}=M \rtimes ((M_{\underline{a}})_{N_2})^{(j-\text{Fr.})},
$$
where $M$ is an $\mathbb{F}_2[\Omega_{N_2}]$-submodule of $\mathbb{F}_2[L_{N_2}]$. Therefore by Proposition \ref{Uniseriality} we conclude that there must be a non-negative integer $k_0$ such that
$$M=I_{\Omega_{N_2}}^{k_0} \cdot \mathbb{F}_2[L_{N_2}].
$$
Notice that if we can prove that $k_0=2^j-1$, then the proof is concluded thanks to Theorem \ref{derived series} and the fact that $N_2$ is an arbitrary integer $\ge N+1$.

Obviously we have that
$$(M_{\underline{a}})_{N_1}^{(j-\text{Fr.})} \subseteq (\Omega_{N_1})^{(j-\text{Fr.})}. 
$$
Therefore, in virtue of Proposition \ref{derived series}, we obtain that $k_0 \geq 2^{j}-1$. Hence, thanks to Proposition \ref{Uniseriality}, in case we are able to produce an element of $M$ of weight $2^j-1$, then we conclude that $k_0=2^j-1$. Pick any element $v$ of $\mathbb{F}_2[L_{N_2}]$ of weight $0$. Observe that $M_{\underline{a}}$ surjects onto $\Omega_{j}$. Therefore for each non-negative integer $h$ we have that $M_{\underline{a}}^{(h-\text{Fr.})}$ surjects onto $\Omega_{j}^{(h-\text{Fr.})}$. In particular we can produce 
$$\rho(h) \in M_{\underline{a}}^{(h-\text{Fr.})},
$$
which maps to $\rho^{2^h}\in \Omega_{j}^{(h-\text{Fr.})}$, where $\rho$ is a maximal cycle of $\Omega_{j-1}$. Thanks to Proposition \ref{derandcent in semidir} we have that
$$v'\coloneqq (\rho(j-1)-1) \ldots (\rho(0)-1)\cdot v\in M.
$$
We claim that $w(v')=2^j-1$. Thanks to Proposition \ref{father-son sequence}, it suffices to show that the image of $v'$ in $\mathbb{F}_2[L_j]$ under the chain of maps $f_{j+1} \circ \ldots \circ f_{N_2}$ has the same weight $2^j-1$. On the other hand
$$f_{j+1} \circ \ldots \circ f_{N_2} (v')=(\rho-1)^{2^{j-1}}\cdot \ldots \cdot (\rho-1) \cdot v=(\rho-1)^{2^j-1} \cdot v
$$
and the right hand side has weight exactly equal to $2^j-1$, in virtue of Proposition \ref{Uniseriality}. Hence we have that $k_0=2^j-1$.
\end{proof}
\end{proposition}

We are now ready for the proof of Proposition \ref{prop2} part $(a)$.

\textbf{Proof of Proposition \ref{prop2} part $(a)$:} 
\begin{proof} The combination of Proposition \ref{prop2 at level N} and Proposition \ref{control at infinite level} gives us that the two sides of the isomorphism of Proposition \ref{prop2}, viewed as closed subgroups of $\Omega_{\infty}$, coincide at every finite level. Being closed, this implies that they coincide, as desired. 
\end{proof}
We now focus on part $(b)$. As we will see, the following special case of Proposition \ref{prop2} part $(b)$ contains all the difficulty. 
\begin{proposition} \label{the case N+1}
We have that  
$$(M_{\underline{a}})^{((N+1)-\textup{Fr.})}=((\Omega_{\infty})_{\textup{fib}}^2)^{2^{N+1}}.   
$$
\begin{proof}
Thanks to Proposition \ref{prop2} part $(a)$, which we have established above, we know that  
$$\mathcal{G}\coloneqq(M_{\underline{a}})^{(N-\textup{Fr.})}=(M_{(1,0,\ldots,0,\ldots)})^{2^N}_{\textup{fib}}=(\Omega_{\infty})_{\textup{fib}}^{2^{N}} \cap \text{ker}(\Omega_{\infty} \to \Omega_{N+1}).  
$$
By Theorem \ref{derived series} and Corollary \ref{reducing descending central to augmentational}, in terms of iterated semi-direct products the latter is the subgroup of elements $(\sigma_h)_{h \geq 0}$, where $\sigma_h=0$ for every $h \leq N$, and $\sigma_h \in I_{\Omega_h}^{2^N-1}\cdot \mathbb{F}_2[L_h]$, for every $h>N$. Therefore we can compute its commutator subgroup by means of Proposition \ref{derandcent in semidir}. To ease this computation, for each integer $s \geq N+2$, let us define $\mathcal{G}_s$ to be the image of $\mathcal{G}$ in $\Omega_s$. Then we have that
$${\mathcal{G}}_{s+1}=I_{\Omega_s}^{2^{N}-1} \cdot \mathbb{F}_2[L_s] \rtimes \mathcal{G}_s,
$$
and we can reconstruct $\mathcal{G}$ by taking the inverse limit along $s$ of all these semi-direct products. Let us now focus on computing the commutator subgroup of each $\mathcal{G}_{s+1}$. By Proposition \ref{derandcent in semidir} we have
$$[{\mathcal{G}}_{s+1},{\mathcal{G}}_{s+1}]=I_{\mathcal{G}_s} \cdot I_{\Omega_s}^{2^{N}-1} \cdot \mathbb{F}_2[L_s] \rtimes [\mathcal{G}_s,\mathcal{G}_s].
$$
Now given the fact that $\mathcal G_s$ is trivial whenever $s\leq N$, and in view of Theorem \ref{derived series} and Corollary \ref{reducing descending central to augmentational}, in order to prove the claim it is enough to show that
\begin{equation}\label{shape of augmentation}
    I_{\mathcal{G}_s} \cdot I_{\Omega_s}^{2^{N}-1} \cdot \mathbb{F}_2[L_s]= I_{\Omega_s}^{2^{N+1}}\mathbb{F}_2[L_s].
\end{equation}
Let us start by showing that
\begin{equation}\label{inclusion1}
    I_{\mathcal{G}_s} \cdot \mathbb{F}_2[L_s] \subseteq I_{\Omega_s}^{2^{N+1}}\mathbb{F}_2[L_s],
\end{equation}
which certainly implies that 
$$I_{\mathcal{G}_s} \cdot I_{\Omega_s}^{2^{N}-1} \cdot \mathbb{F}_2[L_s] \subseteq I_{\Omega_s}^{2^{N+1}}\mathbb{F}_2[L_s].
$$
(Notice that \eqref{shape of augmentation} and \eqref{inclusion1} do not contradict each other, since the order of multiplication matters, and one cannot swap the ideals and get the bigger power $2^{N+1}+2^N-1$ from $I_{\mathcal{G}_s} \cdot \mathbb{F}_2[L_s] \subseteq I_{\Omega_s}^{2^{N+1}}\mathbb{F}_2[L_s]$).

Let us now prove \eqref{inclusion1}. Thanks to Proposition \ref{father-son sequence}, it suffices to show the claim for $s=N+2$. The same proposition shows that $\mathbb{F}_2[L_{N+2}]$ is an extension of $\mathbb{F}_2[L_{N+1}]$ with itself as $\mathbb{F}_2[\Omega_{N+1}]$-module, through the maps $s_{N+2}$ and $f_{N+2}$. But by assumption $\mathcal{G}$ acts trivially on $L_{N+1}$. It follows immediately that
$$I_{\mathcal{G}_{N+2}} \cdot \mathbb{F}_2[L_{N+2}] \subseteq s_{N+2}(\F_2[L_{N+1}])= I_{\Omega_{N+2}}^{2^{N+1}}\mathbb{F}_2[L_{N+2}],
$$
where the last equality is a third application of Proposition \ref{father-son sequence}. This settles \eqref{inclusion1}.

Now let
$$M\coloneqq I_{\mathcal{G}_s} \cdot I_{\Omega_s}^{2^{N}-1} \cdot \mathbb{F}_2[L_s].
$$
To conclude the proof, we need to show that 
\begin{equation}\label{inclusion2}
    I_{\Omega_s}^{2^{N+1}}\F_2[L_s]\subseteq M.
\end{equation}
Since $\mathcal{G}$ is certainly normal in $\Omega_{\infty}$, it follows from Proposition \ref{augnormal gives submodule} that $M$ is a $\mathbb{F}_2[\Omega_s]$-sub-module, and  it follows from Proposition \ref{Uniseriality} that there exists a non-negative integer $k_0$ such that 
$$M=I_{\Omega_s}^{k_0}\cdot \mathbb{F}_2[L_s].
$$
We have already shown that $k_0 \geq 2^{N+1}$. We will now show that this is an exact equality. To this end it suffices to exhibit precisely one element of weight $2^{N+1}$ in $M$. Thanks to Proposition \ref{father-son sequence} it suffices to produce $m_0$ in $I_{\Omega_s}^{2^{N}-1} \cdot \mathbb{F}_2[L_{N+2}]$ and $g$ in $\mathcal{G}$ such that
$$w((g-1)m_0)=2^{N+1}. 
$$
Indeed Proposition \ref{father-son sequence} implies that if we pick $\widetilde{m}_0$ in $\mathbb{F}_2[L_s]$ that projects on $m_0$ through the natural surjection coming from iterating $f_h$'s, we have that also $w((g-1)\widetilde{m}_0)=2^{N+1}$. 
For that purpose consider in
$$I_{\Omega_{N+2}}^{2^{N}-1} \cdot \mathbb{F}_2[L_{N+2}] \rtimes I_{\Omega_{N+1}}^{2^{N}-1} \cdot \mathbb{F}_2[L_{N+1}]
$$
the element $g\coloneqq(0,x_2)$, where $x_2$ equals $s_{N+1}(v)$ for a node $v$ in $L_{N}$, and $m_0=(x_1,0)$, where $x_1$ consists of any sum of $2^N$ elements in $\mathbb{F}_2[L_{N+2}]$ in such a way that every node in $L_N$ has precisely one descendant in the sum. Observe that $x_2$ is in the image of the map $s_{N+1}$ and hence even in $I_{\Omega_{N+2}}^{2^{N}} \cdot \mathbb{F}_2[L_{N+1}] $ thanks to Proposition \ref{father-son sequence}. On the other hand, by design, the element $x_1$ maps to the sum of all elements of $L_N$ through the map $f_{N+1}\circ f_{N+2}$: that is precisely the unique non-trivial element of weight $2^N-1$ in $\mathbb{F}_2[L_N]$. Therefore Proposition \ref{father-son sequence} implies that $w(x_1)=2^N-1$. Now we have that
$$(g-1)m_0=s_{N+2}(v')
$$
where $v'$ is the unique node appearing in $x_2$ that is above a node appearing in $x_1$. Since $(g-1)m_0$ is the value of the map $s_{N+2}$ applied to an element of $L_{N+2}$ (which is an element of weight $0$), it follows from Proposition \ref{father-son sequence} that the image has weight exactly equal to $2^{N+1}$, as desired. This settles \eqref{inclusion2}, and in turn \eqref{shape of augmentation}.
\end{proof}
\end{proposition}

We are now ready to give a proof of Proposition \ref{prop2} part $(b)$.

\textbf{Proof of Proposition \ref{prop2} part $(b)$:} 
\begin{proof} From Proposition \ref{the case N+1} we know that 
$$(M_{\underline{a}})^{((N+1)-\textup{Fr.})}=((\Omega_{\infty})_{\textup{fib}}^2)^{2^{N+1}}.
$$
It follows that for any $s$ non-negative integer
$$(M_{\underline{a}})^{((N+1+s)-\textup{Fr.})}=((M_{\underline{a}})^{((N+1)-\textup{Fr.})})^{(s-\text{Fr.})}=(((\Omega_{\infty})_{\textup{fib}}^2)^{2^{N+1}})^{(s-\text{Fr.})}=$$
$$=(((\Omega_{\infty})_{\textup{fib}}^2)^{(s-\text{Fr.})})^{2^{N+1}}.
$$
On the other hand, Theorem \ref{derived series} for $i=1$ implies that 
$$(\Omega_{\infty})_{\textup{fib}}^2=\Omega_{\infty}^{(1-\text{Fr.})}.
$$
Therefore
$$((\Omega_{\infty})_{\textup{fib}}^2)^{(s-\text{Fr.})}=\Omega_{\infty}^{((s+1)-\text{Fr.})}.
$$
Applying once more Theorem \ref{derived series}, we conclude that the right hand side equals
$$(\Omega_{\infty})_{\textup{fib}}^{2^{s+1}}.
$$
Hence in total we get
$$((\Omega_{\infty})_{\textup{fib}}^{2^{s+1}})^{2^{N+1}}.
$$
Recalling that $N=i_{\text{max}}(\underline{a})$ and that $s+1=N+1+s-i_{\text{max}}(\underline{a})$, we have obtained precisely the desired statement.
\end{proof}

We conclude with the proof of Proposition \ref{prop:generatedbyinvolutions}.

\textbf{Proof of Proposition \ref{prop:generatedbyinvolutions}:}

\begin{proof}
First, suppose that $a_0=0$. By choosing sequences $(x_N)_{N \geq 1}$ of elements with $x_N \in L_N$ for $N \geq 1$, with the property that $x_i$ does not belong to the subtree $T_{\infty}x_j$ for $j<i$, we find that any vector $\underline{v}$ of $\mathbb{F}_2^{\mathbb{Z}_{\geq 0}}=\Omega_\infty^{\ab}$ with $v_0=0$ can be lifted to an involution of $\Omega_{\infty}$. From here we conclude immediately that if $a_0=0$, then there exists a closed subgroup
$$\Gamma \leq (M_{\underline{a}})_{\textup{fib}}^{h}
$$
topologically generated by involutions, such that any element of  $(M_{\underline{a}})_{\textup{fib}}^{h}$ can be represented as $\gamma \cdot y$ with $y$ in $([\Omega_{\infty},\Omega_{\infty}])^{h}$. Hence it suffices to show that $[\Omega_{\infty},\Omega_{\infty}]$ is generated by involutions. It follows immediately from Proposition \ref{derandcent in semidir} that $[\Omega_{\infty},\Omega_{\infty}]$ is an iterated semi-direct product, where at each step $N$ we have the $\mathbb{F}_2$-vector space $I_{\Omega_N}\cdot \mathbb{F}_2[L_N]$. Therefore $[\Omega_{\infty},\Omega_{\infty}]$ is certainly topologically generated by involutions.

Next, if $a_0=1$ and $i_{\max}(\underline{a})=0$ then $M_{\underline{a}}\cong\Omega_{\infty}^2$, that is topologically generated by involutions by what we said above.

Finally, Theorem \ref{Describing abelianized Ma} shows that whenever $a_0=1$ and $i_{\max}(\underline{a})>0$, the group $M_{\underline{a}}$ has a continuous epimorphism onto $\mathbb{Z}/4\mathbb{Z}$. Since any of the $h$ coordinate projections from $(M_{\underline{a}})_{\text{fib}}^h$ is a continuous epimorphism onto $M_{\underline{a}}$, it follows that also $(M_{\underline{a}})_{\text{fib}}^h$ surjects onto $\mathbb{Z}/4\mathbb{Z}$. Certainly a profinite group that has a continuous epimorphism onto $\mathbb{Z}/4\mathbb{Z}$ cannot be topologically generated by involutions.
\end{proof}

\subsection{End of proof of Theorem \ref{thmC}}
\begin{proof} In Section \ref{sbs:prop 2} we have established Proposition \ref{prop2} and Proposition \ref{prop:generatedbyinvolutions}. In Section \ref{sbs:from 2 to 1} we have proved that Proposition \ref{prop2} joint with Proposition \ref{prop:generatedbyinvolutions} imply \ref{prop1}. Hence Proposition \ref{prop1} holds. Finally, in Section \ref{sbs: from 1 to C} we have proved that Proposition \ref{prop1} implies Theorem \ref{thmC}. Therefore Theorem \ref{thmC} also holds, concluding the proof. 
\end{proof}

\bibliographystyle{plain}
\bibliography{bibliography}

\end{document}